\documentclass[11pt,letterpaper]{amsart}

\usepackage{amsthm}
\usepackage{amsmath}
\usepackage{amssymb}
\usepackage[mathscr]{euscript}
\usepackage{dsfont}
\usepackage{mathrsfs}
\usepackage{hyperref}

\usepackage{todonotes}

\usepackage{tikz-cd}

\usepackage[all]{xy}

\usepackage{latexsym}
\usepackage{palatino}

\usepackage{enumitem}

\title[Perverse sheaves on symmetric products of the plane]{Perverse sheaves on symmetric products of the plane}

\author{Tom Braden}
\author{Carl Mautner}

\address{Tom Braden \\
Dept.\ of Mathematics and Statistics\\
         University of Massachusetts, Amherst}
\email{braden@math.umass.edu}

\address{Carl Mautner \\
Department of Mathematics\\
University of California, Riverside}
\email{carl.mautner@ucr.edu}

\subjclass[2020]{14F08, 14C05, 20C30}

\DeclareFontFamily{U}{mathx}{\hyphenchar\font45}
\DeclareFontShape{U}{mathx}{m}{n}{
      <5> <6> <7> <8> <9> <10>
      <10.95> <12> <14.4> <17.28> <20.74> <24.88>
      mathx10
      }{}
\DeclareSymbolFont{mathx}{U}{mathx}{m}{n}
\DeclareFontSubstitution{U}{mathx}{m}{n}
\DeclareMathAccent{\widecheck}{\mathord}{mathx}{"71}

\newtheorem{theorem}{Theorem}[section]
\newtheorem*{theorem*}{Theorem}
\newtheorem{lemma}[theorem]{Lemma}
\newtheorem{proposition}[theorem]{Proposition}
\newtheorem{corollary}[theorem]{Corollary}
\newtheorem{claim}[theorem]{Claim}

{%\theorembodyfont{\normalfont}
\theoremstyle{definition}
\newtheorem{definition}[theorem]{Definition}
\newtheorem{example}[theorem]{Example}

\newtheorem{remark}[theorem]{Remark}
\newtheorem*{remark*}{Remark}
}

\newcommand{\excise}[1]{}

\newcommand{\rank}{\operatorname{rk}}

\newcommand{\id}{\operatorname{id}}
\newcommand{\Coker}{\operatorname{Coker}}
\newcommand{\Ker}{\operatorname{Ker}}
\renewcommand{\dim}{\operatorname{dim}}
\newcommand{\codim}{\operatorname{codim}}

\newcommand{\Hom}{\operatorname{Hom}}
\newcommand{\End}{\operatorname{End}}
\newcommand{\RHom}{{R\mathcal{H}\mathit{\!om}}}
\newcommand{\Ext}{\operatorname{Ext}}

\newcommand{\C}{{\mathbb{C}}}
\newcommand{\Z}{{\mathbb{Z}}}

\newcommand{\R}{{\mathbb{R}}}

\renewcommand{\iff}{\Leftrightarrow}

\newcommand{\D}{\mathbb D}

\newcommand{\IC}{\mathop{\mathbf{IC}}}

\newcommand{\cal}{\mathcal}
\newcommand{\cA}{{\cal A}}
\newcommand{\cB}{{\cal B}}
\newcommand{\cC}{{\cal C}}

\newcommand{\cF}{{\cal F}}
\newcommand{\cG}{{\cal G}}
\newcommand{\cH}{{\cal H}}

\newcommand{\cL}{{\cal L}}

\newcommand{\cN}{{\cal N}}
\newcommand{\cO}{{\cal O}}

\newcommand{\cT}{{\cal{T}}}

\newcommand{\cS}{{\mathscr S}}
\newcommand{\cU}{{\cal U}}
\newcommand{\cV}{{\cal V}}
\newcommand{\cW}{{\cal W}}

\newcommand{\bI}{{\mathbf I}}
\newcommand{\bone}{{\mathbf 1}}

\newcommand{\sS}{\mathscr S}
\newcommand{\fS}{\mathfrak S}

\newcommand{\wt}{\widetilde}

\newcommand{\md}{\mathrm{-mod}}

\renewcommand{\setminus}{\smallsetminus}

\newcommand{\kk}{{\Bbbk}}
\newcommand{\uk}{\underline{\kk}} % constant sheaf

\renewcommand{\emptyset}{\varnothing}
\DeclareMathOperator{\Perv}{Perv}

\DeclareMathOperator{\Vect}{Vect}
\renewcommand{\Vec}{\operatorname{\mathsf{Vec}}}

\DeclareMathOperator{\Sh}{Sh}
\DeclareMathOperator{\Hilb}{Hilb}
\DeclareMathOperator{\Fun}{\mathsf{Fun}}
\DeclareMathOperator{\Mod}{\mathsf{Mod}}
\DeclareMathOperator{\Bimod}{\mathsf{Bimod}}
\DeclareMathOperator{\GL}{GL}

\DeclareMathOperator{\Conf}{Conf}

\newcommand{\mmod}{\operatorname{-mod}}

\newcommand{\Ind}{\mathop{\mathrm{Ind}}\nolimits}

\newcommand{\bs}{\backslash}

\newcommand{\Pol}{\operatorname{\mathsf{Pol}}}
\newcommand{\Emb}{\operatorname{\mathsf{Emb}}}

\newcommand*{\Part}[1]{\operatorname{\mathsf{Part}}(#1)}

\AtBeginDocument{%
	\def\MR#1{}
}

\begin{document}

\setlist{itemsep=.8ex}

\maketitle

\begin{abstract}
For any field $\kk$, we give an algebraic description of the category $\Perv_\cS(S^n (\C^2),\kk)$ of perverse sheaves on the $n$-fold symmetric product of the plane $S^n(\C^2)$ constructible with respect to its natural stratification and with coefficients in $\kk$.  In particular, we show that it is equivalent to the category of modules over a new algebra that is closely related to the Schur algebra.
As part of our description we obtain an analogue of modular Springer theory for the Hilbert scheme $\Hilb^n(\C^2)$ of $n$ points in the plane with its Hilbert-Chow morphism.
\end{abstract}

%%%%%%%%%%%%%%%%%%%
\section{Introduction}
%%%%%%%%%%%%%%%%%%%

Let $\kk$ be a field of arbitrary characteristic and $n,d\geq 0$ be integers.  The Schur algebra $S_\kk(n,d)$ can be defined as the finite dimensional algebra 
$$S_\kk(n,d):= \End_{\fS_d} ((\kk^n)^{\otimes d}),$$ where the symmetric group $\fS_d$ acts on $(\kk^n)^{\otimes d}$ by permuting the tensor factors.

Let $\cN_{\GL_n} \subset \mathfrak{gl}_n(\C)$ denote the cone of nilpotent matrices and $\pi: \widetilde{\cN} \to \cN_{\GL_n}$ denote the Springer resolution. Recall that $\cN_{\GL_n}$ is a symplectic singularity and $\pi$ is a symplectic resolution.  Let $\Perv_{\GL_n}(\cN_{\GL_n},\kk)$ denote the abelian category of $\GL_n$-equivariant perverse sheaves with coefficients in $\kk$ on $\cN_{\GL_n}$.

The results of this paper are motivated by the following theorem:
\begin{theorem*}[\cite{Mautner}]
There is an equivalence of categories
\[\Perv_{\GL_n}(\cN_{\GL_n},\kk) \simeq S_\kk(n,n) \md.\]
\end{theorem*}

Our main result is an analogous theorem for another conical symplectic singularity, the $n$-fold symmetric product of the plane $S^n(\C^2) = (\C^2)^n/\fS_n$.

For a set $A$, let $\Sh(A)$ denote the category of sheaves of complex vector spaces on $A$ (equivalently, $A$-graded complex vector spaces $\Vec_A$).  Consider the category of sheaves on the symmetric group $\Sh(\fS_d)$ as a tensor category with tensor structure coming from convolution of sheaves on the group $\fS_d$.

\begin{definition}
Let $[n]=\{1,\ldots,n\}$. We define the \textit{Hilbert-Schur algebra} $HS_\kk(n,d)$ to be the finite dimensional algebra: 
\[ HS_\kk(n,d) := K[\End_{\Sh(\fS_d)} (\Sh([n]^d))] \otimes_\Z \kk. \]
Here $\Sh([n]^d)$ is viewed as a module category over the tensor category $\Sh(\fS_d)$, we write $\End_{\Sh(\fS_d)} (\Sh([n]^d))$ to denote the tensor category of (additive) endofunctors of $\Sh([n]^d)$ as a $\Sh(\fS_d)$-module category, and the symbol $K$ denotes the Grothendieck ring.   

The reader who is not familiar with the theory of tensor categories can look at an alternative  description of this algebra given in Theorem \ref{cor:main} via $K$-theory of equivariant sheaves on finite sets.  We also give a presentation by generators and relations later in Section \ref{sec:the ring R}.

%In the notation of~\cite{EGNO}, $\End_{\Sh(\fS_d)}(\Sh([n]^d)) = (\Vec_{\fS_d})^*_{\Vec_{[n]^d}}$ is the dual tensor category to $\Vec_{\fS_d}$ with respect to $\Vec_{[n]^d}$.\footnote{We encourage readers who are uncomfortable with the language of tensor categories to read on as we will give an alternative generators and relations approach in Section~\ref{sec:2-category}.}
\end{definition}

Let $\cS$ denote the natural stratification of $S^n(\C^2)$ and let $\Perv_\cS(S^n (\C^2),\kk)$ denote the abelian category of $\cS$-constructible perverse sheaves with coefficients in $\kk$ on $S^n(\C^2)$.

Our main result is:  
\begin{theorem}\label{thm-equiv}
There is an equivalence of categories
\[\Perv_\cS(S^n (\C^2),\kk) \simeq HS_\kk(n,n) \md.\]
\end{theorem}

Before describing the proof, we comment briefly on some basic properties of the Hilbert-Schur algebra and its relation to the Schur algebra.
\begin{enumerate}
\item The Schur algebra $S_\kk(n,d)$ is a quotient of the Hilbert-Schur algebra $HS_\kk(n,d)$ --- the quotient homomorphism is induced by sending a $\Sh(\fS_d)$-module endofunctor of $\Sh([n]^d)$ to the corresponding $\kk[\fS_d]$-module endomorphism of $K(\Sh([n]^d)) \otimes \kk \cong (\kk^n)^{\otimes d}$.  

\item Unlike $S_\kk(n,d)$, the algebra $HS_\kk(n,n)$ is not quasihereditary, but it is self-injective.

\item Recall that the simple $S_\kk(n,n)$-modules are parametrized by partitions of $n$.  As a corollary of the theorem, the simple $HS_\kk(n,n)$-modules are parametrized by pairs $(\lambda,V)$, where $\lambda$ is a partition of $n$ and $V$ is an irreducible $\kk[N_{\fS_n}(\fS_\lambda)/\fS_\lambda]$-module.  In particular, the number of simple objects depends on $\kk$.
\end{enumerate}

We remark that Theorem~\ref{thm-equiv} fits into a broader program of the authors to study modular representation theory arising from perverse sheaves on conical symplectic singularities.  Previously, motivated by~\cite{Mautner,AM} and the work of Braden--Licata--Proudfoot--Webster on hypertoric category $\cO$ and \textit{symplectic duality} \cite{BLPWtorico,BLPWgco}, we carried out this program in the case of affine hypertoric varieties.  In~\cite{BMHyperRingel,BMMatroid}, we proved that an appropriate category of perverse sheaves on an affine hypertoric variety $\mathfrak M$ is equivalent to modules over a quasihereditary algebra $R(\mathfrak M)$ depending only on the matroid underlying the variety $\mathfrak M$.  Moreover, the algebra $R(\mathfrak M)$ is Ringel dual to the algebra $R(\mathfrak M^\vee)$ associated to the symplectic dual hypertoric variety or dual matroid.

\subsection{Structure of proof} To prove our theorem we construct a projective generator of the category of perverse sheaves $\Perv_\cS(S^n (\C^2),\kk)$ and then compute the endomorphism algebra of this projective generator.

\subsubsection{Projective generators}
Our generator is defined as a direct sum of projective objects, one for each composition $\lambda = (\lambda_1, \ldots, \lambda_r)$ of $n$.

Recall that the Hilbert--Chow morphism  $$\rho_n:\Hilb^n(\C^2) \to S^n(\C^2)$$ from the Hilbert scheme $\Hilb^n(\C^2)$ of $n$ points in the plane is a symplectic resolution of singularities of $S^n(\C^2)$.

\begin{definition}
The \textit{Hilbert-Chow sheaf} $\cF_n$ is the self-dual perverse push-forward sheaf 
$$\cF_n := \rho_{n*} \underline{\kk}_{\Hilb^n(\C^2)}[2n] \in \Perv_\cS(S^n (\C^2),\kk).$$
\end{definition}

The following is a special case of Theorem~\ref{thm-representF_lambda}:

\begin{theorem}
The Hilbert-Chow sheaf $\cF_n$ represents an exact hyperbolic restriction functor.  As a consequence, $\cF_n$ is a projective-injective object in $\Perv_\cS(S^n (\C^2),\kk)$.
\end{theorem}

\begin{remark}
This result is analogous to and motivated by the fact~\cite[Proposition 7.10]{AHR} that the the Springer sheaf $$\mathcal S:= \pi_* \uk_{\widetilde{\cN}}[\dim \widetilde{\cN}] \in \Perv_{\GL_n}(\cN_{\GL_n},\kk)$$ is also a projective (and injective) object represented by an exact hyperbolic restriction functor (parabolic restriction).
\end{remark}

On the other hand, there is another projective-injective object of $\Perv_\cS(S^n (\C^2),\kk)$ that is even easier to construct.  Let $\sigma_n: (\C^2)^n \to S^n(\C^2)$ be the quotient or symmetrization map.  It is a finite-to-one morphism that is the universal cover over the generic locus.  Then $\cF_{(1^n)}:= \sigma_{n*} \underline{\kk}_{(\C^2)^n}[2n]$ is projective.\footnote{While we prove a more general statement below, here is quick way to see that $\cF_{(1^n)}$ is a projective perverse sheaf.  Note that $\cF_{(1^n)}$ is self-dual and its restriction to the singular locus of $S^n (\C^2)$ is contained in $D^{\leq -2n} \subset {}^p D^{<-1}$.  Thus $\cF_{(1^n)}$ is isomorphic to the perverse !- (and *-) extension of its restriction to the smooth locus, which is the local system whose monodromy is the regular representation of the fundamental group.  By the adjunction $({}^p j_!,p^*)$ and exactness of $j^*$, it follows that $\cF_{(1^n)}$ is also projective (and by duality, injective).}

More generally, there are perverse sheaves that interpolate between these two examples.  For any composition $\lambda = (\lambda_1, \ldots, \lambda_r)$ of $n$, let 
$$\sigma_\lambda: S^{\lambda_1}(\C^2) \times \ldots \times S^{\lambda_r}(\C^2) \to S^n(\C^2)$$
be the symmetrization (or sum) map and let $\cF_\lambda := \sigma_{\lambda*}(\cF_{\lambda_1} \boxtimes \ldots \boxtimes \cF_{\lambda_r} )$.

In Theorem~\ref{thm-representF_lambda} and Proposition~\ref{prop-projgen} we prove:

\begin{theorem}
For any composition $\lambda = (\lambda_1, \ldots, \lambda_r)$ the perverse sheaf $\cF_\lambda$ is projective and injective and the set of $\cF_\lambda$, where $\lambda$ runs over all partitions is a generating set of projective objects, so that $\cF = \bigoplus_\lambda \cF_\lambda$ is a projective generator.
\end{theorem}

\subsubsection{Endomorphisms}  The previous theorem implies that $\Perv_\cS(S^n (\C^2),\kk)$
is equivalent to the category of finite-dimensional modules over $\End(\cF)^{\mathrm{op}}$.  To complete our proof of the main result it remains to give an explicit description of the spaces of homomorphisms $\Hom(\cF_\lambda,\cF_\mu)$ and their composition maps.

Our proof builds off of the special case of the Hilbert-Chow sheaf:

 \begin{theorem}\label{thm-spr}
The endomorphism algebra of the Hilbert-Chow sheaf $\cF_n := \rho_{n*} \underline{\kk}_{\Hilb^n(\C^2)}[2n]$ is isomorphic to the tensor product of $\kk$ with the character ring $R_\C(\fS_n) = K(\mathrm{Rep} (\C[\fS_n] ))$ of $\fS_n$.  In other words,  $$\End(\cF_n) \cong R_\C(\fS_n)  \otimes_\Z \kk.$$
In particular, there is a natural action of the character ring $R_\C(\fS_n)$ on the cohomology of the Hilbert-Chow fibers.
\end{theorem}

We note that this is an analogue of the following central result of Springer theory, which first appeared in work of Borho-MacPherson in the case $\kk=\C$ and was extended by Juteau to general $\kk$:

\begin{theorem*}[\cite{BoMcP81,BoMcP,Springer,Juteau}] If $G$ is a reductive algebraic group over $\C$ and $\pi\colon \wt\cN \to \cN$ is the Springer resolution of the nilpotent cone of $\operatorname{Lie}(G)$, then the endomorphism algebra of the Springer sheaf $\mathcal S:= \pi_* \uk_{\widetilde{\cN}}[\dim \widetilde{\cN}]$ is isomorphic to the group algebra of the Weyl group $W$ of $G$.  In other words,
$$\End(\mathcal S) \cong \Z[W] \otimes_\Z \kk.$$
\end{theorem*}

\begin{remark}
The origin of Springer's correspondence was work of J. A. Green, from which it follows that for $GL_n$, the numbers of irreducible components of the Springer fibers are equal to the dimensions of the irreducible representations of the symmetric group.

In contrast, by a theorem of Brian\c{c}on~\cite{Bri77}, the fibers of the Hilbert-Chow morphism are all irreducible, and so one might not expect an interesting algebra to arise from the Hilbert-Chow map.  On one level this is correct. The character ring $R_\C(\fS_n)$ is commutative and so when the characteristic of $\kk$ is zero, the algebra $R_\C(\fS_n)  \otimes_\Z \kk$ is isomorphic to the unremarkable algebra $\kk^{p(n)}$, where $p(n)$ is the number of partitions of $n$.
In fact, a proof of Theorem~\ref{thm-spr} in the case when $\kk$ has characteristic zero is considerably easier and has appeared in the literature already --- see ~\cite[Lemma 7.7]{GS}, for example.

Nonetheless, $R_\C(\fS_n)$ has interesting integral structure and when $\kk$ is a field of small positive characteristic $R_\C(\fS_n) \otimes_\Z \kk$ has nontrivial representation theory, which has been studied by Bonnaf\'e~\cite{Bon06, Bon08}.
\end{remark}

For a composition $\lambda = (\lambda_1, \ldots, \lambda_r)$, let $\fS_\lambda \subset \fS_n$ be the parabolic subgroup $\fS_\lambda = \fS_{\lambda_1}\times \ldots \times \fS_{\lambda_r}$.  We prove (see Theorem~\ref{thm:main isom}):

\begin{theorem}\label{thm-main}
\begin{enumerate}
\item For any compositions $\lambda = (\lambda_1, \ldots, \lambda_r), \mu = (\mu_1, \ldots, \mu_s)$ there is a natural isomorphism
\[ \Hom(\cF_\lambda,\cF_\mu) = K(\Sh [\fS_\lambda \backslash \fS_n/\fS_\mu])\otimes_\Z \kk.\]
Here $K(\Sh [\fS_\lambda \backslash \fS_n/\fS_\mu])$ denotes the Grothendieck group of $\Sh [\fS_\lambda \backslash \fS_n/\fS_\mu]$, the category of $\fS_\lambda$-equivariant sheaves of complex vector spaces on the set of left cosets $\fS_n/\fS_\mu$, or equivalently the category of sheaves on the quotient stack $ [\fS_\lambda \backslash \fS_n/\fS_\mu]$.
\item Under the isomorphisms above, the composition map
\[ \Hom(\cF_\mu,\cF_\nu) \times \Hom(\cF_\lambda,\cF_\mu) \to \Hom(\cF_\lambda,\cF_\nu) \]
corresponds to the map on Grothendieck groups induced by the natural convolution operation
\[ \Sh[\fS_\mu \backslash \fS_n /\fS_\nu] \times \Sh[\fS_\lambda \backslash \fS_n/\fS_\mu] \to \Sh[\fS_\lambda \backslash \fS_n /\fS_\nu]. \]
\end{enumerate}
\end{theorem}

\begin{example}\label{ex:extreme cases}
\begin{enumerate}
\item Consider the special case $\lambda = \mu = (n)$.  Then 
\[\Hom(\cF_n,\cF_n) = K(\Sh[\fS_n \backslash \fS_n /\fS_n])\otimes \kk = K(\Sh[~\cdot~/\fS_n])\otimes \kk = R_\C(\fS_n) \otimes \kk\]
and convolution of sheaves corresponds to multiplication in the character ring.  Thus Theorem~\ref{thm-main} recovers Theorem~\ref{thm-spr} as a special case.
\item At the other extreme, let $\lambda = \mu = (1^n)$.  Then
\[\Hom(\cF_{(1^n)},\cF_{(1^n)}) = K(\Sh[\fS_{(1^n)} \backslash \fS_n /\fS_{(1^n)}])\otimes \kk = K(\Sh[\fS_n])\otimes \kk = \kk[\fS_n]\]
and convolution of sheaves corresponds to multiplication in the group algebra.
\end{enumerate}
\end{example}

To summarize the results above:

\begin{theorem}\label{cor:main}
The category of perverse sheaves $\Perv_\cS(S^n (\C^2),\kk)$ is equivalent to the category of $R$-modules, where $R$ is the algebra with orthogonal idempotents $e_\lambda$ for each composition $\lambda$ of $n$ such that 
$$e_\mu R e_\lambda = K(\Sh [\fS_\lambda \backslash \fS_n/\fS_\mu])\otimes_\Z \kk$$
and the multiplication map $e_\nu R e_\mu \times e_\mu R e_\lambda \to e_\nu R e_\lambda$ is the map induced by the convolution of sheaves.
\end{theorem}

Finally, the proof of Theorem~\ref{thm-equiv} is complete with the observation that $HS_\kk(n,n)$ is Morita equivalent to the algebra $R$ described in Theorem~\ref{cor:main}.

\subsection{Outline of the paper}

We begin in Section~\ref{sec:2-category} by considering the 2-category underlying the algebra $R$ of Theorem~\ref{cor:main} and give a simple presentation of $R$ in terms of generators and relations.  We also explain why $R$ is Morita equivalent to $HS_\kk(n,n)$.
%The presentation is used later in the paper to prove our main theorem by constructing corresponding endomorphisms of the projective generator $\cF = \oplus_\lambda \cF_\lambda$ and checking that they satisfy the relations in $R$ and that the induced map is indeed an isomorphism.  

In Section~\ref{sec:perv} we introduce the category $\Perv_\cS(S^n (\C^2),\kk)$ and define for each composition $\lambda$ of $n$ two important functors, a restriction and a symmetrization functor, between perverse sheaves on $S^n(\C^2)$ and 
perverse sheaves on $S^{\lambda_1}(\C^2) \times \ldots \times S^{\lambda_r}(\C^2)$ constructible with respect to the product stratification.  The key property is that these functors are biadjoint on constructible perverse sheaves.  The functors and the biadjunction are later used to define the desired endomorphisms of $\cF$ and to check the relations.

In Section~\ref{sec:HCsheaves} we construct the perverse sheaves $\cF_\lambda$ and prove that they represent hyperbolic stalk functors.  It follows that each $\cF_\lambda$ is projective (and as they are self-dual, also injective) and together they generate the category $\Perv_\cS(S^n (\C^2),\kk)$.

In Section~\ref{sec:hom} we construct a homomorphism $R \to \End(\cF)$ using our presentation of $R$ from Section~\ref{sec:2-category}.  In checking that the relations hold, a key role is played by a Cartesian diagram relating the restriction and symmetrization functors with respect to different compositions of $n$.

Finally it remains to show that the homomorphism constructed in Section~\ref{sec:hom} is an isomorphism.  This is proved in Section~\ref{sec:iso}, first for the special case of Theorem~\ref{thm-spr} and then in general.

We conclude with Section~\ref{sec:functors} in which we prove Theorem~\ref{thm-interp}  -- an interpretation of the restriction and symmetrization functors in terms of the algebra $R$.

\begin{remark}
For the sake of simplicity, in this introduction we have defined the objects $\cF_\lambda$ and the idempotents of $R$ to be parametrized by compositions $\lambda$ of $n$.  In the body of the text we find it convenient to work more generally with arbitrary set partitions.
\end{remark}

\begin{remark}
The endomorphism algebra $\End(\cF_n)$ (and more generally $\End(\cF)$) can be reinterpreted as a convolution algebra on Borel--Moore homology in the sense of Ginzburg \cite[Theorem 8.6.7]{CG}.  Such algebras can be studied by finding simple Lagrangian correspondences which generate them.  In the setting of Hilbert schemes, related Lagrangian correspondences were studied extensively by Nakajima \cite{Nak97,NakBook} and Grojnowski \cite{Groj}.  Our approach is perhaps closest to the operators studied by Grojnowski.  One could alternatively deduce our theorem using Nakajima's work instead, however the operators he studies do not generate the algebra over the integers and so more care is needed to keep track of denominators.

Another possible approach to our main results, which we hope to explore in future work, would be to use the work of Bezrukavnikov--Kaledin \cite{BezKal} or Bridgeland--King--Reid and Haiman \cite{BKR,Haim}, which gives an isomorphism between the $K$-theory of $\Hilb^n(\C^2)$ and the $\fS_n$-equivariant K-theory of $(\C^2)^n$.  Using this isomorphism, the top Borel-Moore homology of the Ginzburg convolution variety inside $\Hilb^n(\C^2) \times\Hilb^n(\C^2)$ should be identified with a piece of the $(\fS_n\times \fS_n)$-equivariant K-theory of the union of graphs of elements of $\fS_n$ acting on $(\C^2)^n$.  We expect that this approach will provide a more conceptual explanation for the appearance of equivariant sheaves on $\fS_n$ in our result. 
\end{remark}

%%%%%%%%%%%%%%%%%%%
\subsubsection*{Acknowledgements}
%%%%%%%%%%%%%%%%%%%

The authors thank David Ben-Zvi, Ben Elias, Geordie Williamson, Anthony Licata, C\'edric Bonnaf\'e, Hiraku Nakajima and Roman Bezrukavnikov for stimulating remarks and conversations.
The first author thanks the Institute for Advanced Study for excellent virtual working conditions during the semester when some of these results were worked out.  The second author was supported in part by NSF grant DMS-1802299.

%%%%%%%%%%%%%%%%%%%
\section{A $2$-category and its decategorification}\label{sec:2-category}
%%%%%%%%%%%%%%%%%%%
In this section, we define a $\C$-linear (strict) $2$-category $\cC_{[n]}$, and consider the associated ``decategorified" algebra $R_{[n]} = K(\cC_{[n]})$.  We will give a simple presentation of this algebra in terms of generators and relations.

%%%%%
\subsection{Set partitions and integer partitions}\label{sec:partitions}
%%%%%

We begin by establishing some notation and terminology associated with set partitions and integer partitions.  We let $\Part{[n]}$ denote the set of partitions of the set $[n] := \{1,2,\dots,n\}$.  By definition a partition $\lambda \in \Part{[n]}$ is a set of non-empty subsets of $[n]$ such that each element of $[n]$ is in exactly one of these subsets, or equivalently it can be described by giving an equivalence relation $\stackrel{\lambda}{\sim}$ on $[n]$ whose equivalence classes are the subsets of $\lambda$.  

We order $\Part{[n]}$ by refinement:
we put $\lambda \le \mu$ if every part of $\lambda$ is contained in a part of $\mu$, or equivalently if $i \stackrel{\lambda}{\sim} j$ implies $i \stackrel{\mu}{\sim} j$.  
The unique maximum element of $\Part{[n]}$ is the partition with a single part, and the unique minimum is the partition with $n$ parts.

For a partition $\lambda \in \Part{[n]}$, we have a subgroup $\fS_\lambda$ of the symmetric group $\fS_n$ consisting of elements which permute elements of each part of the partition among themselves: 
\[\fS_\lambda = \prod_{I \in \lambda} \fS_I,\]
considered as a subset of $\fS_n$ in the obvious way.
This subgroup can be defined equivalently in terms of the equivalence relation $\stackrel\lambda\sim$ given by
\[\fS_\lambda = \{w \in \fS_n \mid w(i) \stackrel{\lambda}{\sim} i \,\mbox{ for all }\, i \in [n]\}.\]

The poset $\Part{[n]}$ is a lattice, and the meet $\lambda \wedge \mu$ of partitions $\lambda, \mu \in \Part{[n]}$ is the partition whose parts are all nonempty intersections of parts of $\lambda$ with parts of $\mu$.  We clearly have
\[\fS_{\lambda \wedge \mu} = \fS_\lambda \cap \fS_\mu.\]

The action of $\fS_n$ on $[n]$ induces an action on $\Part{[n]}$, and 
for any $w \in \fS_n$ and $\lambda \in \Part{[n]}$ we have
\[\fS_{w\lambda} = w\fS_\lambda w^{-1}.\]
Thus up to conjugacy the subgroups $\fS_\lambda$ are determined by elements of the  quotient set $\Part{[n]}/\fS_n$.  
An element of this quotient depends only on the sizes of the equivalence classes and not on their particular elements; in other words, it is an integer partition of $n$.  We denote this quotient set by $\Part{n}$, and we denote the natural quotient map $\Part{[n]} \to \Part{n}$ by $\lambda \mapsto \hat\lambda$.  For instance, if $n = 6$ and $\lambda = \{\{1,2,4\},\{3\},\{5,6\}\}$, then $\hat\lambda = (3,2,1)$.  

At times it will be useful to fix a section of this quotient, or in other words to consider $\Part{n}$ as a subset of $\Part{[n]}$.  We will do this in the obvious way where the elements of $[n]$ are grouped into consecutive intervals, with the largest groups coming first.  None of our results would be affected if a different embedding were used.

The refinement order $\le$ on $\Part{[n]}$ induces a partial order $\preceq$ on $\Part{n}$: we say that $\hat \lambda \preceq \hat \mu$ if $\lambda \le z\mu$ for some $z \in \fS_n$.

%%%%%
\subsection{Defining $\cC_{[n]}$: objects and $1$-morphisms}\label{subsec-Cn}
%%%%%

We now define our $2$-category $\cC_{[n]}$.
Objects of $\cC_{[n]}$ are set partitions $\lambda \in \Part{[n]}$. 
For $\mu$, $\nu \in \Part{[n]}$, the $1$-morphisms 
$\mu \to \nu$ form a category given by 
\[\hom(\mu, \nu) = \Sh([\fS_\nu \bs \fS_n / \fS_\mu]),\]
the category of finite-dimensional $\C$-sheaves on the finite quotient stack $[\fS_\nu \bs \fS_n / \fS_\mu]$.  

In elementary terms, an object $F$ of $\hom(\mu, \nu)$ is given by a finite-dimensional vector space $F_x$ for all $x\in \fS_n$, 
together with a left action of $\fS_\nu$ and a right action of $\fS_\mu$ on the total space 
$\bigoplus_{x \in \fS_n} V_x$, commuting with each other,  such that 
\[g \cdot F_x \subset F_{gx} \; \mbox{and} \; F_x \cdot h \subset F_{xh}\]
for all $x \in \fS_n$, $g \in \fS_\nu$, $h \in \fS_\mu$.

For an element $w$ in $\fS_n$, we use the notation ${}_\nu[w]_\mu$ for the stacky double coset $[\fS_\nu \backslash (\fS_\nu w \fS_\mu)/\fS_\mu]$ containing a permutation $w\in \fS_n$.  It is a single point with automorphism group $\fS_\nu \cap w\fS_\mu w^{-1} = \fS_{\nu \wedge w\mu}$.  
The double quotient $[\fS_\nu \bs \fS_n / \fS_\mu]$ is the disjoint union of ${}_{\nu}[w]_\mu$ as $w$ runs over a set of representatives of double cosets.  
%To fix a particular choice, we let 
%$\ZZ\nu{n}\mu$ denote the set of representatives of double cosets of minimal length.  More generally, if $\mu, \nu \le \lambda$, we let $\ZZ{\nu}{\lambda}{\mu}$ be the set of minimal length representatives of double cosets in $\fS_\nu \backslash\fS_\lambda/\fS_\mu$.  One can verify that nothing that follows would be affected by making a different choice of representatives.

If we fix partitions $\mu, \nu \in \Part{[n]}$, we can omit the subscripts and write $[w] = {}_\nu [w]_\mu$. Then we let $\hom_{[w]}(\mu,\nu)$ denote the full subcategory of $1$-morphisms supported on $[w]$. 
If $F$ is supported on the double coset
$[e]$ of the identity, then conjugation makes $F_e$ into a representation of $\fS_{ \nu\wedge \mu}$.  In fact, this defines an equivalence of $1$-categories 
\begin{equation}\label{eqn:representation equivalence}
\hom_{[e]}(\mu,\nu) \stackrel{\sim}{\longrightarrow} \C\fS_{ \nu\wedge \mu}\mmod.
\end{equation}
More generally, a $1$-morphism $F \in \hom_{[w]}(\mu,\nu)$ supported on $[w]$ is equivalent to a
representation of $\fS_{\nu \wedge w\mu}$, via the action on $F_w$ given by 
\[z \cdot v = z v w^{-1}z^{-1}w.\]

\subsection{Composition in $\cC_{[n]}$} For $\lambda, \mu, \nu \in \Part{[n]}$, composition of $1$-morphisms is given by a functor
\[\hom(\mu, \nu) \times \hom(\lambda, \mu) \to \hom(\lambda, \nu),\;\; (A, B) \mapsto A \circ B\]
defined as follows.  If $A \in \hom(\mu,\nu)$ and $B \in \hom(\lambda,\mu)$, then 
$A \boxtimes B$ is a sheaf on $\fS_n \times \fS_n$, equivariant for an action of 
$\fS_\nu \times \fS_\mu \times \fS_\mu \times \fS_\lambda$. Restricting to the action of the 
subgroup 
\[\{(u, v, v^{-1}, w) \mid u \in \fS_\nu, \, v \in \fS_\mu, \, w \in \fS_\lambda\} \cong \fS_\nu \times \fS_\mu \times\fS_\lambda\] and pushing forward by multiplication
\[\fS_n \times \fS_n \to \fS_n\]
gives an $\fS_\lambda \times \fS_\mu \times \fS_\nu$-equivariant sheaf on $\fS_n$, where $\fS_\mu$ acts trivially on $\fS_n$.  The composition $A \circ B$ is defined to be the $\fS_\mu$-invariants inside this sheaf.

Less formally, the composition is given by the formula
\[(A \circ B)_x = \Big[\bigoplus_{yz=x} A_y \otimes_\C B_z\Big]^{\fS_\mu},\]
where the action of $\fS_\mu$ is by $w \ast(a \otimes b) = (a\cdot w^{-1}) \otimes (w \cdot b)$.  

From this it is easy to check the strict associativity of the composition.  We can also describe the identity elements: the identity ${}_\mu{\mathbf I}_\mu$ on $\mu$ is the sheaf given by
\[\left(_\mu{\mathbf I}_\mu\right)_x = 
\begin{cases}
\C & x \in \fS_\mu \\
0 & x \notin \fS_\mu
\end{cases}\] 
with the left and right actions of $\fS_\mu$ given by identity maps.

%%%%%
\subsection{Generating $1$-morphisms}
%%%%%

We consider three special classes of $1$-morphisms in $\cC_{[n]}$.

\begin{enumerate}
	\item Generalizing the identity morphisms in the previous section, for partitions $\mu \le \nu$ we let 
	${_\nu\bI_\mu}\in \hom(\mu,\nu)$ be the sheaf which is $\C$ on every $w \in \fS_\nu$ and $0$ otherwise, and with the left and right actions given by the identity maps.  In other words, it is the sheaf which is sent to the trivial representation of $\fS_{\mu \wedge \nu} = \fS_\mu$ under the equivalence \eqref{eqn:representation equivalence}.  We define $_\mu\bI_\nu \in \hom(\nu,\mu)$ similarly.
	
	\item For any $w \in \fS_n$ and any $\lambda \in \Part{[n]}$, let ${}_{w\lambda}(w_*)_{\lambda}$ be the sheaf supported on the double coset ${}_{w\lambda}[w]_\lambda$ which is sent to the trivial representation under the equivalence 
	\[\hom_{[w]}(\lambda,w\lambda) \stackrel{\sim}{\longrightarrow} \C\fS_{w\lambda\wedge w\lambda}\mmod = \C\fS_{w\lambda}\mmod.\]
	In other words, it is $\C$ on every element of the double coset ${}_{w\lambda} [w]_\lambda$, zero on all other elements, and the left and right actions are by identity maps.  We will sometimes omit the subscripts and write simply $w_*$ in formulas where the source and/or target objects are clear because of another morphism it is composed with.
%	We put 
%	\[w_* = \sum_{\lambda\in \Part{[n]}} {}_{w\lambda}(w_*)_{\lambda}.\]
	
	\item For any $\lambda \in \Part{[n]}$ and any representation $V$ of $\fS_\lambda$, let $\langle V\rangle \in \hom(\lambda,\lambda)$ be given by 
	$\langle V\rangle_w = V$ if $w \in \fS_\lambda$ and is zero otherwise, with the left action given by the $\fS_\lambda$-module structure and the right action given by identity maps.	Thus this sheaf 
	is sent to $V$ under the equivalence \eqref{eqn:representation equivalence}.  Note that for the trivial representation ${\mathds{1}}_{\fS_\lambda}$ we have $\langle{\mathds{1}}_{\fS_\lambda}\rangle \simeq {_\lambda\bI_\lambda}$.  
\end{enumerate}

\begin{proposition}\label{prop:categorical relations}
	These $1$-morphisms satisfy the following relations:
	\begin{enumerate}[label=(\alph*)]
		\item If $\lambda \le \mu \le \nu$ or $\nu \le \mu \le \lambda$, then
		\[_\nu\bI_\mu \circ {}_\mu\bI_\lambda \simeq {}_\nu\bI_\lambda.\]
		
		\item For any $w, z\in \fS_n$ and any $\lambda \in \Part{[n]}$, we have
		\[{}_{wz\lambda}(w_*)_{z\lambda} \circ {}_{z\lambda}(z_*)_\lambda \simeq {}^{}_{wz\lambda}\left((wz)_*\right)_{\lambda}.\]
		If $w\in \fS_\lambda$, then ${}_\lambda(w_*)_\lambda \simeq {}_\lambda\bI_\lambda \in \hom(\lambda,\lambda)$.
		\item For any $w \in \fS_n$ and any $\mu, \nu \in \Part{[n]}$ with $\mu \le \nu$ or $\nu\le \mu$, we have
		\[w_* \circ {}_\nu\bI_{\mu} \simeq  {}_{w\nu}\bI_{w\mu} \circ w_*.\]
%		If $w\in \fS_\nu$ then $w_* \circ {}_\nu\bI_{\mu} \simeq {}_\nu\bI_{\mu}$, and if $w\in \fS_\mu$ then ${}_{\nu}\bI_{\mu} \circ w_* \simeq {}_{\nu}\bI_{\mu}$.
		
		\item If $\lambda, \mu \le \nu$, then
		\begin{align*}
		{}_\mu\bI_\nu \circ {}_\nu\bI_\lambda & \simeq \bigoplus_z {}_\mu\bI_{\mu\wedge z\lambda} \circ z_* \circ  {}_{z^{-1}\mu\wedge\lambda }\bI_\lambda \\
		& \simeq \bigoplus_z {}_\mu\bI_{\mu\wedge z\lambda} \circ {}_{\mu\wedge z\lambda}\bI_{z\lambda} \circ z_*,
		\end{align*}
		where the sum is over a set of representatives $z$ of double cosets in $\fS_\mu \backslash \fS_\nu / \fS_\lambda$.
		
		\item If $\lambda \le \mu$ then for any $V \in \C\fS_\lambda\mmod$ and $W \in \C\fS_\mu\mmod$, we have
		
		\[{}_\lambda\bI_\mu \circ \langle W\rangle \simeq \langle W|_{\fS_\lambda}\rangle \circ {}_\lambda\bI_\mu,\] \[\langle W\rangle \circ  {}_\mu\bI_\lambda \simeq {}_\mu\bI_\lambda \circ \langle W|_{\fS_\lambda}\rangle, \;\mbox{and}\]
		\[{}_\mu\bI_\lambda \circ \langle V \rangle \circ {}_\lambda\bI_\mu \simeq \langle \Ind_\lambda^\mu V\rangle,\]
		where $\Ind_\lambda^\mu$ denotes induction of representations from $\fS_\lambda$ to $\fS_\mu$.
		
		\item For any $V \in \C\fS_\lambda\mmod$ and any $w \in \fS_n$, there is an isomorphism
		\[w_* \circ \langle V\rangle \simeq \langle {}^wV \rangle \circ w_*,\]
		where ${}^wV$ is the representation of $\C\fS_{w\lambda}$ obtained by pulling $V$ back by conjugation  
		$\fS_{w\lambda} = w\fS_\lambda w^{-1} \to \fS_\lambda$.
		
		\item If $V, W \in \C\fS_\lambda\mmod$, then 
		\[\langle V\rangle \circ \langle W\rangle \simeq \langle V \otimes W\rangle.\]
	\end{enumerate}
\end{proposition}

\begin{proof}
	For the first four relations (a)-(d), it is easy to see that each sheaf is a rank one constant sheaf supported on a subset of $\fS_n$, with left and right actions given by identity maps.  So verifying these isomorphisms only requires checking that the supports are the same on both sides, which is straightforward.  For example, relation (d) expresses the set-theoretic decomposition of $\fS_\nu$ into $\fS_\mu$-$\fS_\lambda$ double cosets. The remaining relations (e), (f), and (g) are also easily verified.
\end{proof}

\begin{remark}\label{rmk:C[n] and Cn}
	The relation (b) says that the $1$-morphisms $w_*$ are all invertible, so any two set partitions with the same underlying integer partition are isomorphic. Thus we can get an equivalent $2$-category $\cC_n$ by restricting the objects to one representative for each integer partition.	For instance, one could take the image of the inclusion $\Part{n} \subset \Part{[n]}$ described at the end of Section \ref{sec:partitions}.
\end{remark}

The next result gives a presentation for the $1$-morphisms in $\cC_{[n]}$.

\begin{proposition}\label{prop:generating the category}
Any object $F$ of the morphism category $\hom(\mu,\nu)$ has a unique expression of the form
\begin{equation}\label{eqn:double coset sheaf}
	F\simeq \bigoplus_{w} {_\nu}\bI_{\nu\wedge w\mu} \circ \langle V_w \rangle \circ {}_{\nu\wedge w\mu}\bI_{w\mu} \circ {}_{w\mu}(w_*)_\mu,
\end{equation}
where the sum runs over a set of representatives for the double cosets in $\fS_\nu \backslash \fS_n / \fS_\mu$.  In particular, the $1$-morphisms ${}_\nu\bI_\mu$, $z_*$, $\langle V\rangle$ generate all $1$-morphisms in $\cC_{[n]}$, up to isomorphism.
	
With respect to the normal form \eqref{eqn:double coset sheaf}, the relations of Proposition \ref{prop:categorical relations} completely determine the composition law on $1$-morphisms.
\end{proposition}
\begin{proof}
An object $F$ of $\hom(\mu,\nu)$ is a sheaf on  $[\fS_\nu \backslash \fS_n / \fS_\mu]$, so it splits as a direct sum of its restriction to double cosets.  The restriction to the double coset $[w]$ corresponds to a representation $V_w$ of $\fS_{\nu\wedge w\mu}$, giving the isomorphism \eqref{eqn:double coset sheaf}.

	To show that our relations determine the composition of $1$-morphisms expressed in the form \eqref{eqn:double coset sheaf}, it is enough to show that the composition of a summand \[ F|_{[w]}= {_\nu}\bI_{\nu\wedge w\mu} \circ \langle V \rangle \circ {}_{\nu\wedge w\mu}\bI_{w\mu} \circ {}_{w\mu}(w_*)_\mu\]
     with 
    each of the three types of generators can be transformed using those relations into an expression of the form \eqref{eqn:double coset sheaf}.
    
	If we compose on the left by a $1$-morphism $z_*$, the relations (b), (c), and (f) allow it to pass through to the right and combine with $w_*$, resulting in an expression which is again of the same form, but with $\nu$ replaced by $z\nu$ and $V$ replaced by ${}^zV$.  This also means that we can and will assume that $w$ is the identity in the remaining computations.
	
	If we compose on the left by a $1$-morphism $\langle W\rangle$, where $W$ is a representation of $\fS_\nu$, then using relations (e) and (g) gives
	\[{_\nu}\bI_{\nu\wedge\mu} \circ \langle W|_{\fS_{\nu\wedge \mu}} \otimes V \rangle \circ {}_{\nu\wedge \mu}\bI_{\mu},\]
	which is again of the required form.
	
	Finally, if we compose on the left by ${}_\lambda\bI_\nu$, there are two cases. If
	$\lambda \ge \nu$, then using relations (a) and (e), we get
	\begin{align*}
	{}_\lambda\bI_\nu \circ F|_{[w]} & \simeq {_\lambda}\bI_{\nu\wedge \mu} \circ \langle V \rangle \circ {}_{\nu\wedge \mu}\bI_{\mu}\\
	& \simeq {_\lambda}\bI_{\lambda\wedge \mu} \circ {}_{\lambda\wedge \mu}\bI_{\nu\wedge \mu} \circ \langle V \rangle \circ {}_{\nu\wedge \mu}\bI_{\lambda\wedge \mu} \circ {}_{\lambda\wedge \mu}\bI_{\mu}\\
	& \simeq {_\lambda}\bI_{\lambda\wedge \mu} \circ \langle \Ind^{\lambda\wedge \mu}_{\nu\wedge \mu} V \rangle \circ {}_{\lambda\wedge \mu}\bI_{\mu}.
	\end{align*}
	On the other hand, if $\lambda \le \nu$, then the relation (d) gives
	\begin{align*}
	{}_\lambda\bI_\nu \circ F & \simeq \bigoplus_z {}_\lambda\bI_{\lambda\wedge z(\nu\wedge\mu)} \circ {}_{\lambda\wedge z(\nu\wedge\mu)}\bI_{z(\nu \wedge \mu)} \circ z_* \circ \langle V \rangle \circ {}_{\nu\wedge \mu}\bI_{\mu}\\
	& \simeq  \bigoplus_z {}_\lambda\bI_{\lambda\wedge z(\nu\wedge\mu)} \circ \langle {}^zV|_{\fS_{\lambda \wedge z(\nu\wedge \mu)}}\rangle \circ {}_{\lambda\wedge z(\nu\wedge\mu)}\bI_{z(\nu \wedge \mu)} \circ {}_{z(\nu\wedge\mu)}\bI_{z\mu} \circ z_*\\
	& \simeq \bigoplus_z {}_\lambda\bI_{\lambda\wedge z\mu} \circ \langle \Ind_{\lambda \wedge z(\nu\wedge \mu)}^{\lambda\wedge z\mu}({}^zV|_{\fS_{\lambda \wedge z(\nu\wedge \mu)}})\rangle \circ {}_{\lambda \wedge z\mu}\bI_{z\mu} \circ z_*,
	\end{align*}
	where $z$ runs over a set of representatives of double cosets $\fS_\lambda \backslash \fS_\nu / \fS_{\nu\wedge\mu}$.
\end{proof}

%%%%%
\subsection{The decategorified ring}\label{sec:the ring R}
%%%%%

Taking Grothendieck groups of the morphism categories of our $2$-category $\cC_{[n]}$ gives a ring
\[R = R_{[n]} := K(\cC_{[n]}) = \bigoplus_{\lambda,\mu\in \Part{[n]}} K(\hom(\lambda,\mu))\]
with multiplication induced by the composition in $\cC_{[n]}$.
We denote the summand $K(\hom(\lambda,\mu))$ by ${}_\mu R_{\lambda}$.

 It is then clear by Remark \ref{rmk:C[n] and Cn} that the algebra $R_n := K(\cC_n)$ can be viewed as a subalgebra of $R_{[n]}$, and the inclusion is a Morita equivalence.

The special $1$-morphisms considered in the previous section induce elements of $R_{[n]}$:
\begin{itemize}
	\item ${}_\mu\bone_{\lambda} = [{}_\mu\bI_\lambda] \in {}_\mu R_\lambda$ for partitions $\lambda \le \mu$ or $\mu \le \lambda$, 
	\item $w_* = \sum_{\lambda \in \Part{[n]}} [{}_{w\lambda}(w_*)_{\lambda}]$ for $w \in \fS_n$, and
	\item $[V]$ for $V \in \C\fS_\lambda\mmod$.
\end{itemize}
The elements ${}_\lambda \bone_\lambda$ are mutually orthogonal idempotents whose sum is the identity, but they are not in general primitive.

By Proposition \ref{prop:generating the category} these elements generate $R_{[n]}$, and the relations of Proposition \ref{prop:categorical relations} induce relations which give a presentation of $R_{[n]}$.  However, the presentation on the decategorified level admits a significant simplification, given in the next Theorem.  

\begin{theorem}\label{thm: presenting Rn}
	The ring $R_{[n]}$ is generated over $\Z$ by the elements ${}_\lambda\bone_{\mu}$ and $w_*$.
	For any $\lambda,\nu \in \Part{[n]}$, a $\Z$-basis of ${}_\nu R_\lambda$ is given by the set of expressions
	\begin{equation}\label{eqn:Z-basis}
	{}_\nu\bone_\mu \cdot {}_\mu\bone_{w\lambda} \cdot w_*,
	\end{equation}
	where $w \in \fS_n$ runs over a set of representatives of double cosets
	$\fS_\nu \backslash \fS_n / \fS_\lambda$, and $\mu$ runs over a set of representatives of orbits of 
	$\fS_{\nu\wedge w\lambda}$ acting on the interval $[(1^n),\nu\wedge w \lambda]$.
	Furthermore, the relations on $R_{[n]}$ induced by the relations (a)-(d) of Proposition \ref{prop:categorical relations} give a complete set of relations for $R_{[n]}$.  
	
	In other words, the relations in $R_{[n]}$ are generated by
	\begin{enumerate}[label=(\alph*)]
		\item If $\lambda \le \mu \le \nu$ or $\nu \le \mu \le \lambda$, then
		\[_\nu\bone_\mu \cdot {}_\mu\bone_\lambda = {}_\nu\bone_\lambda.\]
		The idempotents ${}_\lambda\bone_\lambda$ are orthogonal: if $\mu\ne\lambda$, then
		\[{}_\lambda\bone_\lambda \cdot {}_\mu\bone_\mu = 0.\]
		\item For any $w, z\in \fS_n$, we have
		\[w_* \cdot z_* = (wz)_*\]
		and $e_* = \sum_\lambda {}_\lambda\bone_\lambda$.
		\item For any $w \in \fS_n$ and any $\mu, \nu \in \Part{[n]}$ with $\mu \le \nu$ or $\nu\le \mu$, we have
		\[w_* \cdot {}_\nu\bone_{\mu} =  {}_{w\nu}\bone_{w\mu} \cdot w_*.\]
		If $w\in \fS_\nu$ then $w_* \cdot {}_\nu\bone_{\mu} = {}_\nu\bone_{\mu}$, and if $w\in \fS_\mu$ then ${}_{\nu}\bone_{\mu} \cdot w_* = {}_{\nu}\bone_{\mu}$.
		
		\item If $\lambda, \mu \le \nu$, then
		\begin{align*}
		{}_\mu\bone_\nu \cdot {}_\nu\bone_\lambda & = \sum_z {}_\mu\bone_{\mu\wedge z\lambda} \cdot z_* \cdot  {}_{z^{-1}\mu\wedge\lambda }\bone_\lambda \\
		& = \sum_z {}_\mu\bone_{\mu\wedge z\lambda} \cdot {}_{\mu\wedge z\lambda}\bone_{z\lambda} \cdot z_*,
		\end{align*}
		where the sum is over a set of representatives $z$ of double cosets in $\fS_\mu \backslash \fS_\nu / \fS_\lambda$.
	\end{enumerate}
	
\end{theorem}
\begin{proof}
	To see that the elements $[V]$ are redundant, note that the relation (5) of Proposition \ref{prop:categorical relations} implies that for $\lambda \le \mu$, we have
	\[{}_\mu\bone_\lambda \cdot {}_\lambda\bone_\mu = [\Ind_\lambda^\mu {{\mathds{1}}_{\fS_\lambda}}].\]
	If we fix $\mu$ and let $\lambda$ run over a set of representatives of $\fS_\mu$-orbits on the set of $\lambda \le \mu$, then the resulting elements of $_\mu R_\mu$	
	have the same $\Z$-span as the elements $[V]$, $V \in \C\fS_\mu\mmod$.  This follows from the well-known fact that the induced modules $\Ind_{\fS_\lambda}^{\fS_n} \mathds{1}_{\fS_\lambda}$ have a decomposition into Specht modules which is triangular in the dominance order, with ones on the diagonal (see \cite[Corollary 2.4.7]{Sag-book}, for instance).
	
	Given $\lambda,\mu,\nu \in \Part{[n]}$ and $w \in \fS_n$ with $\mu \le \nu \wedge w\lambda$, the expression 
	\eqref{eqn:Z-basis} is the class of a sheaf on $[\fS_\nu \backslash \fS_n / \fS_\lambda]$ supported on the double coset $[w]$, whose stalk at $w$ is the representation $\Ind_\mu^{\nu \wedge w\lambda} \mathds{1}_{\fS_\mu}$ of $\fS_{\nu\wedge w\lambda}$.  Thus if we fix $\nu$ and $\lambda$ and take one $w$ in each double coset and one $\mu$ in each $\fS_{\nu\wedge w\lambda}$-orbit, the resulting elements give a $\Z$-basis for ${}_\nu R_\lambda$.
	
	Now let $R'$ be the algebra generated by symbols  
	${}_\lambda\bone_{\mu}$, for $\lambda \le \mu$ or $\mu \le \lambda$ and $w_*$ for $w \in \fS_n$, subject to the relations (a)-(d) of this theorem.  We have shown that there is a surjective homomorphism $R' \to R_{[n]}$.  To show it is an isomorphism, it is enough to show that the expressions \eqref{eqn:Z-basis} that give a basis of $R_{[n]}$ also span $R'$.  To see this, take an arbitrary product of our generators in $R'$.  Using relations (b) and (c) we can write it in the form
	\[{}_{\alpha_1}\bone_{\alpha_2} \cdot {}_{\alpha_2}\bone_{\alpha_3} \cdot \;\dots\; \cdot {}_{\alpha_{r-1}}\bone_{\alpha_r} \cdot w_*\]
	where $\alpha_i$ is comparable to $\alpha_{i+1}$ for $1 \le i < r$.  
	
	If $\alpha_i = \alpha_{i+1}$, then ${}_{\alpha_i}\bone_{\alpha_{i+1}}$ is an idempotent which can be removed from the sequence using relation (a).  Using relation (a) again we can assume that for three consecutive elements 
	$\alpha_{i-1}, \alpha_i, \alpha_{i+1}$ we either have
	$\alpha_i > \alpha_{i-1}, \alpha_{i+1}$ (a ``peak") or $\alpha_i < \alpha_{i-1}, \alpha_{i+1}$ (a ``valley").  If there are any peaks, then relation (d) can be applied to a maximal peak, replacing the product by a sum of terms which have valleys in that position, and whose maximal peaks are lower.  
	
	Iterating this, we see that our element of $R'$ is a sum of elements of the form 
	\begin{equation}\label{eqn:R' generator}
	{}_\nu\bone_\mu \cdot {}_\mu\bone_{w\lambda} \cdot w_*
	\end{equation}
	with $\mu \le \nu$ and $\mu \le w\lambda$.  If $x \in \fS_\nu$, we have
	\[{}_\nu\bone_\mu \cdot{}_\mu\bone_{w\lambda} \cdot w_* = x_* \cdot {}_\nu\bone_\mu \cdot{}_\mu\bone_{w\lambda} \cdot w_* = {}_{\nu}\bone_{x\mu} \cdot {}_{x\mu}\bone_{xw\lambda} \cdot (xw)_*\]
	using relations (b) and (c).  Similarly if $y \in \fS_\lambda$, we get
	\[{}_\nu\bone_\mu \cdot{}_\mu\bone_{w\lambda} \cdot w_* = {}_\nu\bone_\mu \cdot{}_\mu\bone_{w\lambda} \cdot (wyw^{-1})_*w_* = {}_\nu\bone_\mu \cdot{}_\mu\bone_{w\lambda} \cdot (wy)_*,\]
	since $wyw^{-1} \in \fS_{w\lambda}$.  Thus $R'$ is spanned by elements \eqref{eqn:R' generator} where $w$ lies in a set of double coset representatives for $\fS_\nu \backslash \fS_n / \fS_\lambda$.
	
	On the other hand, if $x \in \fS_{\nu\wedge w \lambda}$, we have
	\begin{align*}
	{}_\nu\bone_\mu \cdot{}_\mu\bone_{w\lambda} \cdot w_* & = {}_\nu\bone_\mu \cdot{}_\mu\bone_{w\lambda} \cdot x_* \cdot w_* \\
	& = x_* \cdot {}_\nu\bone_{x^{-1}\mu} \cdot{}_{x^{-1}\mu}\bone_{w\lambda} \cdot  w_*\\
	& =  {}_\nu\bone_{x^{-1}\mu} \cdot{}_{x^{-1}\mu}\bone_{w\lambda} \cdot  w_*\\
	\end{align*}
	again using relations (b) and (c).  So we can further restrict the elements \eqref{eqn:R' generator}
    by taking one representative $\mu$ for each orbit of $\fS_{\nu\wedge w \lambda}$ on partitions below $\nu\wedge w \lambda$. This completes the proof of the Theorem.
\end{proof}

%%%%%
\begin{example}\label{ex:n=3}

As an example, let us look at the ring $R_3$.  There are three idempotents: $e_{(1^3)}$, $e_{(21)}$, and $e_{(3)}$.  Then the category of $R_3$-modules is identified with representations of the quiver

\[\xymatrix@C+20pt{V_{(1^3)} \ar@(u,ul)_s[]\ar@(d,dl)[]^t\ar@<.5ex>[r]^p & V_{(21)}\ar@<.5ex>[l]^q \ar@<.5ex>[r]^x & V_{(3)}\ar@<.5ex>[l]^y
}\]
where we put
\[\begin{matrix}
s = (1 2) & t = (2 3) \in \fS_3 \\
p = {}_{(21)}1_{(1^3)} & q = {}_{(1^3)}1_{(21)} \\
x = {}_{(3)}1_{(21)} & y = {}_{(21)}1_{(3)},
\end{matrix}
\]
subject to the relations (from Theorem \ref{thm: presenting Rn})
\begin{align*}
s^2 & = t^2 = (st)^3 = e_{(1^3)}  &\mbox{(relation (b) + relations in $\kk[\fS_3]$)}\\
ps & = p \;\;\; sq = q & \mbox{(from relation (c))}\\ 
xp & = xpt\;\;\;\; qy=tqy & \mbox{(from relation (c))}\\
qp & = e_{(1^3)}+s & \mbox{(from relation (d))}\\
yx & = e_{(21)} + ptq & \mbox{(from relation (d))}.
\end{align*}
Here relation (a) shows that the maps ${}_{(1^3)}1_{(3)}$ and ${}_{(3)}1_{(1^3)}$ are not needed, and relation (c) implies that the $\fS_3$-actions on $V_{(21)}$ and $V_{(3)}$ are trivial.  We have fixed our choice of set partitions so that 
$\fS_{(21)} = \fS_2 \times \fS_1 = \langle s\rangle$.

\end{example}
%%%%%

%%%%%
\subsection{Relation with the Schur and Hilbert-Schur algebras}\label{sec:relation with Schur}
%%%%%

In this section we show that the Hilbert-Schur algebra $HS_\kk(n,n)$ is Morita equivalent to the algebra $R_{[n]} \otimes \kk$.

We first recall an analogous fact for the Schur algebra $S_\kk(n,d)$.

The Schur algebra is the algebra
\[S_\kk(n,d) := \End_{\fS_d}((\kk^n)^{\otimes d})\]
of endomorphisms of the $d$th tensor power $(\kk^n)^{\otimes d}$ that commute with the natural right action of $\fS_d$ permuting the factors.  Tensor products of elements of the standard basis $e_1, \dots, e_n$ of $\kk^n$ give a basis of $(\kk^n)^{\otimes d}$ indexed by $[n]^d$.  This basis is invariant under the action of $\fS_d$, so we can decompose $(\kk^n)^{\otimes d}$ into permutation modules indexed by $\fS_d$ orbits of $[n]^d$:
\[(\kk^n)^{\otimes d} = \bigoplus_{O\in [n]^d/\fS_d} \kk[O].\]
Each orbit $O$ is isomorphic as a right $\fS_d$-set to a quotient $O_\lambda := \fS_{\lambda}\backslash \fS_d$ for a set partition $\lambda \in \Part{[d]}$ with at most $n$ parts, so the module $\kk[O]$ is just the induced module $\Ind_{\fS_\lambda}^{\fS_d} \bone_{\fS_\lambda}$.  Note that every partition $\lambda$ of $d$ with at most $n$ parts is represented by some orbit, and in particular when $n\geq d$, then every partition appears.

For partitions $\lambda,\mu$ of $d$, we have
\begin{align*}
\Hom_{\kk[\fS_d]}(\kk[O_\lambda],\kk[O_\mu]) & = (\kk[\fS_\mu\backslash \fS_d] \otimes \kk[\fS_\lambda\backslash \fS_d]^*)^{\fS_d} \\
& \cong (\kk[\fS_\mu\backslash\fS_d] \otimes \kk[\fS_d/\fS_\lambda])^{\fS_d}\\
& \cong \kk[\fS_\mu\backslash\fS_d \times_{\fS_d} \fS_d/\fS_\lambda]\\
& \cong \kk[\fS_\mu\backslash \fS_d / \fS_\lambda].
\end{align*}
If we view this last space as the space of $(\fS_\mu \times \fS_\lambda)$-invariant functions on $\fS_d$, then the composition of homomorphisms $\phi \in \Hom_{\kk[\fS_d]}(\kk[O_\lambda],\kk[O_\mu])$, $\psi \in \Hom_{\kk[\fS_d]}(\kk[O_\mu],\kk[O_\nu])$
is given by the convolution of functions:
\[(\psi\circ \phi)(x) = \sum_{\fS_\mu y \in O_\mu} \psi(xy^{-1})\phi(y).\]

Thus we have the following well-known fact:
\begin{proposition} The Schur algebra $S_\kk(n,d)$ is Morita equivalent to the convolution algebra
\[\bigoplus_{\lambda,\mu\in \Part{n,[d]}} \kk[\fS_\mu\backslash \fS_d / \fS_\lambda], \]
where $\Part{n,[d]}$ denotes the set of partitions of $[d]$ with at most $n$ non-zero parts.
\end{proposition}

To carry out the analogous reasoning for the Hilbert-Schur algebra we need to recall some tensor category preliminaries.  We refer the reader to \cite{EGNO} for details about tensor categories.

For a finite group $G$ with multiplication $m\colon G\times G \to G$, the category $\Sh(G)$ of finite-dimensional sheaves of $\C$-vector spaces on $G$ is a finite tensor category, with tensor operation 
\[V\otimes W = m_*(V \boxtimes W), \;\; V,W \in \Sh(G).\]
(In \cite{EGNO}, this category is denoted $\Vec_G$.)  If $Z$ is a finite set with a right $G$-action
given by $a\colon Z\times G \to Z$, then $\Sh(Z)$ is a (right) module category over $\Sh(G)$ via 
\[(A,V) \mapsto a_*(A \boxtimes V),\;\; A \in \Sh(Z),\; V \in \Sh(G).\]
If the action is transitive, so that $Z \cong H\backslash G$ as right $G$-sets for a subgroup $H \subset G$,
then this module category can be described as follows.  The constant sheaf $\cO_H = \underline{\C}_H$ supported on $H$ is naturally a ring object in $\Sh(G)$; this amounts to considering the group algebra $\C[H]$ as a $G$-graded object.  Then $\Sh(H\backslash G)$ is naturally isomorphic as a module category over $\Sh(G)$ to $\Mod_{\Sh(G)}(\cO_H)$, the category of (left) modules over the ring object $\cO_H$.

Recall from the introduction that we defined the Hilbert-Schur algebra as
\[HS_\kk(n,d) := K(\End_{\Sh(\fS_d)}(\Sh([n]^d)))\otimes \kk,\]
where $\End_{\Sh(\fS_d)}(\Sh([n]^d))$ is the category of (additive) endofunctors of $\Sh([n]^d)$ as a module category over $\Sh(\fS_d)$. 

As in the vector space case, $\Sh([n]^d)$ decomposes as a module category over $\Sh(\fS_d)$ into permutation modules indexed by $\fS_d$-orbits of $[n]^d$:
\[\Sh([n]^d) = \bigoplus_{O\in [n]^d/\fS_d} \Sh(O).\]
Each orbit $O$ is isomorphic as a right $\fS_d$-set to a quotient $O_\lambda := \fS_{\lambda}\backslash \fS_d$ for a set partition $\lambda \in \Part{[d]}$ with at most $n$ parts.

For subgroups $H,K \subset G$, \cite[Proposition 7.11.1]{EGNO} gives an isomorphism of categories
\begin{align*}
\Fun_{\Sh(G)}(\Sh(H\backslash G), \Sh(K\backslash G)) & \simeq \Fun_{\Sh(G)}(\Mod_{\Sh(G)}(\cO_H), \Mod_{\Sh(G)}(\cO_K)) \\
& \simeq \Bimod_{\Sh(G)}(\cO_K,\cO_H)
\end{align*}
between right exact module functors $\Sh(H \backslash G) \to \Sh(K\backslash G)$ and $\cO_K$-$\cO_H$ bimodules.  These in turn are naturally isomorphic to $\Sh([K\backslash G/H])$.

Under this equivalence, composition of module functors is sent to tensor product of bimodules: for 
$V \in \Bimod_{\Sh(G)}(\cO_H,\cO_K)$ and $W \in \Bimod_{\Sh(G)}(\cO_K,\cO_H)$, the tensor $V \otimes_{\cO_K} W$ is obtained from the tensor $V \otimes W$ taken in $\Sh(G)$ by taking coinvariants for the action of $\C[K]$.  

Specializing to the case $G=\fS_d$, $H=\fS_\lambda$ and $K=\fS_\mu$, we have
\[ \Fun_{\Sh(\fS_d)}(\Sh(\fS_\lambda\backslash \fS_d),\Sh(\fS_\mu\backslash\fS_d)) \simeq \Sh([\fS_\mu \backslash \fS_d/\fS_\lambda]).\]

We conclude:
\begin{proposition}
There is a Morita equivalence
\[HS_\kk(n,d) \stackrel{\mathrm{Morita}}{\simeq}  \bigoplus_{\lambda,\mu\in \Part{n,[d]}} K(\Sh([\fS_\mu \backslash \fS_d/\fS_\lambda]))\otimes \kk, \]
where the multiplication on the right is defined as in Section~\ref{subsec-Cn}.
\end{proposition}

Note that in Section~\ref{subsec-Cn}, composition is defined by taking invariants, not coinvariants.  But the category of $\C[K]$-modules is semisimple, so the invariant and coinvariant functors are naturally isomorphic.

When $n\geq d$, the right hand side is simply $R_{[d]}\otimes \kk$. Thus, letting $n=d$:
\begin{corollary} The Hilbert-Schur algebra $HS_\kk(n,n)$ is Morita equivalent to the algebra $R_{[n]} \otimes \kk$.
\end{corollary}

%%%%%%%%%%%%%%%%%%%
\section{Perverse sheaves on $S^nX$}\label{sec:perv}
%%%%%%%%%%%%%%%%%%%

Let $X=\C^2$.  In this section we consider the topology of $S^n X = X^n/\fS_n$, the $n$-fold symmetric product of the plane.  More generally, for any set partition $\lambda \in \Part{[n]}$, we let $$S^\lambda X = X^n/\fS_\lambda \cong S^{\hat\lambda_1}X \times S^{\hat\lambda_2}X \times \dots\times S^{\hat\lambda_r}X,$$
where $\hat\lambda = (\hat\lambda_1,\hat\lambda_2, \ldots,\hat\lambda_r)$
is the associated integer partition.  Up to isomorphism it only depends on $\hat\lambda$, but just as with our $2$-category $\cC_{[n]}$ it will be convenient to make this definition for any set partition.

Let $\kk$ be a field.  If $Y$ is a complex algebraic variety, we denote by $D^b(Y,\kk)$ the constructible derived category of sheaves of $\kk$-vector spaces on $Y$.  For a fixed stratification $\cT$ of $Y$, let $\Perv_\cT(Y,\kk) \subset D^b(Y,\kk)$ be the abelian category of middle perversity perverse sheaves of $\kk$-vector spaces on $Y$ constructible with respect to $\cT$.  In this paper we are interested in the category $\Perv_\cS(S^n X,\kk)$ of perverse sheaves for a natural stratification $\cS$ which we now define.

%%%%%
\subsection{Stratification} 
%%%%%
For any partition $\lambda = (\lambda_1,\lambda_2, \ldots,\lambda_r)$ of $n = |\lambda| = \sum_i \lambda_i$, let $S^n_\lambda X \subset S^n X$ be the locally closed smooth subvariety of those configurations that can be expressed as $\sum_i \lambda_i p_i$ for distinct points $p_i \in X$.

This defines a stratification $\cS = \cS^n$ of $S^nX$: $$S^n X = \bigsqcup_{\lambda~ \vdash ~n} S^n_\lambda X.$$
Note that $\dim S^n_\lambda X = 2\ell(\lambda)$, where $\ell(\lambda)$ is the length or number of parts of $\lambda$.  We have $S^n_\lambda X \subset \overline{S^n_\mu X}$ if and only if $\mu$ is a refinement of $\lambda$.

More generally, we have a stratification $\cS^\lambda$ of $S^\lambda X$ obtained as the product of the stratification $\cS^{\hat\lambda_i}$ on each factor $S^{\hat\lambda_i} X$.  Note that when points from different factors coincide, it does not create a new stratum.  For example, if $\lambda$ is the minimal partition, so $\fS_\lambda = 1$ and $S^\lambda X = X^n$, then there is only one stratum.  In a slight abuse of notation, we will sometimes denote this stratification of $S^\lambda X$ simply by $\cS$.

To describe these strata more precisely, for any partition $\mu \le \lambda$ we can define the stratum $S^\lambda_\mu X$ to be the image in $S^\lambda X$ of the locus of points in $X^n$ whose stabilizer in $\fS_\lambda$ is $\fS_\mu$. Note that this stratum only depends on $\mu$ up to the action of $\fS_\lambda$, or in other words it only depends on the integer partitions that the set partition $\mu$ induces on each $\sim_\lambda$-equivalence class.

%%%%%
\subsection{Stratification preliminaries} \label{sec:stratification preliminaries}
%%%%%

In this section we collect some definitions and results about stratifications and constructible sheaves that we will need.  Since we will sometimes consider open subsets of $S^nX$ in the analytic topology, we work with complex analytic varieties rather than algebraic varieties.  The notion of stratification we will work with is the following.

\begin{definition}\label{defn:stratification}
	A stratification of an analytic variety $Y$ is a decomposition
	\[Y = \coprod_{S \in \cS} S\]
	into disjoint locally closed smooth subvarieties, such that for any $S \in \cS$ with $\dim S > 0$, we have
	\begin{enumerate}
		\item  the closure $\overline{S}$ is a union of elements of $\cS$, and
		\item for any point $y \in S$ there exists a neighborhood $U$ of $y$ in $Y$, an analytic variety $N$ (a normal slice) endowed with a stratification $$\cS_N = \{N_T \mid T \in \cS \;\mbox{and}\; S \subset \overline{T}\}$$
		 and an isomorphism
		$$\eta\colon   N \times (U\cap S) \to U$$ such that if $T \in \cS$ and $S\subset \overline{T}$, we have $\eta^{-1}(T) = N_T \times (U\cap S)$, the stratum $N_{S}$ is a single point $\{o\}$, and $\eta(o,s) = s$ for all $s \in U\cap S$. 
	\end{enumerate}
\end{definition}
\begin{remark}
In the condition (2), note that $\dim N < \dim Y$, so this definition is recursive rather than circular.  It is common to also include a condition that the stratification on $N$ is topologically a stratified cone with apex $\{o\}$, as well as compatibility relations between the maps $\eta$ for different strata.  But we will not need these more refined notions of stratifications.	
\end{remark}

It is clear that our decomposition $\cS$ of $S^nX$ satisfies condition (1); we explain why condition (2) holds in Proposition  \ref{prop:cS is a stratification} below.

A complex $\cF\in D^b(Y,\kk)$ is called \emph{$\cS$-constructible} if for every stratum $S \in \cS$, the cohomology sheaves of $\cF|_S$ are $\kk$-local systems of finite rank.  Let $D^b_\cS(Y,\kk)$ denote the full triangulated subcategory of $\cS$-constructible complexes.  The conditions on our stratification imply that the functors $j_*, j_!, j^*, j^!$ preserve $\cS$-constructibility, where $j$ is the inclusion of one locally closed union of strata into another.  

Next, we give several technical results about constructible sheaves that we will need.  For all of them we will take a variety $Y$ with a stratification $\cS$.  For our first result, let $Z$ be a locally contractible space, and give $Y \times Z$ the decomposition $\cS\times Z := \{S \times Z \mid S \in \cS\}$. 
Note that this is not a stratification in the sense of Definition \ref{defn:stratification} unless $Z$ is smooth, but we can define $(\cS\times Z)$-constructibility in the same way, and it is still true that the functors $j_*, j_!, j^*, j^!$ preserve $(\cS\times Z)$-constructibility for any inclusion $j$ of locally closed unions of elements of $\cS\times Z$.

For a point $x \in Z$, let $\imath_x\colon \{x\} \to Z$ be the inclusion.
If $W$ is another locally contractible space and $g\colon W \to Z$ is a continuous map, we will denote the product map $\id_{Y} \times g \colon  Y \times W \to Y \times Z$ by $\hat g$.  
Then for each $x \in Z$ we have the restriction functor
\[(\hat\imath_x)^*\colon D^b_{\cS\times Z}(Y\times Z, \kk) \to D^b_\cS(Y,\kk).\]
The following result is a more precise statement of \cite[Lemma 9.5]{BMHyperRingel}, and has the same proof.  

\begin{lemma}\label{lem:stratified homotopy}
	For any continuous path $\gamma\colon [0,1] \to Z$, there is a natural isomorphism of functors
	\[(\hat\imath_{\gamma(0)})^* \simeq (\hat\imath_{\gamma(1)})^*,\]
	which only depends on the class of $\gamma$ under endpoint-preserving homotopy.  In particular, if $Z$ is path-connected, all the functors $(\hat\imath_x)^*$ are isomorphic, and if $Z$ is simply connected they are all canonically isomorphic.
\end{lemma}

For our next result, we restrict to taking a product with a smooth variety, and we restrict from the whole derived category to perverse sheaves.

\begin{lemma}\label{lem:stratified pullback}
	Let $M$ be a smooth variety of dimension $d$, and let $\pi\colon  Y\times M \to Y$ be the projection.  If $M$ is simply-connected, then
	\[\pi^\dagger\colon \Perv_\cS(Y,\kk) \to \Perv_{\cS\times M}(Y \times M,\kk)\]
	is an equivalence of categories, where we put $\pi^\dagger\cF = \pi^*\cF[2d] \cong \cF \boxtimes \underline{\kk}_M[2d]$.
\end{lemma}
\begin{proof}
	The fact that $\pi^\dagger$ preserves perversity is standard; see \cite[Proposition 3.6.1]{AcharBook}, for example.  It clearly sends $\cS$-constructible objects to $\cS \times M$-constructible objects.  \cite[Theorem 3.6.6]{AcharBook} shows that $\pi^\dagger$ is fully faithful on perverse sheaves, so we only need to show that it is essentially surjective.  To see this, we first check that $\pi^\dagger$ is surjective on isomorphism classes of simple objects. A simple object of 
	$\Perv_{\cS\times M}(Y \times M,\kk)$ will be an IC complex $\IC(\cL)$ associated to a $\kk$-local system $\cL$ on a stratum $S \times M \in \cS \times M$ (we use the convention that all local systems and IC complexes are shifted to be perverse).  Since $M$ is simply connected, we have 
	$\cL \cong  \cL_0\boxtimes \underline{\kk}_M[2d]$ for some $\kk$-local system on $S$, and it follows that $\IC(\cL) \cong \pi^\dagger\IC(\cL_0)$.  To complete the proof of essential surjectivity it is enough to show that $\pi^\dagger$ induces an isomorphism
	\[\Ext^1_{\Perv(Y,\kk)}(\cF,\cG) \to \Ext^1_{\Perv(M\times Y,\kk)}(\pi^\dagger\cF,\pi^\dagger\cG)\] 
	for any $\cF,\cG\in \Perv(Y,\kk)$.  But we have
	\begin{align*}
	\Ext^1_{\Perv(M\times Y,\kk)}(\pi^\dagger\cF,\pi^\dagger\cG) & \cong \mathbb{H}^1(\RHom(\pi^\dagger\cF,\pi^\dagger\cG)) \\
	& \cong \mathbb{H}^1(\pi^*\RHom(\cF,\cG)),	
	\end{align*}
	since $\RHom$ commutes with smooth pullbacks (\cite[Principle 2.2.11]{AcharBook}).
	Then we have
	\[\mathbb{H}^\bullet(\pi^*\RHom(\cF,\cG))\cong \mathbb{H}^\bullet(\RHom(\cF,\cG))\otimes_\kk H^\bullet(M,\kk),\]
	and the result follows from the facts that $H^0(M,\kk) \cong \kk$, $H^1(M,\kk)=0$, and $\mathbb{H}^j(\RHom(\cF,\cG))=R\Hom^j(\cF,\cG)=0$ for $j < 0$.
\end{proof}
 
%%%%%
\subsection{Symmetrization and factorization maps} 
%%%%%

Let $\sigma_\lambda \colon S^\lambda X \to S^n X$ be the symmetrization (or sum) map.  
More generally, if $\lambda \le \mu$ we have an induced map $\sigma_\lambda^\mu\colon S^\lambda X \to S^\mu X$.  It is a proper finite map, and for any stratum $S^\lambda_\nu X$ of $S^\lambda X$ and any stratum $S^\mu_\xi X$ of $S^\mu X$, the restriction
\[S^\lambda_\nu X \cap (\sigma_\lambda)^{-1}(S^\mu_\xi X)\]
is a finite covering map.  As a result, the pushforward $(\sigma^\mu_\lambda)_*$ preserves perversity and $\cS$-constructibility.

On the other hand, strata of $S^\lambda X$ are not sent to strata of $S^\mu X$ by $\sigma_\lambda^\mu$, because points indexed by different $\sim_\lambda$ equivalence classes can collide, landing in a different stratum of $S^\mu X$.  As a result, the pullback $(\sigma^\mu_\lambda)^*$ does \emph{not} preserve $\cS$-constructibility.   To get a pullback functor on $\cS$-constructible perverse sheaves, we will view $S^\lambda X$ as a slice in $S^\mu X$ by defining injective maps $S^\lambda X \to S^\mu X$ induced by including the different factors of $S^\lambda X$ into separate ``bubbles" in $X$. 

Let $B = B_1(0)$ be the open unit ball in $X$ with respect to the standard metric, and fix once and for all a diffeomorphism $c\colon X \to B$.  We let $\Emb(X)$ denote the set of all maps 
$X \to X$ of the form $\phi = \phi_{p,r} = c^{-1}j_{p,r}c$, where $p \in B$, $0 < r \le 1 - |p|$ and 
\[j_{p,r}(x) = p + rx, \;x \in B.\]
Then $\Emb(X)$ is naturally a manifold with boundary diffeomorphic to $B \times (0,1]$.  It is a topological category: it contains the identity map $\phi_{0,1}$, and composition $(\phi, \phi') \mapsto\phi \circ \phi'$ is a continuous map $\Emb(X)\times \Emb(X) \to \Emb(X)$.

\begin{definition}
Let $\lambda \le \mu$ be partitions, and consider a map $\psi\colon X^n \to X^n$ of the form
\[\psi(p_1,\dots,p_n) = (\phi_1(p_1), \phi_2(p_2),\dots,\phi_n(p_n)),\]
where $\phi_1,\dots,\phi_n\in \Emb(X)$.  
We say that $\psi$ is 
\emph{$\lambda$-adapted} if whenever $i \stackrel{\lambda}{\sim} j$, we have $\phi_i = \phi_j$.  We say $\psi$ is \emph{$(\lambda,\mu)$-separated} if it is $\lambda$-adapted and whenever $i \not\stackrel{\lambda}{\sim} j$ but $i \stackrel{\mu}{\sim} j$, then the images $\phi_i(X)$ and $\phi_j(X)$ have disjoint closures.
If a map is $(\lambda,(n))$-separated, we say simply that it is $\lambda$-separated.
\end{definition}

It follows that the set of $\lambda$-adapted maps $X^n \to X^n$ is parametrized by 
\[\Emb(X,\lambda) := (\Emb(X)^n)^{\fS_\lambda} \cong \Emb(X)^{\ell(\lambda)}\] and the set of $(\lambda,\mu)$-separated maps is parametrized by an open subset
$\Emb_\mu(X,\lambda) \subset \Emb(X,\lambda)$. 

\begin{lemma}\label{lem:Emb is simply connected}
	For any partitions $\lambda \le \mu$, the set $\Emb_\mu(X,\lambda)$ is simply connected.
\end{lemma}
\begin{proof}
	If $\lambda$ is not the minimal partition, one can take the quotient of the indexing set $\{1,\dots,n\}$ by the equivalence relation $\sim_\lambda$, and let $\lambda',\mu'$ be the induced partitions on the quotient set.  Then $\lambda'$ is the minimal partition and $\Emb_\mu(X,\lambda)$ is homeomorphic to $\Emb_{\mu'}(X,\lambda')$, so we can assume without loss of generality that $\lambda$ is the minimal partition $(1^n)$.  On the other hand, if $\mu$ is not maximal, let $\mu_1,\dots,\mu_{r}$ be the sizes of the parts of $\mu$.  Then we have a homeomorphism
	\[\Emb_\mu(X,\lambda) \cong \Emb_{(\mu_1)}(X,(1^{\mu_1})) \times \dots \times \Emb_{(\mu_r)}(X,(1^{\mu_r})),\]
	so we can also assume without loss of generality that $\mu$ is the maximal partition $(n)$.
	Then $\Emb_\mu(X,\lambda)$ is given by tuples $(p_1,\dots,p_n)\in B^n$ and $(r_1,\dots,r_n)\in (0,1)^n$
	satisfying 
	\[0 < r_i \le 1 - |p_i|, \;\; r_i + r_j < |p_i - p_j|, \;\; i \ne j.\]
  Forgetting the numbers $r_i$ defines a map
	\[\pi\colon \Emb_\mu(X,\lambda) \to \Conf_n(B)\]
	to the configuration space of ordered $n$-tuples of distinct points in $B$.
	This map has a section $s\colon \Conf_n(B) \to \Emb_\mu(X,\lambda)$ given by $s(p_1,\dots,p_n) = ((p_1,\dots, p_n), (r,\dots,r))$, where
	\[r = \min(1 - |p_i|, |p_i - p_j|)/3.\]
	The maps $\pi$ and $s$ are homotopy inverses, since the fibers of $\pi$ are convex sets.  A configuration space of a simply connected manifold of dimension at least three is simply connected (see \cite[Corollary 2.3.4]{knudsen2018configuration}, for example), so it follows that  
    $\Emb_\mu(X,\lambda)$ is simply connected.
\end{proof}
 
For any partition $\lambda$, a $\lambda$-adapted map $\psi$ induces a map $S^\lambda X \to S^\lambda X$, which we denote by $\psi_\lambda$ or just $\psi$ if no confusion will occur.  The set of such maps is parametrized by $\Emb(X,\lambda)$.

\begin{lemma}\label{lem:psi isom id}
	If $\psi\colon S^\lambda X \to S^\lambda X$ is induced from a $\lambda$-adapted map, then the pullback 
	\[\psi^*\colon D^b_\cS(S^\lambda X, \kk) \to  D^b_\cS(S^\lambda X, \kk)\]
	is canonically isomorphic to the identity functor.
\end{lemma}

\begin{proof}
	Consider the ``universal $\lambda$-adapted map"
	\[\Psi\colon S^\lambda X \times \Emb(X,\lambda) \to S^\lambda X, \;\;\; (x,\psi) \mapsto \psi_\lambda(x).\]
	If $\cF$ is $\cS$-constructible, then $\Psi^*\cF$ is constructible with respect to the stratification $\{S \times  \Emb(X,\lambda)\}_{S \in \cS}$.  
	Since $\Emb(X,\lambda)$ is simply connected, Lemma \ref{lem:stratified homotopy} implies that for any two $\lambda$-adapted maps $\psi,\psi'$ there is a natural isomorphism $(\psi_\lambda)^* \simeq (\psi'_\lambda)^*$ on $D^b_\cS(S^\lambda,\kk)$.  Since the identity $(\phi_{0,1},\dots, \phi_{0,1})$ is $\lambda$-adapted, this proves the result.
\end{proof}

Now take a partition $\mu \ge \lambda$, and consider the composition \[\omega^\mu_\lambda = \sigma^\mu_\lambda \circ \psi_\lambda\colon S^\lambda X\to S^\mu X\] 
of a $(\lambda,\mu)$-separated map $\psi_\lambda\colon S^\lambda X \to S^\lambda X$ with the symmetrization map $\sigma^\mu_\lambda: S^\lambda X \to S^\mu X$.  We will call a map such as $\omega^\mu_\lambda$
a \emph{factorization map}.  It is an open embedding, and it is compatible with the stratifications $\cS$ on the source and target, in the sense that $\omega_\lambda^\mu(S^\lambda_\nu X) \subset S^\mu_\nu X$ for any partition $\nu$ with $\nu \le \lambda \le \mu$.  (Note that the stratum $S^\lambda_\nu X$ is determined by $\nu$ up to the action of $\fS_\lambda$, while $S^\mu_\nu X$ is determined only up to the action of the smaller group $\fS_\mu$, so more than one stratum of $S^\lambda X$ can be sent into the same stratum of $S^\mu X$.)

In particular we get a $t$-exact functor:
\[(\omega^\mu_\lambda)^*\colon D^b_\cS(S^\mu X, \kk) \to D^b_{\cS}(S^\lambda X, \kk).\]

\begin{remark}\label{rmk:pullback stratification}
	Note that Lemma \ref{lem:psi isom id} does not imply that the pullback $(\omega^\mu_\lambda)^*$ is isomorphic to $(\sigma^\mu_\lambda)^*$ on $\Perv_\cS(S^\mu X, \kk)$, since $(\sigma^\mu_\lambda)^*$ does not preserve $\cS$-constructibility.  Instead one can think of $(\omega^\mu_\lambda)^*$ as giving the ``generic behavior" of $(\sigma^\mu_\lambda)^*$ on a set where points from different parts of $S^\lambda X$ do not collide when projected into $S^\mu X$.  
	
	Put another way, the map $\omega^\mu_\lambda$ identifies $S^\lambda X$ with an open neighborhood of a point of the stratum $S^\mu_\lambda X \subset S^\mu X$. This fact can be used to give a proof that our stratification $\cS$ satisfies the second condition of Definition \ref{defn:stratification}.

%	Although there is no group action on $S^nX$ such that the strata $S^n_\lambda X$ are orbits, any point of the stratum $S^\lambda X$ is in the image of a factorization map $\omega_\lambda\colon S^\lambda X\to S^nX$, and $(\omega_\lambda)^{-1}(S^n_\lambda X) = S^\lambda_\lambda X \cong X^{\ell(\lambda)}$ has a structure of an additive group which acts on all of $S^\lambda X$.  The existence of a normal slice $N$ and local trivialization $\eta$ follow easily from this.
\end{remark}

\begin{proposition}\label{prop:cS is a stratification}
	The decomposition $\cS$ of $S^\mu X$ is a stratification in the sense of Definition \ref{defn:stratification}.
\end{proposition}
\begin{proof}
	Take a partition $\lambda \le \mu$, and let $p$ be any point of the stratum $S = S^\mu_\lambda X$.  There is a factorization map $\omega^\mu_\lambda\colon S^\lambda X \to S^\mu X$ with open image $U \subset S^\mu X$ such that $p \in U$.
%	 which contains $p$ in its image, which is an open set $U \subset S^\mu X$.  
	
	We view $G = (X^n)^{\fS_\lambda} \cong X^{\ell(\lambda)}$ as a group under addition.  Its action on $X^n$ descends to an action on $S^\lambda X$ which preserves the strata.  The action is free, and we get an isomorphism
	$\theta\colon G\times N \stackrel{\sim}\longrightarrow S^\lambda X$, where $N$ is the quotient $(S^\lambda X)/G$ (alternatively, $N$ can be described as the subvariety of $S^\lambda X$ consisting of points for which the parts in each $\sim_\lambda$-equivalence class sum to zero).  Let $\cS_N'$ be the decomposition of $N$ by the quotients $(S^\lambda_\nu)/G$ of strata of $S^\lambda X$, indexed by $\fS_\lambda$-orbits of partitions $\nu \le \lambda$.
	
	The composition $\eta = \omega^\mu_\lambda \circ \theta$ then gives the local isomorphism required by property (2) of Definition \ref{defn:stratification}.  The action identifies the smallest stratum $S^\lambda_\lambda X$ of $S^\lambda X$ with $G$, and so the corresponding stratum of $N$ is a single point.  Then the map $\omega^\mu_\lambda$ further identifies $G$ with $U \cap S$.  For any $\nu\le \lambda$, the inverse image $(\omega^\mu_\lambda)^{-1}(S^\mu_\nu X)$ is the union of $S^\lambda_{w\nu} X$, over a set of representatives $w$ of cosets in $\fS_\mu/\fS_\lambda$.  Taking the quotient of these unions by the $G$-action gives a stratification $\fS_N$ of $N$ satisfying the property (2).  It also follows that $p$ is in the closure $\overline{S^\mu_\nu X}$ if and only if $\nu \le \lambda$, which implies the property (1). 
\end{proof}

\begin{remark}
Note that for each factorization map $\omega^\mu_\lambda: S^\lambda X \to S^\mu X$ there is a \textit{unique} $(\lambda,\mu)$-separated map $\psi_\lambda: S^\lambda X \to S^\lambda X$ such that $\omega^\mu_\lambda = \sigma^\mu_\lambda \psi_\lambda$.  In other words, the map from $\Emb_\mu(X,\lambda)$ to the set of factorization maps $S^\lambda X \to S^\mu X$ is injective (and thus a bijection). 
\end{remark}

\begin{lemma}\label{lem:omega pullback}
	Let $\psi_1, \psi_2\colon S^\lambda X \to S^\lambda X$ be two $(\lambda,\mu)$-separated maps, and let $\omega_1 = \sigma_\lambda^\mu \circ \psi_1$, $\omega_2 = \sigma_\lambda^\mu \circ \psi_2$ be the corresponding factorization maps.  Then the functors
\[(\omega_1)^*,(\omega_2)^* \colon D^b_\cS(S^\mu X, \kk) \to D^b_{\cS}(S^\lambda X, \kk)\] 
are canonically isomorphic.
\end{lemma}

\begin{proof}
	As in the proof of Lemma~\ref{lem:psi isom id}, we have a ``universal factorization map"
	\[\Omega\colon S^\lambda X \times \Emb_\mu(X,\lambda) \to S^\mu X, \;\;\; (x, \psi_\lambda) \mapsto \sigma^\mu_\lambda(\psi_\lambda(x)).\]
	Take $\cF\in D^b_\cS(S^\mu X, \kk)$.  Although $\cG = (\sigma_\lambda^\mu)^*\cF$ is not $\cS$-constructible, if $W = \psi_\lambda(S^\lambda X)$ is the image of a $\lambda$-separated map, then $\cG|_W$ is $\cS\cap W$-constructible.  Thus $\Omega^*\cF$ is constructible for the stratification
	$\{S\times \Emb_\mu(X,\lambda)\}_{S \in \cS_\lambda}$.  Since $\Emb_\mu(X,\lambda)$ is path-connected and simply-connected, the result follows again by Lemma \ref{lem:stratified homotopy}.
\end{proof}

%%%%%
\subsection{Biadjunction of $\sigma_*$ and $\omega^*$} 
%%%%%

Fix partitions $\lambda \le \mu \in \Part{[n]}$.  To keep the notation simple, for this section we will let 
\[\sigma \colon S^\lambda X\to S^\mu X, ~\psi: S^\lambda X \to S^\lambda X \quad \mathrm{and} \quad \omega = \sigma\circ \psi\colon S^\lambda X\to S^\mu X\]
 denote the symmetrization map, a choice of $(\lambda,\mu)$-separated map and the corresponding factorization map, respectively.  

\begin{theorem}\label{thm-biadjoint}
The functor $\sigma_* \colon \Perv_\cS (S^\lambda X, \kk) \to \Perv_\cS (S^\mu X,\kk)$ is biadjoint to the functor $\omega^* \colon  \Perv_\cS (S^\mu X,\kk) \to \Perv_\cS (S^\lambda X, \kk)$.  For $\cF \in \Perv_\cS(S^\mu X, \kk)$ and $\cG \in \Perv_\cS(S^\lambda X,\kk)$, the adjunction $(\omega^*, \sigma_*)$ is given by
\[\Hom(\cF,\sigma_*\cG) \stackrel{\sim}{\longrightarrow} \Hom(\sigma^*\cF, \cG) \stackrel{\psi^*}{\longrightarrow} \Hom(\psi^*\sigma^* \cF, \psi^* \cG) \stackrel{\sim}{\longrightarrow} \Hom(\omega^*\cF, \cG),\]
where the first arrow is the standard adjunction $(\sigma^*, \sigma_*)$, the second is composition with $\psi^*$, and the third comes from $\psi^*\sigma^* = \omega^*$ and Lemma \ref{lem:psi isom id}.
\end{theorem}

We can rewrite the formula for the adjunction to give a formula for the counit:
\begin{corollary}\label{cor:counit formula}
	The counit of the adjunction $(\omega^*, \sigma_*)$ is the composition
	\[\omega^*\sigma_* = \psi^*\sigma^*\sigma_* \to \psi^* \to 1_{\Perv_\cS(S^\lambda X,\kk)}\]
	of $\psi^*$ applied to the counit of $(\sigma^*,\sigma_*)$ and the isomorphism of Lemma \ref{lem:psi isom id}.
\end{corollary}

As $\sigma$ is proper and $\omega$ is an open inclusion, we have $\sigma_*= \sigma_!$ and $\omega^*=\omega^!$.  Thus by duality it suffices to show that $\omega^*$ is left adjoint to $\sigma_*$ as functors between $\Perv_\cS(S^\lambda X, \kk)$ and $\Perv_\cS(S^\mu X,\kk)$.

There is of course an adjunction $(\sigma^*,\sigma_*)$ of functors between the triangulated categories $D^b(S^\lambda X,\kk)$ and $D^b(S^\mu X,\kk)$, but as noted above $\sigma^*$ does not preserve either perversity or $\cS$-constructibility in general.  
To prove Theorem~\ref{thm-biadjoint} it suffices to show the following proposition.

\begin{proposition}\label{prop-hominto}
For any  $\cF \in \Perv_\cS (S^\mu X,\kk)$ and $\cG \in \Perv_\cS (S^\lambda X,\kk)$ there is a natural isomorphism
\[ \Hom_{D^b(S^\lambda X)}(\sigma^*\cF,\cG) \cong \Hom_{\Perv_{\cS}(S^\lambda X)}(\omega^* \cF, \cG)\]
which is induced by applying $\psi^*$ and using the isomorphism of Lemma \ref{lem:psi isom id}.
\end{proposition}

The proof of Proposition \ref{prop-hominto} will be based on the following results. Let $j\colon U \hookrightarrow S^\lambda X$ denote the inclusion of the open locus $U$ of points $(x_1,\dots,x_r) \in S^\lambda X = S^{|\lambda_1|} X \times \dots \times S^{|\lambda_r|} X$ such that $x_i \cap x_j = \emptyset$ for all $i\neq j$.  

\begin{lemma}\label{lem-openset}
	Let $Z = S^\lambda X \setminus U$.  Then
	\begin{enumerate}
		\item For any stratum $S^\lambda_\nu X \subset S^\lambda X$, the intersection $S^\lambda_\nu X \cap Z$ has codimension at least two in $S^\lambda_\nu X$.
		\item For any $\nu \le \lambda$, we have
		$\sigma(S^\lambda_\nu X \cap U) \subset S^\mu_\nu X.$
		\item The restriction $\sigma|_U \colon U \to S^\mu X$ is a local homeomorphism.
	\end{enumerate}
	Furthermore, for any $\cF\in \Perv_\cS(S^\mu X, \kk)$, we have 
	$j^*\sigma^*\cF \in \Perv_{\cS\cap U}(U,\kk)$.
\end{lemma}

\begin{proposition}\label{prop-pervext} 
Let $Y$ be a variety endowed with a stratification $\cS$, and let $U$ be an open subset of $Y$ with complement $Z$.  If for every stratum $S \in \cS$, $S\cap Z$ has codimension at least two in $S$, then for any perverse sheaf
$\cF \in \Perv_{\cS\cap U}(U,\kk)$, the perverse extensions
\[ {}^p\! j_! \cF,\; j_{!*} \cF,\; {}^p\! j_* \cF \in \Perv(Y,\kk)\]
along the inclusion $j\colon U \to Y$
are $\cS$-constructible, and the natural maps ${}^p\! j_! \cF\to j_{!*} \cF\to {}^p\! j_* \cF$ are isomorphisms.  Moreover, for any $\cG \in \Perv_\cS(Y, \kk)$, the adjunction map $\cG \to {}^p\!j_*j^*\cG$ is an isomorphism.  
\end{proposition}

\begin{proof}[Proof of Proposition \ref{prop-hominto} assuming Lemma \ref{lem-openset} and Proposition  \ref{prop-pervext}]
 Given $\cF \in \Perv_\cS (S^\mu X,\kk)$ and $\cG \in \Perv_\cS (S^\lambda X,\kk)$, Proposition \ref{prop-pervext} gives an isomorphism $\cG \cong {}^p\!j_*j^*\cG = {}^p\!j_*j^!\cG = {}^p\tau_{\le 0}j_*j^!\cG$, using the fact that $j_*j^!\cG$ is in $D^{\ge 0}(S^\lambda X,\kk)$.
Also note that $\sigma^*\cF$ is in ${}^p D^{\le 0}(S^\lambda X,\kk)$, because $\sigma$ is a finite map.  Then we have natural isomorphisms
\begin{equation}\label{eqn:j adjunction}
\begin{split}
	\Hom(\sigma^*\cF,\cG)  & \cong \Hom(\sigma^*\cF, {}^p\tau_{\le 0}j_*j^!\cG) \\ & \cong \Hom(\sigma^*\cF,j_*j^!\cG) \\ & \cong \Hom({}^p\!j_!j^*\sigma^*\cF, \cG), 
\end{split}	
\end{equation}
where the second isomorphism is the adjunction between ${}^p\tau_{\le 0}$ and the inclusion ${}^p D^{\le 0}(S^\lambda X,\kk) \subset D^b(S^\lambda X,\kk)$, and the third isomorphism is from the adjunctions $(j^*, j_*)$ and $({}^p\!j_!, j^!)$. (Since the latter adjunction holds on categories of perverse sheaves, we use the fact that $j^*\sigma^*\cF$ is perverse, by Lemma \ref{lem-openset}.)

Another application of Proposition \ref{prop-pervext} shows that ${}^p\!j_!j^*\sigma^*\cF$ is $\cS$-constructible, so by Lemma \ref{lem:psi isom id} we have a natural isomorphism
\begin{equation}\label{eqn:independence}
{}^p\!j_!j^*\sigma^*\cF \cong \psi^*\,{}^p\!j_!j^*\sigma^*\cF \cong \psi^*\sigma^*\cF = \omega^*\cF,
\end{equation}
where the middle isomorphism holds because the image of $\psi$ is contained in $U$.

This establishes the isomorphism of Proposition \ref{prop-hominto}. The fact that the isomorphism is given by applying $\psi^*$ and using Lemma \ref{lem:psi isom id} follows because if we apply $\psi^*$ to each term of \eqref{eqn:j adjunction}, each term is naturally isomorphic to $\Hom(\omega^*\cF, \cG)$, and the isomorphisms become the identity map.
\end{proof}

Before proving Lemma \ref{lem-openset} and Proposition~\ref{prop-pervext} we pause to record the following consequence of the previous proof:
\begin{corollary}\label{cor:unit}
There is a natural isomorphism of functors ${}^p\!j_!j^*\sigma^* \cong \omega^*$, and the unit for the adjunction $(\omega^*,\sigma_*)$ described in Theorem~\ref{thm-biadjoint} can be expressed as
\[ 1_{\Perv_\cS(S^n X)} \to \sigma_* \sigma^* \to \sigma_* {}^p\! j_* j^*  \sigma^* \cong \sigma_*\omega^*,\]
where the first two maps are the units for the adjunctions $(\sigma^*,\sigma_*)$ and $(j^*,\,^{p}\!j_*)$.
\end{corollary}

\begin{proof}[Proof of Lemma \ref{lem-openset}]
	Let $\tau \colon X^n \to S^\lambda X$ be the quotient map.  We can express membership in the open set $U$ in terms of stabilizers of the $\fS_n$-action: for a point $p\in X^n$, we have
	\[\tau(p) \in U \iff (\fS_n)_p \subset \fS_\lambda,\]
	which implies that $(\fS_\lambda)_p = (\fS_n)_p = (\fS_\mu)_p$ for any $p\in \tau^{-1}(U)$.
	We can also express the strata in terms of stabilizers: if $\nu \le \lambda$, we have
	\[\tau(p) \in S^\lambda_\nu X \iff (\fS_\lambda)_p = \fS_{w\nu}\; \mbox{for some}\; w \in \fS_\lambda.\]
	Putting this together, if $\tau(p) \in S^\lambda_\nu X \cap U$, it follows that
	$(\fS_\mu)_p = (\fS_{w\nu})_p$ for some $w \in \fS_\lambda \subset \fS_\mu$, and so $\sigma(\tau(p)) \in S^\mu_\nu X$.  This proves the statement (2).
	
	To show (1) it is enough to show that $\tau^{-1}(S^\lambda_\nu X \cap Z)$ has codimension at least two in $\tau^{-1}(S^\lambda_\nu X)$.  But $\tau^{-1}(S^\lambda_\nu X)$ is the union over $w\in \fS_\lambda$ of the locus whose $\fS_\lambda$-stabilizer is $\fS_{w\nu}$, which is Zariski open in the linear subspace $X^n_{w\nu} \subset X^n$ cut out by equations $p_i = p_j$ for all $i \sim_{w\nu} j$.  This linear subspace is clearly isomorphic to $X^{\ell(\nu)}$.  On the other hand, $\tau^{-1}(Z)$ is the union over all $i \not\sim_\lambda j$ of the linear space $X^n_{ij}$ with equation $p_i = p_j$.  But none of the $X^n_{ij}$ are contained in any $X^n_{w\nu}$, 
	since $w\nu \le \lambda$ implies that $i \sim_{w\nu} j \implies i \sim_\lambda j$.  And since $\dim X = 2$, it follows that each $X^n_{ij}$ meets $\tau^{-1}(S^\lambda_\nu Z)$ in codimension two.  The statement (1) follows.
	
	To prove (3), take any $p\in \tau^{-1}(U)$.  Then, letting $B = B_\epsilon(p)$ denote the open ball of radius $\epsilon$ in the standard round metric $d$ on $X^n$, we can choose $\epsilon>0$ so that $B\subset \tau^{-1}(U)$ and $d(p,wp) > 2\epsilon$ for any $w\in \fS_\mu$ such that $p\ne wp$. Then $B$ is invariant under $G = (\fS_\mu)_p = (\fS_\lambda)_p$, and every point of $B$ has $\fS_\mu$-stabilizer contained in $G$, since if $q \in B$ and $w \in \fS_\mu \setminus G$, then $wq$ is in $wB = B_\epsilon(wp)$, which is disjoint from $B$.  	As a result we have $\tau(B) \cong B/G \cong \sigma(\tau(B))$, i.e.\ $\tau(B)$ is a neighborhood of $\tau(p)$ which projects homeomorphically onto the neighborhood $\sigma(\tau(B))$ of $\sigma(\tau(p))$.  Thus $\sigma|_U$ is a local homeomorphism, as desired.
	
	Finally, for $\cF\in \Perv_\cS(S^\mu X, \kk)$, the property (3) implies that $j^*\sigma^*\cF= (\sigma^*\cF)|_U$ is perverse, and property (2) implies that it is $(\cS\cap U)$-constructible.
\end{proof}

\begin{proof}[Proof of Proposition \ref{prop-pervext}]
We choose a total order $\leq$ on the set of $\cS$-strata of $Y$ refining the closure relation.  We will use this total order to prove the claim inductively.

For any stratum $Q \in \cS$, let
\[ U_Q = U \cup \bigsqcup_{\substack{P \in \cS\\P>Q}} P, \]
In particular $U_Q$ is an open set containing $U$, and we let $$j_Q: U_Q \hookrightarrow U_Q \cup Q$$ denote the open inclusion.

We will now show 
\begin{claim} For any $Q \in \cS$ and any perverse sheaf $\cG \in \Perv_{\sS \cap U_Q}(U_Q,\kk)$, the perverse extensions
\[ {}^p (j_Q)_! \cG , ~(j_Q)_{!*} \cG ,~ {}^p (j_Q)_* \cG \in \Perv(U_Q \cup Q,\kk)\]
are $\cS \cap (U_Q \cup Q)$-constructible and the natural maps
\[ {}^p (j_Q)_! \cG \to (j_Q)_{!*} \cG \to {}^p (j_Q)_* \cG\]
are isomorphisms.  
\end{claim}
Our proposition will follow from this claim, as we can express each of the extensions ${}^p j_! \cG, j_{!*} \cG$ and ${}^p j_* \cG$ as iterated extensions across one stratum at a time.  For example:
\[ j_{!*} \cG = (j_{Q_N})_{!*}(j_{Q_{N-1}})_{!*} \ldots (j_{Q_1})_{!*} \cG,\]
where $Q_1 > Q_2 > \ldots >Q_{N-1} > Q_N$ is a total order that refines the closure partial order on the set of all strata.

The isomorphisms and constructibility of the claim are local statements, so it is enough to verify them in a neighborhood of any point $q \in (U_Q\cup Q)\setminus U_Q = Q\cap Z$.  We can find a neighborhood $W$ of $q$, an analytic variety $N$, a stratification 
$\cS_N = \{N_T \mid T \ge Q\}$ of $N$ and a stratification-preserving homeomorphism $\eta\colon  N \times V\to W$, where $V = W \cap Q$ and $N \times V$ has the product stratification 
$\cS_N \times V.$
 Shrinking $W$ if necessary, we can assume that $V$ is homeomorphic to an open ball.
 Since $\eta^{-1}(Q) = N_Q \times V$ and $N_Q = \{o\}$ is a single point, we have $\eta^{-1}(Z\cap Q) = \{o\}\times Z'$ where $Z' = Z \cap V$.  Then the pullback
 $\cG' = \eta^*\cG$ is a perverse sheaf on $U' = (N \times V) \setminus (\{o\}\times Z')$, constructible with respect to the restriction of $\cS_N \times V$ to $U'$.
 
 Because $Z\cap Q$ has codimension at least $2$ in $Q$, it follows that $V \setminus Z' \cong V\cap U$ is simply connected.  (To see this, one can use \cite{Loj} to obtain a triangulation of $Q$ such that $Z\cap Q$ is a subcomplex of real codimension at least $4$.  Then nullhomotopies in $Q$ can be perturbed to be transverse to $Z$.)
 By Lemma \ref{lem:stratified pullback}, we have an isomorphism
 $$\cG'|_{N \times (V\setminus Z')} \cong \cH\boxtimes\cL_0 ,$$
 where $\cH$ is an $\cS_N$-constructible perverse sheaf on $N$ and $\cL_0= \underline{\kk}_{V\setminus Z_0}[2\dim_\C Q]$ is a constant local system.  Another application of Lemma \ref{lem:stratified pullback} gives an isomorphism  
 $$\cG'|_{(N\setminus \{o\})\times V} \cong \cH_0\boxtimes \cL,$$
 where $\cH_0 \in \Perv_{\cS_N}(N\setminus \{o\},\kk)$ and $\cL = \underline{\kk}_V[2\dim_\C Q]$.  By restricting to $ N \times \{p\}$ and $(N\setminus \{o\}) \times \{p\}$ for some point $p\in V$, we see that $\cH_0 \cong \cH|_{N\setminus\{o\}}$, and so gluing these perverse sheaves gives an isomorphism
 \[\cG'|_{U'} \cong (\cH\boxtimes\cL)|_{U'}.\]

Let $i'\colon \{o\}\times Z' \to N \times V$ and $j' \colon U' \to  N \times V $ be the inclusions.  Then $i'$ factors as
\[\{o\}\times Z' \stackrel{i_1}{\longrightarrow} \{o\}\times V \stackrel{i_2}{\longrightarrow} N\times V.\]
Since $\cH\boxtimes\cL$ is perverse, we have $(i_2)^*(\cH\boxtimes\cL) \in {}^pD^{\le 0}(\{o\}\times V, \kk)$.
Then since $(i_2)^*(\cH\boxtimes\cL)$ has locally constant cohomology sheaves and $Z'$ has codimension at least two in $V$, we have 
\[(i')^*(\cH\boxtimes\cL) = (i_1)^*(i_2)^*(\cH\boxtimes\cL) \in {}^pD^{\le -2}(\{o\}\times Z',\kk).\]
Similarly, and dually, we have $(i')^!(\cH\boxtimes\cL) \in {}^pD^{\ge 2}(\{o\}\times Z',\kk)$.  Then \cite[Corollaire 1.4.24 and following]{BBD} implies that $\cH\boxtimes\cL$ is isomorphic to ${}^pj'_!(\cG'|_{U'})$, ${}^pj'_{!*}(\cG'|_{U'})$, and ${}^pj'_*(\cG'|_{U'})$.  Thus these sheaves are all isomorphic to each other, as desired.  In addition, they are $\cS$-constructible, since $\cH\boxtimes\cL$ is.  This completes the proof of the claim, and so the proposition follows.
\end{proof}

%%%%%
\section{The Hilbert-Chow sheaves and hyperbolic restriction} \label{sec:HCsheaves}
%%%%%

The $n$-fold symmetric product of the plane has a natural resolution  $$ \rho_n: \Hilb^n\! X \to S^n X,$$
where $\Hilb^n X$ denotes the Hilbert scheme of $n$ points in the plane, whose points are zero-dimensional subschemes $Z \subset X$ of length $n$.  The map $\rho_n$ is a symplectic resolution of singularities and sends $Z \in \Hilb^n X$ to the cycle
$$ \sum_{x\in X} \dim_\C (\cO_{Z,x}) \cdot x \in S^n X.$$
We may equivalently view $\Hilb^n X$ as parametrizing the ideals $I_Z \subset \cO(X) = \C[x,y]$ such that $\dim_\C \C[x,y]/I_Z = n$.

More generally, for a partition $\lambda \in \Part{[n]}$, we have a resolution 
$\rho_\lambda\colon \Hilb^\lambda\! X \to S^\lambda X$, where 
\[\Hilb^\lambda\! X = \Hilb^{\hat\lambda_1}\! X \times \dots \times \Hilb^{\hat\lambda_r}\! X.\]

It is well-known that the maps $\rho_\lambda$ are semi-small and $\cS$-stratified, in the sense that over each stratum $S$ the map $\rho^{-1}(S) \to S$ is a fiber bundle.  This can be seen by direct computation of the fibers \cite{Bri77,I77}, or by using the results of Kaledin \cite{Kal06}, who showed that these properties hold for any symplectic resolution (where the stratification of the target is the stratification by symplectic leaves).
Thus, pushing forward the (shifted) constant sheaf from the resolution gives an $\cS$-constructible perverse sheaf.

\begin{definition}
	\begin{enumerate}
		\item The Hilbert-Chow sheaf $\cF_n$ is the perverse sheaf $$ \cF_n := (\rho_n)_* \uk_{\Hilb^n\!X}[2n] \in \Perv_\cS(S^n X,\kk).$$
		\item For any partition $\lambda \in \Part{[n]}$, define
		\[\cB_\lambda := (\rho_\lambda)_*\uk_{\Hilb^\lambda\! X}[2n] \in \Perv_\cS(S^\lambda X,\kk),\] 
		and for any partition $\mu$ such that $\lambda \leq \mu$, let
		$$ \cF^\mu_\lambda := (\sigma^\mu_\lambda)_* \cB_\lambda \in \Perv_\cS(S^\mu X,\kk).$$  When $\mu=(n)$, we simplify notation and denote $\cF^{(n)}_\lambda$ by $\cF_\lambda$.
	\end{enumerate}
\end{definition}

\begin{example}
	Note that if $\lambda = (n)$, then $\cF_{(n)}$ is just the Hilbert-Chow sheaf $\cF_n$.  At the other extreme, if $\lambda = (1^n)$, then $\cF_{(1^n)} = (\sigma_{(1^n)})_* \uk_{X^n} [2n]$.  Since the map $\sigma_{(1^n)}$ is small, this is the intersection cohomology complex of a local system on the open stratum $S^n_{(1^n)}X$ whose monodromy is the regular representation of $\pi_1(S^n_{(1^n)}X) \cong \fS_n$.
\end{example}

%%%%%
\subsection{Hyperbolic restriction functors} \label{sec:hyperbolic restriction}
%%%%%

Let $\C^*$ act on $X=\C^2$ by $$t\cdot(x,y) = (tx,t^{-1}y).$$  Note that the origin $0 \in X$ is the unique fixed point of this action.  Let $\Sigma \subset X$ be the attracting locus (i.e., the $x$-axis of $\C^2$).

Fix a partition $\mu \in \Part{[n]}$ and consider the induced $\C^*$-action on $S^\mu X$.  The unique fixed point for this action is $n\cdot 0 \in S^\mu X$, the origin with multiplicity $n$,  and the attracting locus is $S^\mu \Sigma \subset S^\mu X$.

Let $f\colon \{n \cdot 0\} \hookrightarrow S^\mu \Sigma$ and $g\colon S^\mu \Sigma \hookrightarrow S^\mu X$ denote the inclusion maps.  We will consider the hyperbolic restriction functor $f^! g^*: D^b(S^\mu X,\kk) \to D^b(n \cdot 0,\kk)$.  There are several equivalent forms of this functor which we will find useful.
First, \cite[Theorem 1]{Br03} gives a natural isomorphism of functors $f^! g^* \simeq \bar f^* \bar g^!$, where 
\[\{n\cdot 0\} \stackrel{\bar f}{\longrightarrow} S^\mu\bar\Sigma \stackrel{\bar g}{\longrightarrow} S^\mu X\]
are the analogous inclusions defined using the repelling set $\bar\Sigma$ of the action (i.e., the $y$-axis) instead of $\Sigma$.  On the other hand, since $\cS$-constructible sheaves are invariant under the action of $GL_2(\C)$ on $X$, a homotopy argument rotating $\bar\Sigma$ to $\Sigma$ shows that $\bar f^* \bar g^! \simeq f^* g^!$.

\begin{proposition}\label{prop-exact}
	The functor $f^!g^*\simeq f^*g^!$ is $t$-exact.  In other words, for any perverse sheaf $\cF \in \Perv_\cS(S^\mu X,\kk)$, we have $H^k(f^!g^*(\cF))=0$ for $k\neq 0$.  In particular, we obtain an exact functor 
	\[ \Phi^\mu := H^0f^!g^* \simeq H^0f^*g^! \colon \Perv_\cS(S^\mu X,\kk) \to \Vect_\kk.\]
\end{proposition}

\begin{proof}
	The following argument is adapted from \cite[Theorem 3.5]{MV}.
	Suppose that $\cF \in {}^pD^{\le 0}_\cS(S^\mu X,\kk)$.
	If $q\colon S^\mu \Sigma \to \{n \cdot 0\}$ is the constant map, the contraction lemma (see for example~\cite[Cor. 1]{SprContraction} or \cite[Theorem 2.10.3]{AcharBook})
    implies that 
    \[H^k(f^!g^*\cF) \cong H^k(q_!g^*\cF) \cong H^k_c(S^\mu \Sigma, \cF|_{S^\mu\Sigma}).\]
    For any refinement $\lambda$ of $\mu$, put $S^\mu_\lambda\Sigma := S^\mu_\lambda X \cap S^\mu\Sigma$.  Then we have 
    \[\dim_\C S^\mu_\lambda X = 2\ell(\lambda) \quad \mbox{and} \quad \dim_\C S^\mu_\lambda\Sigma = \ell(\lambda).\]
    Then $\cF \in {}^pD^{\le 0}_\cS(S^\mu X,\kk)$ implies that the cohomology sheaves of $\cF|_{S^\mu_\lambda X}$ vanish in degrees greater than $-2\ell(\lambda)$, so we have
    \[H^k_c(S^\mu_\lambda \Sigma, \cF|_{S^\mu_\lambda\Sigma}) = 0 \quad \mbox{for} \quad k > -2\ell(\lambda) + \dim_\R S^n_\lambda\Sigma = 0.\]
    It follows that 
	$$ H^k_c(S^\mu_\lambda \Sigma, \cF)=0 \quad \mathrm{for} \quad k>0,$$
	So the functor $f^!g^*$ is right $t$-exact.  
	Then using the alternate expression for our functor, we see that $f^*g^! \simeq \D f^!g^*\D$ is left $t$-exact.
\end{proof}

\begin{definition}
	For any set partitions $\lambda \leq \mu \in \Part{[n]}$, let $$\Phi^\mu_\lambda : \Perv_\cS(S^\mu X,\kk) \to \Vect_\kk$$ be the exact functor defined as
	$ \Phi^\mu_\lambda := \Phi^\lambda (\omega^\mu_\lambda)^*.$  When $\mu=(n)$ we may simply write $\Phi_\lambda$ instead of $\Phi^{(n)}_\lambda$.
%	\[ \Phi_\lambda := (\Phi_{\hat\lambda_1} \boxtimes \Phi_{\hat\lambda_2} \boxtimes \dots \boxtimes \Phi_{\hat\lambda_r})\circ (\omega_\lambda)^*. \]
%	Here we have identified $$\Perv_\cS (S^\lambda X) \simeq \Perv_\cS (S^{\hat\lambda_1} X) \boxtimes \Perv_\cS (S^{\hat\lambda_2} X) \boxtimes \dots \boxtimes \Perv_\cS (S^{\hat\lambda_r} X).$$
\end{definition}

\begin{example}
	When $\lambda = (n)$, then 
	$\Phi_{(n)}=\Phi^{(n)}$ is the functor of hyperbolic restriction of perverse sheaves on $S^nX$ to the fixed point $\{n\cdot 0\}$, and when $\lambda = (1^n)$, then $\Phi_{(1^n)}$ is isomorphic to the stalk functor at a point of $S^n_{(1^n)} X$.
\end{example}

The ``hyperbolic stalks" provided by the functors $\Phi^\mu_\lambda$ are complete, in the following sense.

\begin{lemma}\label{lem:hyperbolic stalks are complete}
	Let $\mu \in \Part{[n]}$. Take perverse sheaves $\cF, \cG \in \Perv_\cS(S^\mu X,\kk)$.
	\begin{enumerate}
		\item 	If $\Phi^\mu_\lambda(\cF) = 0$ for all $\lambda \leq \mu$, then $\cF = 0$.
		\item For a homomorphism $f\colon \cF\to \cG$, if $\Phi^\mu_\lambda(f) \in 
		\Hom_{\kk\mmod}(\Phi^\mu_\lambda(\cF),\Phi^\mu_\lambda(\cG))$ is an injection, (respectively a surjection or an isomorphism) for every $\lambda\leq \mu$, then $f$ is an injection (respectively, a surjection or an isomorphism).
		\item The map
		\[\Hom(\cF,\cG)\to \bigoplus_{\lambda \leq \mu} \! \Hom(\Phi^\mu_\lambda\cF, \Phi^\mu_\lambda\cG)\]
		is injective.  In other words, the functor $\bigoplus_{\lambda \leq \mu} \Phi^\mu_\lambda$ is faithful.
	\end{enumerate}
\end{lemma}

\begin{proof}
	To see the statement (1), suppose that $\cF \ne 0$.  Its support is a closed union of strata, so we can take a maximal stratum $S^\mu_\lambda X$ in the support.  Restricting to this stratum, we have an isomorphism
	\[\cF|_{S^\mu_\lambda X} \cong \cL[-\dim_\C S^\mu_\lambda X],\] where $\cL$ is a nonzero local system on $S^\mu_\lambda X$.  Then $\Phi^\mu_\lambda(\cF)$ is simply the stalk of this local system at some point of $S^\mu_\lambda X$, so in particular it is nonzero. For the statement (2), apply (1) to $\Ker f$ and $\Coker f$.  
	
	For the statement (3), let $f\colon \cF\to \cG$ be a homomorphism. For any $\lambda$, by the exactness of $\Phi^\mu_\lambda$, we get an exact sequence
\[ 0 \to \Phi^\mu_\lambda \Ker f \to \Phi^\mu_\lambda \cF \to \Phi^\mu_\lambda \cG \to \Phi^\mu_\lambda \Coker f \to 0. \]
If $\Phi^\mu_\lambda f = 0$ for all $\lambda \leq \mu$, then by (2) $\Ker f \to \cF$ and $\cG \to \Coker f$ must be isomorphisms and we conclude that $f=0$.
\end{proof}

%%%%%
\subsection{Hyperbolic restriction is represented by the Hilbert-Chow sheaf}\label{sec:Hilbert-Chow and hyperbolic restriction}
%%%%%

\begin{theorem}\label{thm-representF_lambda}
For any set partitions $\lambda \leq \mu \in \Part{[n]}$, the exact functor $\Phi^\mu_\lambda$ is represented by the perverse sheaf $\cF^\mu_\lambda$.  In other words, there is an isomorphism of functors 
\[ \Hom(\cF^\mu_\lambda,-) \simeq \Phi^\mu_\lambda(-) :\Perv_\cS(S^\mu X,\kk) \to \Vect_\kk. \]
\end{theorem}

%\begin{corollary}\label{cor:dim End Fn}
%	The dimension $\dim_\kk \End(\cF_n)$ is 
%	\[p(n) = |\Part{n}| = \dim_\kk ({}_nR_n) = \rank K(\fS_n\mmod).\]  
%\end{corollary}
%
%\begin{proof}[Proof of Corollary \ref{cor:dim End Fn}]
%     Theorem \ref{thm-representFn} gives $\End(\cF_n) = \Phi(\cF_n) = H_c^{2n}(\rho_n^{-1}(S^n\Sigma))$.  It is well-known (see \cite[Chapter 7]{NakBook}, for instance) that $\rho_n^{-1}(S^n\Sigma)$ is a Lagrangian subvariety of $\Hilb^n X$, with components indexed by partitions of $n$.  
%\end{proof}

In the proof of the Theorem, we will make use of the following action on $X=\C^2$.  Let $\Pol_{n-1}$ denote the additive group of polynomials in $x$ of degree at most $n-1$.  This acts on $X$ via 
\[ f \cdot (x,y) = (x,y+f(x)). \]
This action induces actions of $\Pol_{n-1}$ on $S^n X$ and $\Hilb^n X$, and the map $\rho_n: \Hilb^n X \to S^n X$ is equivariant with respect to these actions.

Recall that we have used $\Sigma \subset X$ to denote the curve $y=0$.  Define a Zariski open subset of $\Hilb^n X$ by 
\[ \cU_n = \{ I \in \Hilb^n X \mid \C[x] \to \C[x,y]/I \textrm{~is~surjective}\}. \]
In other words, $\cU_n$ is the set of points which project to a point of $\Hilb^n \Sigma$ under the projection $X \to \Sigma$.  Since the action of $\Pol_{n-1}$ fixes the subalgebra $\C[x] \subset \C[x,y]$, it preserves $\cU_n$.

\begin{lemma}\label{lem-cU}
The restriction of the action map to $\Hilb^n \Sigma \subset \Hilb^n X$,
\[ \Pol_{n-1} \times \Hilb^n \Sigma  \to \Hilb^n X, \]
maps $\Pol_{n-1} \times \Hilb^n \Sigma$ isomorphically onto $\cU_n$.
\end{lemma}

\begin{proof}[Proof of Lemma~\ref{lem-cU}]
The projection $X \to \Sigma$ induces a map $\pi\colon \cU_n \to \Hilb^n \Sigma$.  More precisely, if $I \in \cU_n$, the kernel of 
$\C[x] \to \C[x,y]/I$ is the ideal of a degree $n$ polynomial $x^n - g(x)$, where $g\in \Pol_{n-1}$.  This ideal $(x^n - g(x))$ represents the point $\pi(I) \in \Hilb^n(\Sigma)$.

As a result, the composition
\[\Pol_{n-1} \to \C[x] \to \C[x,y]/I\]
is an isomorphism of vector spaces.  Define a map $\xi\colon \cU_n \to \Pol_{n-1}$ by letting $\xi(I)$ be the preimage of $\bar{y} \in \C[x,y]/I$ under this map.  Then $I \mapsto (\xi(I),\pi(I))$ is the inverse map $\cU_n \to \Pol_{n-1} \times \Hilb^n \Sigma$.  
\end{proof}

\begin{remark}
	The open set $\cU_n$ was previously considered by Grojnowski (for somewhat different reasons) in \cite[Proposition 6]{Groj}.
\end{remark}

For $\lambda \in \Part{[n]}$, let $\Pol_\lambda = \Pol_{\hat\lambda_1-1} \times \dots \times \Pol_{\hat\lambda_r-1}$ and $\cU_\lambda =\cU_{\hat\lambda_1} \times \dots \times \cU_{\hat\lambda_r}$, and consider the product actions of $\Pol_\lambda$ on $S^\lambda X$ and $\Hilb^\lambda X$.  It follows immediately from Lemma \ref{lem-cU}:

\begin{corollary}\label{cor-cU}
The restriction of the action map to $\Hilb^\lambda \Sigma \subset \Hilb^\lambda X$,
\[ \Pol_{\lambda} \times \Hilb^\lambda \Sigma  \to \Hilb^\lambda X, \]
maps $\Pol_{\lambda} \times \Hilb^\lambda \Sigma$ isomorphically onto $\cU_\lambda$.
\end{corollary}

\begin{proof}[Proof of Theorem~\ref{thm-representF_lambda}]
Since $\Phi^\mu_\lambda = \Phi^\lambda (\omega^\mu_\lambda)^*$, the adjunction of Theorem~\ref{thm-biadjoint} implies that it suffices to show that $\Phi^\lambda$ is represented by $\cB_\lambda = \cF^\lambda_\lambda$.

Suppose $\cG \in \Perv_\cS(\Hilb^\lambda X,\kk)$.  The stratification $\cS$ is stable under the action of the contractible group $\Pol_{\lambda}$, so $\cG$ is $\Pol_{\lambda}$-equivariant.  Thus if
$$ a,p: \Pol_{\lambda} \times S^\lambda X \to S^\lambda X$$
are the action and projection maps, then $$a^* \cG \cong p^* \cG.$$  Furthermore, the map $a$ is smooth with $n$-dimensional fibers, so
\[ a^! \cG \cong a^* \cG[2n] \cong p^* \cG[2n] \]
and
\begin{equation}\label{eq-smooth}
 \cG \cong p_*p^* \cG \cong p_* a^! \cG[-2n].
\end{equation}

Consider the following diagram:
\[
\begin{tikzcd}
&&&\Pol_{\lambda} \times \Hilb^\lambda X \arrow[dl, "\id \times \rho_\lambda"] \arrow[dd,"a'"]&\\
&\Pol_{\lambda} \times S^\lambda \Sigma \arrow[urr, bend left=10, "h"] \arrow[dd, "p_\Sigma"] \arrow[r, "\id \times g"]& \Pol_{\lambda} \times S^\lambda X \arrow[dd,xshift=0.7ex,"a"]\arrow[dd,xshift=-0.7ex,"p"']& &\\
&&&\Hilb^\lambda X \arrow[dl,"\rho_\lambda"]& \cU_\lambda \arrow[l,hook',"\phi"]\\
n \cdot 0 \arrow[r,hook,yshift=0.7ex,"f"] & S^\lambda \Sigma \arrow[l,yshift=-0.7ex,"q"] \arrow[r,hook,"g"]& S^\lambda X&&.
\end{tikzcd}
\]
Here $q$ is the unique map to the point, $p_\Sigma$ is the projection, $a'$ is the action map, $h$ is the identity on the first factor and the inclusion 
$$ S^\lambda \Sigma \cong \Hilb^\lambda \Sigma \hookrightarrow \Hilb^\lambda X$$
on the second factor, and $\phi$ is the inclusion.  The triangles and the inner quadrilaterals all commute, and the left square is a fiber square.

Using the alternate definition of $\Phi^\lambda$ we find:
\begin{align*}
\Phi^\lambda (\cG) & = \Hom(\uk_{n\cdot 0}, f^* g^! \cG)\\ & \cong \Hom(\uk_{n\cdot 0}, q_* g^! \cG) & \mbox{(Contracting lemma)}\\ 
& \cong \Hom(\uk_{n\cdot 0}, q_* g^! (p_* a^! \cG [-2n]))  & \mbox{(Equation \ref{eq-smooth})} \\
 & \cong \Hom(a_! p^* g_! q^* \uk_{n\cdot 0} [2n], \cG) & \mbox{(adjunction)}\\
& \cong \Hom(a_! (\id \times g)_! p_\Sigma^* q^* \uk_{n\cdot 0} [2n], \cG) & \mbox{(base change)}\\
& \cong \Hom(a_! (\id \times g)_! \uk_{\Pol_{\lambda} \times S^\lambda \Sigma} [2n], \cG)\\
& \cong \Hom((\rho_\lambda)_! a'_! h_! (\uk_{\Pol_{\lambda} \times S^\lambda \Sigma} [2n]), \cG)\\
& \cong \Hom((\rho_\lambda)_! \phi_! (\uk_{\cU} [2n]), \cG). & \mbox{(Corollary \ref{cor-cU})}
\end{align*}
This isomorphism holds for any complex $\cG$ for which $g^!\cG$ is constructible with respect to a conic stratification on $S^\lambda\Sigma$.  So for such complexes we have shown that $\Phi^\lambda$ is represented by $(\rho_\lambda)_! \phi_! (\uk_{\cU_\lambda} [2n]) \in D^b(S^\lambda X, \kk)$.  It remains to show that when $\cG$ is perverse and $\cS$-constructible, we have
\[ \Hom((\rho_\lambda)_! \phi_! (\uk_{\cU} [2n]),\cG) \cong \Hom( \cB_\lambda, \cG). \]

Let $Z = \Hilb^\lambda X \setminus \cU_\lambda$ and let $\rho_Z: Z \to S^\lambda X$ be the restriction of $\rho_\lambda$ to $Z$.  Consider the  triangle
\[ (\rho_\lambda)_! \phi_! (\uk_{\cU} [2n]) \to (\rho_\lambda)_! (\uk_{\Hilb^\lambda X} [2n]) = \cB_\lambda \to (\rho_Z)_! \uk_Z [2n] \stackrel{[1]}{\to}.\]

By the long exact sequence, it suffices to show 
$$\Hom^i((\rho_Z)_! \uk_Z [2n],\cG)=0 \quad \mathrm{for} \quad i=0,1.$$

Let $r\colon S^\lambda X \to S^\lambda X$ be the map induced by the map $X \to X$, $(x,y) \mapsto (x,0)$.  For partitions $\mu, \nu \le \lambda$ define
\[Y_{\mu,\nu} = S^\lambda_\mu X\cap r^{-1}(S^\lambda_\nu X).\]
It is empty unless $\mu\le \nu$, and it is a smooth variety of dimension $\ell(\mu)+\ell(\nu)$.  
The decomposition of $S^\lambda X$ into sets $Y_{\mu,\nu}$ refines the stratification $\cS$.  It is possible to show that it is a stratification in the sense of Definition \ref{defn:stratification}, but we do not need this.

Fix an ordering $Y_1,\dots,Y_r$ of the sets $Y_{\mu,\nu}$ such that each union $Y_1\cup\dots\cup Y_p$ is closed.  Following~\cite[3.4]{BGS}, there is a spectral sequence that abuts to $\Hom^k((\rho_Z)_! \uk_Z [2n],\cG)$
with 
$E_1$-term
\[  E_1^{p,q} = \Hom^{p+q}(i_{Y_p}^* (\rho_Z)_! \uk_Z [2n], i_{Y_p}^! \cG),\]
where $i_Y\colon Y \to S^\lambda X$ denotes the inclusion.
By this spectral sequence, it suffices to show that for any $Y = Y_{\mu,\nu}$,
$$\Hom^k(i_Y^*(\rho_Z)_* \uk_Z [2n],i_Y^! \cG)=0 \quad \mathrm{for} \quad k=0,1.$$

Suppose $Y = Y_{\mu,\nu}\subset S^\lambda_\mu X$.  Recall that $\dim S^\lambda_\mu X = 2 \ell(\mu)$.  Let $c_1 = \ell(\mu) - \ell(\nu)$ be the codimension of $Y$ in $S^\lambda_\mu X$. As $\cG$ is assumed to be $\cS$-constructible, it follows that 
$$i_Y^! \cG \in {}^p D^{\geq c_1}(Y, \kk).$$

Note that over each $Y = Y_{\mu,\nu}$ the map $\rho_Z$ restricts to a fibration.
For any $y \in Y$, let $c_2$ be the codimension of $\rho_Z^{-1}(y)$ in $\rho_\lambda^{-1}(y)$.  Since $\rho_\lambda$ is semismall, the restriction of  
$(\rho_Z)_! \uk_Z [2n]$ to $S^\lambda_\mu X$ lies in ${}^p D^{\leq - 2c_2}(S^n_\lambda,\kk).$  Restricting further to $Y$, we get that
$i_Y^* (\rho_Z)_! \uk_Z [2n] \in {}^p D^{\leq - 2c_2-c_1}(Y,\kk)$.
Combining these statements, we have \[\Hom^k(i_Y^*(\rho_Z)_! \uk_Z [2n],i_Y^! \cG)=0 \;\,\mbox{for any}\;\, k< 2(c_1+c_2).\] So our result will follow if we can show that either $c_1$ or $c_2$ is non-zero.  

If $Y=Y_{\mu,\nu}$ with $\mu < \nu$, then $c_1 > 0$.  So consider the case $Y = Y_{\mu,\mu}$.
Any point $y \in Y$ is in the image of $\Pol_{\lambda} \times S^n \Sigma \to S^n X$, and so $\cU\cap \rho_\lambda^{-1}(y)$ is not empty.  But $\cU$ is a Zariski open set, so $\rho_Z^{-1}(y)$ is a proper closed subvariety of the fiber $\rho_\lambda^{-1}(y)$.  This fiber is irreducible (\cite{Bri77}), so the codimension $c_2$ of $\rho_Z^{-1}(y)$ in $\rho_\lambda^{-1}(y)$ is positive.
\end{proof}

The functor $\Phi^\lambda$ is $t$-exact, by Proposition \ref{prop-exact}, and pullback by the open inclusion $\omega^\mu_\lambda$ is $t$-exact, so Theorem \ref{thm-representF_lambda} implies that $\cF^\mu_\lambda$ is a projective object in $\Perv_\cS(S^\mu X,\kk)$.
Since $\cF^\mu_\lambda$ is self-dual, we have the following.

\begin{corollary}
	The sheaf $\cF^\mu_\lambda$ is both projective and injective in $\Perv_{\cS}(S^\mu X, \kk).$
\end{corollary}

\begin{remark}
	Up to isomorphism, the functor $\Phi^\mu_\lambda$ and the perverse sheaf $\cF^\mu_\lambda$  only depend on the $\fS_\mu$-orbit of $\lambda$, which is given by an integer partition for each part of $\mu$.
\end{remark}

\begin{proposition}\label{prop-projgen}
	The sum $\cF^\mu := \bigoplus_{\lambda \le \mu} \cF^\mu_\lambda$ is a projective generator and an injective cogenerator of $\Perv_\cS(S^\mu X, \kk)$.
\end{proposition}
\begin{proof}
	By the self-duality of $\cF^\mu$, it will be enough to show $\cF$ is a projective generator.  
To do that, it's enough to show that the functor $\Hom(\cF^\mu,-)$ is faithful.  But Theorem~\ref{thm-representF_lambda} implies that $\cF^\mu$ represents the exact functor $\bigoplus_{\lambda\le \mu} \Phi^\mu_\lambda$, which is faithful by Lemma~\ref{lem:hyperbolic stalks are complete}(3).
	
%	Let $c_\lambda \in \Phi_\lambda(\cF_\lambda)$ correspond to the identity $1_{\cF_\lambda}$ under the adjunction of Theorem \ref{thm-representF_lambda}.
%	Then for any $\cG \in \Perv_\cS(S^nX,\kk)$, the adjunction
%	$\Hom(\cF_\lambda,\cG) \stackrel{\sim}{\to} \Phi_\lambda(\cG)$ takes 
%	a map $f$ to $\Phi_\lambda(f)(c_\lambda)$.  This implies that the map
%	\[\bigoplus_{\lambda\in \Part{n}} \Phi_\lambda(\cG) \otimes \cF_\lambda \to \cG\]
%	obtained by adjunction induces surjections onto every hyperbolic stalk $\Phi_\lambda(\cG)$.  Lemma \ref{lem:hyperbolic stalks are complete} then implies that the map itself is surjective.
\end{proof}

\begin{lemma}\label{lem:omega* B}
	If $\lambda \le \mu \in \Part{[n]}$, then there is an isomorphism of perverse sheaves:
	\[(\omega^\mu_\lambda)^* \cB_\mu \cong \cB_\lambda.\]
	In particular, we have $(\omega_\lambda)^*\cF_n \cong \cB_\lambda$.	
	
	Moreover, this isomorphism can be made canonical in the following sense. Given two factorization maps $\omega_1,\omega_2 \colon S^\lambda X \to S^\mu X$, there is a commutative diagram of isomorphisms
	\[\xymatrix{
		\omega_1^* \cB_\mu \ar[rr]^\cong \ar[dr]_\cong && \omega_2^* \cB_\mu \ar[dl]^\cong \\
		& \cB_\lambda  & 
	}\]
	where the top isomorphism is induced by the canonical isomorphism of functors of Lemma~\ref{lem:omega pullback}.
\end{lemma}

\begin{proof}
	It is enough to consider the case $\mu = (n)$, since $\omega_\lambda^\mu$ is a product of maps $\omega_{\nu_i}\colon S^{\nu_i}X \to S^{\hat\lambda_i}X$, where the partition $\nu_i \in \Part{[\hat\lambda_i]}$ has parts which are the same size as the parts of $\mu$ contained in the $i$th part of $\lambda$.
	
	Let $\psi\colon S^\lambda X \to S^\lambda X$ be the $\lambda$-separated map such that $\omega_\lambda = \sigma_\lambda \circ \psi$, and let $W = \omega_\lambda(S^\lambda X)$.  There are disjoint open subsets $V_1, \dots, V_r \subset X$ so that $W$ is the locus of points $x = \sum_{i=1}^n p_i$ for which $\hat\lambda_1$ of the $p_i$ are in $V_1$, $\hat\lambda_2$ are in $V_2$, and so on.   
    The local nature of the Hilbert scheme implies that there is a fiber square

\[\xymatrix{\Hilb^{\hat\lambda_1} V_1 \times \dots \times \Hilb^{\hat\lambda_r} V_r \ar[r]^>>>>>\sim\ar[d]^{\rho_\lambda} & \rho_n^{-1}(W) \ar[d]^{\rho_n} \\
	\psi(S^\lambda X) = S^{\hat\lambda_1} V_1 \times \dots \times S^{\hat\lambda_r}V_r \ar[r]_>>>>>>>>\sim^>>>>>>>>{\sigma_\lambda} & W 
}\]
    Since $\rho_\lambda$ and $\rho_n$ are proper, base change gives an isomorphism
	\[(\omega_\lambda)^*\cF_n = \psi^*(\sigma_\lambda)^*\cF_n\cong \psi^*\cB_\lambda \cong \cB_\lambda,\]
	where the last isomorphism holds by Lemma \ref{lem:psi isom id}.
	
To prove the canonicity of this isomorphism, let $E = \Emb_n(X,\lambda)$ and consider the map $\bar\Omega \colon S^\lambda X\times E \to S^nX \times E$ whose first factor is the universal factorization map $\Omega$ as in the proof of Lemma \ref{lem:omega pullback}, and whose second factor is the projection onto $E$.  Similarly, let $\bar\Psi\colon S^\lambda X \times E \to S^\lambda X \times E$ be the product of
the universal $\lambda$-separated map of Lemma \ref{lem:psi isom id} (restricted to $E$) and the projection onto $E$.  The image $\bar\Omega(S^\lambda\times E)$ is an open subset $\cW$ of $S^n\times E$. Then the fiber square above can be constructed for every $\lambda$-separated map $\psi\in E$, as a family over $E$:

\[\xymatrix{\Hilb^{\hat\lambda_1} \cV_1 \times_E \dots \times_E \Hilb^{\hat\lambda_r} \cV_r \ar[r]^>>>>>\sim\ar[d]^{\rho_\lambda} & (\rho_n\times \id_E)^{-1}(\cW) \ar[d]^{\rho_n\times \id_E} \\
	\bar\Psi(S^\lambda X) = S^{\hat\lambda_1} \cV_1 \times_E \dots \times_E S^{\hat\lambda_r}\cV_r \ar[r]_>>>>>>>>>\sim^>>>>>>>>>{\sigma_\lambda \times \id_E} & \cW 
}\]
Here $\cV_1,\dots,\cV_r\subset X \times E$ are the corresponding families of open sets, and $\Hilb^{\hat\lambda_r} \cV_r$ is the union of Hilbert schemes of each open set in the family. 
Now base change gives
\[\Omega^*\cF_n = \bar\Omega^*(\cF_n \boxtimes \underline\kk_E) = \bar\Psi^*(\sigma_\lambda \times \id_E)^*\cF_n \cong \bar\Psi^*(\cB_\lambda \boxtimes \underline{\kk}_E)\cong \cB_\lambda\boxtimes \underline\kk_E.\]
The result follows from this, because the isomorphism of Lemma \ref{lem:omega pullback} is obtained from applying Lemma \ref{lem:stratified homotopy} to the left hand side, while applying it to $\cB_\lambda\boxtimes \underline\kk_E$ gives the identity.
\end{proof}

%%%%%
\section{Defining the homomorphism} \label{sec:hom}
%%%%%
In this section we define a ring homomorphism
\[\Theta\colon R_{[n]}\otimes_\Z \kk \to E_{[n]} := \End(\oplus_{\lambda \in \Part{[n]}} \cF_\lambda).\]  To do this, we describe the images of the 
generators ${}_\lambda\bone_\mu$ and $z_*$ and show that they satisfy the relations given by Theorem \ref{thm: presenting Rn}.

Take two partitions $\lambda, \mu \in \Part{[n]}$ with $\lambda \le \mu$.  We have the quotient map $\sigma_\lambda^\mu\colon S^\lambda X \to S^\mu X$ and choose a $(\lambda,\mu)$-separated map $\psi:S^\lambda X \to S^\lambda X$ with corresponding factorization map $\omega_\lambda^\mu = \sigma_\lambda^\mu \circ \psi \colon S^\lambda X \to S^\mu X$. By Theorem \ref{thm-biadjoint},
$(\omega_\lambda^\mu)^*$ and $(\sigma_\lambda^\mu)_*$ are biadjoint functors between $\Perv_\cS(S^\lambda X,\kk)$ and $\Perv_\cS(S^\mu X,\kk)$.  In particular, we have the unit morphism $\cB_\mu \to (\sigma_\lambda^\mu)_*(\omega_\lambda^\mu)^*\cB_\mu$.

Applying $(\sigma_\mu)_*$ to this, and using the isomorphism $(\omega_\lambda^\mu)^*\cB_\mu \cong \cB_\lambda$ given by Lemma \ref{lem:omega* B} and the fact that
$\sigma_\mu \sigma_\lambda^\mu = \sigma_\lambda$, we get a morphism
\[_\lambda 1_\mu\colon \cF_\mu = (\sigma_\mu)_*\cB_\mu \to (\sigma_\lambda)_*(\omega_\lambda^\mu)^*\cB_\mu = (\sigma_\lambda)_*\cB_\lambda = \cF_\lambda.\]

\begin{lemma}
The morphism $_\lambda 1_\mu: \cF_\mu \to \cF_\lambda$ is independent of the choice of factorization map $\omega:S^\lambda X \to S^\mu X$.
\end{lemma}

\begin{proof}
Suppose $\omega_1$ and $\omega_2$ are factorization maps $\omega_i:S^\lambda X \to S^\mu X$.  By Corollary~\ref{cor:unit}, the unit of the adjunctions can be expressed as 
\[ 1_{\Perv_\cS(S^\mu X)} \to (\sigma_\lambda^\mu)_* {}^p\! j_* j^*  (\sigma_\lambda^\mu)^* \cong (\sigma_\lambda^\mu)_*(\omega_i)^*.\]
Thus the morphisms $_\lambda 1_\mu$ defined above for $\omega_1$ and $\omega_2$ are respectively the top and bottom paths in the diagram:
\[\xymatrix{
& & (\sigma_\lambda)_* (\omega_1)^* \cB_\mu \ar[dr]^\cong \ar[dd]^\cong &\\
\cF_\mu = (\sigma_\mu)_*\cB_\mu \ar[r] & (\sigma_\lambda)_* (  {}^p j_* j^*  (\sigma_\lambda^\mu)^* \cB_\mu ) \ar[ur]^\cong \ar[dr]_\cong & & (\sigma_\lambda)_* \cB_\lambda = \cF_\lambda. \\
& & (\sigma_\lambda)_* (\omega_2)^* \cB_\mu \ar[ur]_\cong &
 }
 \]
 The triangle on the right is commutative by Lemma~\ref{lem:omega* B} and the one on the left is commutative by the canonicity of the isomorphism of functors from Corollary~\ref{cor:unit}.
\end{proof}

Having obtained a well-defined morphism $_\lambda 1_\mu: \cF_\mu \to \cF_\lambda$, we also get a morphism $_\mu 1_\lambda\colon \cF_\lambda \to \cF_\mu$ by duality, and we define $\Theta({}_\lambda\bone_\mu) = {}_\lambda 1_\mu$ and  $\Theta({}_\mu\bone_\lambda) = {}_\mu 1_\lambda$.

Next, take a permutation $z\in \fS_n$.  Using $z$ to permute the coordinates defines a map 
$X^n \to X^n$ which descends to a well-defined map
$z\colon S^\lambda X \to S^{z\lambda}X$ for any partition $\lambda$. 
We have a commutative square
\[
\xymatrix{
\Hilb^\lambda\! X  \ar[r]^z\ar[d]^{\rho_\lambda} & \Hilb^{z\lambda}\! X \ar[d]^{\rho_{z\lambda}}\\
S^\lambda X \ar[r]^z & S^{z\lambda} X\\
}\]
which induces an isomorphism $z_*\cB_\lambda \stackrel{\sim}{\longrightarrow} \cB_{z\lambda}$.   Applying $\sigma_{\lambda *}$ and using the fact that $\sigma_\lambda z = \sigma_\lambda$ gives a homomorphism
\[{}_{z\lambda}(z_*)_\lambda\colon \cF_\lambda \to \cF_{z\lambda}.\]
Note that even if $z\lambda = \lambda$, this map might not be the identity.
We define $\Theta(z_*) \in E_{[n]}$ to be the sum of these homomorphisms over all partitions $\lambda\in P([n])$.

\begin{theorem}\label{thm:main isom}
	This definition gives an isomorphism $R_{[n]}\otimes_\Z \kk \xrightarrow{\sim} E_{[n]}$.  Restricting to one representative for each integer partition, we get an isomorphism $R_n \cong \End(\oplus_{\lambda\in \Part{n}} \cF_\lambda)$.
\end{theorem}

The rest of this section will be devoted to showing that $\Theta$ is a well-defined ring homomorphism.  In the following section we will prove that it is an isomorphism.

%%%%%
\subsection{Checking the relations I: permutations}
%%%%%

The easiest of our relations to check are the relations (b) and (c)  of Theorem \ref{thm: presenting Rn} involving the action of permutations $z \in \fS_n$.  The relations (b) follow from the obvious fact that the maps $S^\lambda X \to S^{z\lambda}X$ and $\Hilb^{\lambda}\! X \to \Hilb^{z\lambda}\! X$ induced by permutations define actions of $\fS_n$ on $\coprod_\lambda S^\lambda X$ and $\coprod_\lambda \Hilb^{\lambda}\! X$, such that the projection $\coprod_\lambda \rho_\lambda$ is $\fS_n$-equivariant.

To see the first part of the relation (c), suppose that $\lambda \le \mu$, and consider the diagram
\[
\xymatrix{S^\lambda X \ar[r]^z\ar@<.5ex>[d]^\sigma\ar@<-.5ex>[d]_\omega & S^{z \lambda}X \ar@<.5ex>[d]^\sigma\ar@<-.5ex>[d]_\omega \\
	S^\mu X \ar[r]^z & S^{z\mu} X
}
\] 
where we have omitted the sub/superscripts for clarity.  Note that we have two commutative squares: $z\omega = \omega z$ and $z\sigma = \sigma z$.  The fact that the maps $z$ are isomorphisms implies that $z_*$ intertwines the adjunctions $(\omega^*,\sigma_*)$ on the left and right.  In particular, for any $\cA \in \Perv_\cS(S^\mu X,\kk)$, applying $z_*$ to the unit $\cA \to \sigma_*\omega^* \cA$
gives the unit $z_*\cA \to \sigma_*\omega^* z_*\cA$ for $z_*\cA \in \Perv(S^{z\mu},\kk)$.
Applying this to $\cA = \cB_\mu$ implies that 
\[z_* \cdot {}_\lambda 1_\mu = {}_{z\lambda}1_{z\mu} \cdot z_*.\]

Next, suppose that $z\in \fS_\mu$.  Then $z\colon S^\mu X \to S^\mu X$ is the identity map, and the induced homomorphism $z_*\cB_\mu \to \cB_\mu$ is also the identity.  Pushing forward by $\sigma_\mu$, we see that ${}_{z\mu}(z_*)_\mu\colon \cF_\mu \to \cF_{z\mu} = \cF_\mu$ is the identity, and so we have  
${}_{\lambda}1_\mu \cdot z_* = {}_{\lambda}1_\mu$.  Similarly, if $z \in \fS_\lambda$, the homomorphism ${}_{z\lambda}(z_*)_\lambda\colon \cF_\lambda \to \cF_{z\lambda} = \cF_\lambda$ is the identity, and so $z_*\cdot {}_\lambda 1_\mu =  {}_\lambda 1_\mu$.

%%%%%
\subsection{Checking the relations II: up/up and down/down compositions}
%%%%%
Take partitions $\lambda,\mu,\nu \in \Part{[n]}$ with $\lambda\le\mu\le \nu$.  Then we have relative symmetrization and factorization maps
\[
\xymatrix{
S^\lambda X \ar@<.5ex>[r]^{\sigma_\lambda^\mu}\ar@<-.5ex>[r]_{\omega_\lambda^\mu} & S^\mu X \ar@<.5ex>[r]^{\sigma_\mu^\nu}\ar@<-.5ex>[r]_{\omega_\mu^\nu} & S^\nu X.
}\]
Note that the symmetrization map $\sigma_\lambda^\nu \colon S^\lambda X\to S^\nu X$ is obtained as the composition: 
\[\sigma_\lambda^\nu = \sigma_\mu^\nu ~\sigma_\lambda^\mu.\]

%\carltodo{The argument given below seems very indirect (instead of showing that the composition agrees with the unit, instead we claim that that is equivalent to the composition of adjunctions agreeing, which is then equivalent to the composition of counits being the same...).  Maybe the following is better?
%If the unit of the adjunction $(\omega^*,\sigma_*)$ can be expressed as
%\[ 1_{\Perv_\cS(S^n X)} \to \sigma_* \sigma^* \to \sigma_* {}^p j_* j^*  \sigma^* \simeq \sigma_*\omega^*,\]
%then we just need to show that the unit map 
%\[ 1_{\Perv_\cS(S^n X)} \to (\sigma_\lambda)_* {}^p j_{\lambda*} j_\lambda^*  (\sigma_\lambda)^*,\]
%agrees with the composition of unit maps:
%\[ 1_{\Perv_\cS(S^n X)} \to (\sigma_\mu)_* {}^p j_{\mu*} j_\mu^*  (\sigma_\mu)^*\to (\sigma_\mu)_* (\sigma_\lambda^\mu)_* {}^p (j_\lambda)_* (j_\lambda)^* (\sigma_\lambda^\mu)^* {}^p (j_\mu)_* (j_\mu)^*  (\sigma_\mu)^*,\]
%Here $j_\lambda$ and $j_\mu$ are the inclusions of the appropriate open subsets of $S^\lambda X$ and $S^\mu X$ respectively.
%As the image of $j_\lambda$ under $\sigma_\lambda^\mu$ is contained in the image of $j_\mu$, the commutative diagram gives the desired result.
%}

Similarly, the composition $\omega_\mu^\nu ~\omega_\lambda^\mu$ is also a factorization map $\omega_\lambda^\nu \colon S^\lambda X\to S^\nu X$ for an appropriate choice of $(\lambda,\nu)$-separated map.
To be more precise, let the factorization maps $\omega_\mu^\nu$ and $\omega_\lambda^\mu$ be defined as $\omega_\mu^\nu = \sigma_\mu^\nu \psi_\mu$ and $\omega_\lambda^\mu = \sigma_\lambda^\mu \psi'_\lambda$ where
$\psi_\mu\colon S^\mu X \to S^\mu X$ is the map induced by a $(\mu,\nu)$-separated map $\psi:X^n \to X^n$ and $\psi'_\lambda\colon S^\lambda X \to S^\lambda X$ is the map induced by a $(\lambda,\mu)$-separated map $\psi':X^n \to X^n$.  

Then the composition $\psi'' := \psi \psi' : X^n \to X^n$ is $(\lambda,\nu)$-separated and we have
\begin{equation}\label{eqn:composing omegas}
\omega_\mu^\nu ~\omega_\lambda^\mu = 
\sigma_\mu^\nu \psi_\mu \sigma_\lambda^\mu \psi'_\lambda = \sigma_\mu^\nu \sigma_\lambda^\mu \psi_\lambda \psi'_\lambda = \sigma_\lambda^\nu \psi''_\lambda,
\end{equation}
where $\psi_\lambda: S^\lambda X \to S^\lambda X$ and $\psi''_\lambda: S^\lambda X \to S^\lambda X$ are the maps induced by $\psi$ and $\psi''$ respectively.

In particular, we conclude that the composition $\omega_\mu^\nu ~\omega_\lambda^\mu$ is the factorization map $\omega_\lambda^\nu := \sigma_\lambda^\nu \psi''_\lambda$.

Thus the relation ${}_\lambda 1_{\mu} \cdot {}_\mu 1_\nu = {}_\lambda 1_\nu$ will follow if we can show that the unit for the adjunction $((\omega_\lambda^\nu)^*, (\sigma_\lambda^\nu)_*)$ is equal to the composition of the units of the adjunctions $((\omega_\mu^\nu)^*, (\sigma_\mu^\nu)_*)$ and $((\omega_\lambda^\mu)^*, (\sigma_\lambda^\mu)_*)$.  This is equivalent to showing that the counit is the composition of counits, which is what we will prove.

Recall from Corollary \ref{cor:counit formula}) that the counit of the adjunction $((\omega_\mu^\nu)^*,(\sigma_\mu^\nu)_*)$ is the composition
 \[(\omega_\mu^\nu)^*(\sigma_\mu^\nu)_* = (\psi_\mu)^*(\sigma_\mu^\nu)^*(\sigma_\mu^\nu)_* \to (\psi_\mu)^* \to 1_{\Perv_{\cS}(S^\mu X,\kk)}\]
 of the counit of $((\sigma_\mu^\nu)^*, (\sigma_\mu^\nu)_*)$
with the natural isomorphism $(\psi_\mu)^* \stackrel{\sim}\longrightarrow 1_{\Perv_{\cS}(S^\mu X,\kk)}$ given by Lemma \ref{lem:psi isom id}.  The counit of 
$((\omega_\lambda^\mu)^*, (\sigma_\lambda^\mu)_*)$ is expressed similarly.
The fact that the counit $(\omega_\lambda^\nu)^*(\sigma_\lambda^\nu)_* \to 1_{\Perv_\cS(S^\lambda X,\kk)}$ is isomorphic to the composition
\[(\omega_\lambda^\nu)^*(\sigma_\lambda^\nu)_* = (\omega_\lambda^\mu)^*(\omega_\mu^\nu)^*(\sigma_\mu^\nu)_*(\sigma_\lambda^\mu)_* \to (\omega_\lambda^\mu)^*(\sigma_\lambda^\mu)_* \to 1_{\Perv_\cS(S^\lambda X,\kk)}\]
of counits
now follows, using the expression \eqref{eqn:composing omegas} for $\omega_\lambda^\nu$, the fact that the adjunction $((\sigma_\lambda^\nu)^*,(\sigma_\lambda^\nu)_*)$ is the composition of the adjunctions $((\sigma_\mu^\nu)^*,(\sigma_\mu^\nu)_*)$ and $((\sigma_\lambda^\mu)^*,(\sigma_\lambda^\mu)_*)$, and the fact that the natural isomorphism $(\psi''_\mu\psi''_\lambda)^* \stackrel{\sim}\longrightarrow 1_{\Perv_{\cS}(S^\mu X,\kk)}$ is the composition of the isomorphisms for $\psi''_\mu$ and $\psi''_\lambda$.

That concludes the proof of the relation ${}_\lambda 1_{\mu} \cdot {}_\mu 1_\nu = {}_\lambda 1_\nu$. The opposite relation ${}_\nu 1_{\mu} \cdot {}_\mu 1_\lambda = {}_\nu 1_\lambda$ then follows by duality.

%%%%%
\subsection{A Cartesian diagram}
%%%%%

Let $\lambda$, $\mu$ be two partitions of the set $[n]$. In this section, we introduce a key Cartesian diagram which 
connects the problem of computing $\Hom(\cF_\lambda, \cF_\mu)$ with the combinatorics of double cosets.

Our diagram takes the form
\begin{equation}\label{eqn:Cartesian diagram}
\begin{gathered}
\xymatrix@R+10pt@C+25pt{\coprod_z S^{\lambda\wedge z\mu}X \ar[r]^>>>>>>>>{\coprod\hat\omega_z}
	\ar[d]^{\coprod\sigma^\lambda_{\lambda\wedge z \mu}}
	&  S^\mu X \ar[d]^{\sigma_\mu} \\
	S^\lambda X \ar[r]^{\omega_\lambda=\sigma_\lambda \psi_\lambda}  & S^n X}
\end{gathered}
\end{equation}
where $z \in \fS_n$ runs over a set of representatives of double cosets $\fS_\lambda\backslash \fS_n / \fS_\mu$.
Intuitively, $\omega_\lambda$ sends the coordinates indexed by different $\sim_\lambda$-equivalence classes into disjoint disks in $X$, and to find a point in $S^\mu X$ over its image one must ``color" these points according to the parts of $\mu$.  The combinatorially distinct ways to do this are given exactly by the double cosets.  

The maps $\hat\omega_z$ on the top can be described as the composition
\[\xymatrix{
	S^{\lambda\wedge z\mu}X \ar[r]^{\psi_{\lambda\wedge z\mu}} &
	S^{\lambda\wedge z\mu}X \ar[r]^{z^{-1}} &
	S^{z^{-1}\lambda\wedge \mu}X \ar[r]^>>>>>{\sigma} & 
	S^\mu X
}\]
where $\psi_{\lambda\wedge z\mu} \colon S^{\lambda\wedge z\mu}X \to S^{\lambda\wedge z\mu}X$ denotes the map induced by the same underlying map $\psi\colon X^n \to X^n$ as $\psi_\lambda\colon S^\lambda X\to S^\lambda X$, and we omit the sub- and superscript on $\sigma$ for clarity.  This can be rewritten as
\[\xymatrix{
	S^{\lambda\wedge z\mu}X \ar[r]^{z^{-1}} &
	S^{z^{-1}\lambda\wedge \mu}X \ar[r]^{\psi'} &
	S^{z^{-1}\lambda\wedge \mu}X \ar[r]^>>>>>{\sigma} & 
	S^\mu X
}\]
where $\psi' = z^{-1}\psi_{\lambda\wedge z\mu} z$.  

The map $\psi'$ is $(z^{-1}\lambda\wedge\mu,\mu)$-separated, so we have $\hat\omega_z = \omega_{z^{-1}\lambda\wedge\mu}^\mu \circ z^{-1}$.
In particular this means that $(\hat\omega_z)^*$ preserves $\cS$-constructibility and is biadjoint with $\sigma_*(z^{-1})_*$ on perverse sheaves.

\begin{remark}
	Note that if we removed the $\psi$ factors from both horizontal maps, the square would still commute, but it would no longer be Cartesian.
\end{remark}

Our first application of this diagram is to computing the dimension of hom-spaces between our projective/injective sheaves.

\begin{proposition}\label{prop:dimensions agree}
	For any $\lambda, \mu \in \Part{[n]}$, we have
	\[\dim_\kk \Hom(\cF_\mu,\cF_\lambda) = \rank ({}_\lambda R_\mu) = \rank K(\Sh [\fS_\lambda \backslash \fS_n/\fS_\mu])\]
\end{proposition}
\begin{proof}
	First consider the case where $\lambda=\mu=(n)$.  Then a standard result about semi-small maps (see \cite[Lemma 8.6.1]{CG}, for example) implies that 
    
     \[\dim_\kk \End(\cF_n) = \sum_{\lambda\in \Part{n}} \left(\dim_\kk H^{\codim S^n_\lambda X}((\rho_n)^{-1}(x_\lambda),\kk)\right)^2,\]
     	where $x_\lambda$ is any point of $S^n_\lambda X$.  But the fiber $(\rho_n)^{-1}(x_\lambda)$ is irreducible of dimension $\codim S^n_\lambda X$ \cite{Bri77}, so this sum is just      
	\[p(n) = |\Part{n}| = \rank ({}_nR_n) = \rank K(\fS_n\mmod).\]  
	
 	Then for any partition $\nu$, this implies that 
	\begin{align*}
	\dim \End(\cB_\nu) & = \dim \End(\cF_{\hat\nu_1} \boxtimes \dots \boxtimes \cF_{\hat\nu_r}) \\
	& = \prod_{i=1}^r \dim \End(\cF_{\hat\nu_i}) \\
	& = \prod_{i=1}^n p(\hat\nu_i) \\
	& =\rank K(\fS_\nu\mmod).
	\end{align*}
	
	So we have
	\[\Hom(\cF_\mu,\cF_\lambda) = \Hom((\sigma_\mu)_*\cB_\mu,(\sigma_\lambda)_*\cB_\lambda) = \Hom((\omega_\lambda)^*(\sigma_\mu)_*\cB_\mu,\cB_\lambda)\]
	using the adjunction $((\omega_\lambda)^*, (\sigma_\lambda)_*)$.  
	By proper base change in the square \eqref{eqn:Cartesian diagram} and the fact that 
	$(\hat\omega_z)^*\cB_\mu = \cB_{\lambda\wedge z\mu}$ (Lemma \ref{lem:omega* B}), this becomes
	\[\bigoplus_z \Hom((\sigma^\lambda_{\lambda\wedge z\mu})_*\cB_{\lambda\wedge z\mu}, \cB_\lambda) = \bigoplus_z \Hom(\cB_{\lambda\wedge z\mu}, (\omega^\lambda_{\lambda\wedge z\mu})^*\cB_\lambda) = \bigoplus_z \End(\cB_{\lambda\wedge z\mu}).\]
	Thus $\Hom(\cF_\mu,\cF_\lambda)$ has dimension
	\[\sum_z \dim_\kk \End(\cB_{\lambda\wedge z\mu}) = \sum_z \rank K(\fS_{\lambda\wedge z\mu}\mmod) = \rank K(\Sh [\fS_\lambda \backslash \fS_n/\fS_\mu]).\qedhere\]
\end{proof}

%%%%%
\subsection{Checking the relations III: the Mackey decomposition}
%%%%%

Next we will show that the relation (d) from Theorem \ref{thm: presenting Rn} holds.  It is enough to consider the case where $\nu = (n)$ and $\lambda,\mu\in \Part{[n]}$ are arbitrary partitions, because we can handle the parts $(\nu_i)$ of $\nu$ separately, and then take the exterior product of the resulting identities.

We take the Cartesian square \eqref{eqn:Cartesian diagram}, along with the ``opposite" diagram where the roles of $\lambda$,$\mu$ are reversed:

\begin{equation}\label{eqn:double Cartesian diagram}
\begin{gathered}
\xymatrix@R+10pt@C+25pt{ {Z := \coprod_z S^{\lambda\wedge z\mu}X\;} \ar@<.5ex>[r]^>>>>>>>>{\hat\omega}\ar@<-.5ex>[r]_>>>>>>>>{\hat\sigma}
	\ar@<.5ex>[d]^{\bar\sigma} 
	\ar@<-.5ex>[d]_{\bar\omega}
	&  S^\mu X \ar@<.5ex>[d]^{\omega_\mu} \ar@<-.5ex>[d]_{\sigma_\mu} \\
	S^\lambda X \ar@<.5ex>[r]^{\sigma_\lambda} \ar@<-.5ex>[r]_{\omega_\lambda} & S^n X}
\end{gathered}
\end{equation}
(Strictly speaking, reversing $\lambda$ and $\mu$ in \eqref{eqn:Cartesian diagram} would give $\coprod_w S^{\mu\wedge w\lambda} X$ in the upper left.  But letting $w = z^{-1}$, the two spaces are identified under the action of $z$ on each term. This changes where the twisting maps appear, but otherwise does not affect the properties of the diagram.)

Here are the key properties of \eqref{eqn:double Cartesian diagram}.
\begin{itemize}
	\item The squares $(\hat\omega,\bar\sigma, \sigma_\mu, \omega_\lambda)$ and
	$(\hat\sigma, \bar\omega,\omega_\mu, \sigma_\lambda)$ are Cartesian.
	\item We have $\sigma_\mu \hat\sigma = \sigma_\lambda\bar\sigma = \coprod_z \sigma_{\lambda\wedge z\mu}$.
	\item If $\omega_\lambda = \sigma_\lambda\psi_\lambda$ and $\omega_\mu = \sigma_\mu\psi_\mu$, then there are maps $\hat\psi_\lambda, \bar\psi_\mu\colon Z\to Z$ induced by the same underlying maps $X^n \to X^n$ as $\psi_\lambda$, $\psi_\mu$ respectively.  These maps satisfy:
	\[\hat\omega = \hat\sigma \hat\psi_\lambda, \;\;\;\bar\omega = \bar\sigma \bar\psi_\mu,\;\;\; \psi_\lambda \bar\sigma = \bar\sigma \hat\psi_\lambda, \;\;\;
	\psi_\mu\hat\sigma = \hat\sigma\bar\psi_\mu.\]
	\item Although the maps
	\[\omega_\mu\hat\omega = \sigma_\mu\hat\sigma \bar\psi_\mu\hat\psi_\lambda, \;\mbox{and}\; \omega_\lambda\bar\omega = \sigma_\lambda\bar\sigma\hat\psi_\lambda\bar\psi_\mu\]
	may not be equal, $\bar\psi_\mu\hat\psi_\lambda$ and $\hat\psi_\lambda\bar\psi_\mu$ are both $(\lambda\wedge z\mu)$-separated on each component $S^{\lambda\wedge z\mu} X$ of $Z$, and so Lemma \ref{lem:omega pullback} gives a natural isomorphism of functors $\hat\omega^*(\omega_\mu)^* \simeq \bar\omega^*(\omega_\lambda)^*$ on $\Perv_\cS(S^nX,\kk)$.  In the same way we have a natural isomorphism $\hat\omega^*(\psi_\mu)^* \simeq (\bar\psi_\mu)^*\hat\omega^*$ for perverse sheaves which are $(\cS\cap U_\mu)$-constructible on the open set $U_\mu \subset S^\mu X$ of points whose different $\mu$-components don't intersect (this set was used in the proof of Proposition \ref{prop-hominto}).
\end{itemize} 

\begin{lemma}\label{lem:adjunction compatibility}
	Under the natural isomorphism
	\[(\omega_\lambda)^*(\sigma_\mu)_*(\omega_\mu)^* \simeq \bar\sigma_*\hat\omega^*(\omega_\mu)^* \simeq \bar\sigma_*\bar\omega^*(\omega_\lambda)^*,\]
	the maps
	\[(\omega_\lambda)^* \to (\omega_\lambda)^*(\sigma_\mu)_*(\omega_\mu)^*,\;\;\; (\omega_\lambda)^* \to \bar\sigma_*\bar\omega^*(\omega_\lambda)^*\]
	obtained by composing with the units of the adjunctions become equal.
\end{lemma}
\begin{proof}
Take $\cF\in \Perv_\cS(S^nX,\kk)$, and consider the diagram:
	\[\xymatrix{
	     \Hom(\cF,(\sigma_\mu)_*(\omega_\mu)^*\cF) \ar[d]_{adj}\ar[r]^<<<<<<<{(\omega_\lambda)^*} & \Hom((\omega_\lambda)^*\cF, \bar\sigma_*\bar\omega^*(\omega_\lambda)^*\cF) \ar[d]_{adj} &\\ \Hom((\sigma_\mu)^*\cF,(\omega_\mu)^*\cF ) \ar[d]_{(\psi_\mu)^*}\ar[r]^{\hat\omega^*} & \Hom(\bar\sigma^*(\omega_\lambda)^*\cF,\bar\omega^*(\omega_\lambda)^*\cF)\ar[d]_{(\bar\psi_\mu)^*}\\ \Hom((\psi_\mu)^*(\sigma_\mu)^*\cF,(\psi_\mu)^*(\omega_\mu)^*\cF)\ar[d]\ar[r]^{\hat\omega^*} & \Hom((\bar\psi_\mu)^*\bar\sigma^*(\omega_\lambda)^*\cF,(\bar\psi_\mu)^*\bar\omega^*(\omega_\lambda)^*\cF)\ar[d]\\
	     \Hom((\omega_\mu)^*\cF,(\omega_\mu)^*\cF) \ar[r]^{\hat\omega^*} & \Hom(\bar\omega^*(\omega_\lambda)^*\cF,\bar\omega^*(\omega_\lambda)^*\cF)\\
	}\]
Here the vertical columns give the adjunction $((\omega_\mu)^*,(\sigma_\mu)_*)$ and $((\bar\omega)^*,(\bar\sigma)_*)$ as given in Theorem \ref{thm-biadjoint}.  The horizontal maps are given by applying the indicated pullbacks and using base change and commutativity in the diagram.

The top square commutes by the compatibility of the adjunction $((\sigma_\mu)^*,(\sigma_\mu)_*)$ with pullback by the open inclusion $\omega_\lambda$.  The middle square commutes using the natural isomorphism $(\psi_\mu)^*\hat\omega^* \simeq \hat\omega^*(\bar\psi_\mu)^*$.  
%compatibility of the (classical) adjunctions with pullback by an open inclusion. The middle square commutes because $\psi_\mu \hat\omega = \hat\sigma\bar\psi_\mu\hat\psi_\lambda$ and 
%$\hat\omega\bar\psi_\mu = \hat\sigma\hat\psi_\lambda\bar\psi_\mu$, Lemma \ref{lem:omega pullback} shows that their pullbacks agree.

Next we show that the bottom square commutes.  The bottom left vertical map in the diagram is composition with the isomorphism $(\psi_\mu)^*(\omega_\mu)^*\cF \to (\omega_\mu)^*\cF$ of Lemma \ref{lem:psi isom id}, which is given by applying Lemma \ref{lem:stratified homotopy} to the pullback by a family $\Psi_\mu(t)$ of $\mu$-separated maps with $\Psi_\mu(0) = \psi_\mu$ and $\Psi_\mu(1) = \id_{S^\mu X}$.  Applying $\hat\omega^* = (\hat\psi_\lambda)^*\hat\sigma^*$, this becomes an isomorphism
\[(\hat\psi_\lambda)^*(\bar\psi_\mu)^*\hat\sigma^* (\omega_\mu)^*\cF \to (\hat\psi_\lambda)^*\hat\sigma^* (\omega_\mu)^*\cF.\] 
This isomorphism is obtained by the same homotopy argument using the family $t\mapsto \bar\Psi_\mu(t) \hat\psi_\lambda$, where  
$\bar\Psi_\mu$ denotes the path from $\bar\psi_\mu$ to the identity $Z \to Z$ induced by the same family of maps $X^n \to X^n$ as $\Psi_\mu$, so that $\Psi_\mu(t)\hat\sigma = \hat\sigma \bar\Psi_\mu(t)$ for all $t\in [0,1]$.  Each map in this path is $(\lambda\wedge z\mu, \mu)$-separated on the component $S^{\lambda\wedge z\mu}X$ of $Z$.

On the other hand, going right first and then down in the bottom square is given by applying the isomorphism $\hat\omega^*(\psi_\mu)^* \simeq (\bar\psi_\mu)^*\hat\omega^*$ followed by composition with the isomorphism $\bar\psi_\mu\bar\omega^*(\omega_\lambda)^*\cF \to \bar\omega^*(\omega_\lambda)^*\cF$ of Lemma \ref{lem:psi isom id}.  This is given by pulling back by a path of $(\lambda\wedge z\mu, \mu)$-separated maps $Z \to Z$ obtained as the concatenation of a path from $\hat\psi_\lambda\bar\psi_\mu$ to $\bar\psi_\mu\hat\psi_\lambda = \bar\Psi_\mu(0)\hat\psi_\lambda$ and the path $t\mapsto \bar\Psi_\mu(t)\hat\psi_\lambda$. 

The upshot is that both compositions in the bottom square are given by composing with a map obtained by homotopy continuation along a path of $(\lambda\wedge z\mu, \mu)$-separated maps from $\hat\psi_\lambda\bar\psi_\mu$ to $\hat\psi_\lambda$.  These paths are not the same, but since the space of $(\lambda\wedge z\mu, \mu)$-separated maps is simply connected by Lemma \ref{lem:Emb is simply connected}, Lemma \ref{lem:stratified homotopy} implies that these maps are equal.  Thus the lower square commutes.

Starting with the unit $\cF \to (\sigma_\mu)_*(\omega_\mu)^*\cF$ in the top left and mapping down gives the identity in the lower-left, and thus gives the identity in the lower right.  Then following the right column shows that $(\omega_\lambda)^*\cF \to (\omega_\lambda)^*(\sigma_\mu)_*(\omega_\mu)^*\cF = \bar\sigma_*\bar\omega^*(\omega_\lambda)^*\cF$ is the unit map of $(\bar\omega^*,\bar\sigma_*)$ applied to $(\omega_\lambda)^*\cF$, as desired.
 \end{proof}

Now we are ready to show that the relation (d) holds.  As already noted, we can assume $\nu = (n)$.  So we want to show
	\[{}_\mu1_{(n)} \cdot {}_{(n)} 1_\lambda \simeq \sum_z {}_\mu1_{\mu\wedge z\lambda} \cdot {}_{\mu\wedge z\lambda}1_{z\lambda} \cdot z_*,\]
	where the sum is over a set of representatives $z$ of double cosets in $\fS_\mu \backslash \fS_n / \fS_\lambda$.

The composition
\[\cF_\lambda = (\sigma_\lambda)_*(\omega_\lambda)^*\cF_n \xrightarrow{\;{}_{(n)} 1_\lambda\;} \cF_n {\xrightarrow{\;{}_\mu1_{(n)}\;}} (\sigma_\mu)_*(\omega_\mu)^*\cF_n = \cF_\mu\]
can also be written
\begin{equation}\label{eqn:compose adjunctions}
(\sigma_\lambda)_*(\omega_\lambda)^*\cF_n \to (\sigma_\lambda)_*(\omega_\lambda)^*(\sigma_\mu)_*(\omega_\mu)^*\cF_n\to (\sigma_\mu)_*(\omega_\mu)^*\cF_n.
\end{equation}
By Lemma \ref{lem:adjunction compatibility}, the first map is equal to 
\[(\sigma_\lambda)_*(\omega_\lambda)^*\cF_n \to (\sigma_\lambda)_*(\bar\sigma)_*(\bar\omega)^*(\omega_\lambda)^*\cF_n.\]
Pushing forward to $S^nX$ gives the sum
\[\cF_\lambda \xrightarrow{\;\sum {}_{\lambda\wedge z\mu}1_{\lambda}\;} \bigoplus_z \,\cF_{\lambda\wedge z\mu}\]
In a similar way we see that pushing the second map in \eqref{eqn:compose adjunctions} is the sum of 
\[\bigoplus_z \,\cF_{\lambda\wedge z\mu} \xrightarrow{(z^{-1})_*} \bigoplus_z \,\cF_{z^{-1}\lambda\wedge \mu} \xrightarrow{\sum {}^{}_\mu 1_{z^{-1}\lambda\wedge \mu}} \cF_\mu.\]
After some reindexing this is what we wanted to show.

%%%%%
\section{Proof of the isomorphism} \label{sec:iso}
%%%%%

%%%%%
\subsection{Computation of endomorphism ring of Hilbert-Chow sheaf} 
%%%%%

As a first step to showing that our homomorphism $\Theta$ is an isomorphism, we will show that it restricts to an isomorphism $\Theta_{(n)}\colon {}_{(n)}R_{(n)}\otimes \kk \to \End(\cF_n)$, thus proving Theorem \ref{thm-spr} from the introduction.  
We have
\[{}_{(n)}R_{(n)}\otimes \kk = K(\Sh[\fS_n\backslash \fS_n/\fS_n])\otimes \kk = R(\fS_n)\otimes \kk,\]
and a $\kk$-basis for this space is given by the elements
\[[\Ind_{\fS_\lambda}^{\fS_n} 1_{\fS_\lambda}] = {}_{(n)}\bone_\lambda \circ {}_\lambda\bone_{(n)}\]
where $\lambda$ runs over all integer partitions $\Part{n}$ (recall that we are fixing an embedding of this set into $\Part{[n]}$).  Thus it will be enough to prove the following.

\begin{proposition}\label{prop:basis of End Fn}
The elements $\Theta({}_{(n)}\bone_\lambda \circ {}_\lambda\bone_{(n)}) = {}_{(n)}1_\lambda \cdot {}_\lambda 1_{(n)}$ for $\lambda\in \Part{n}$ are a basis of $\End(\cF_n)$.	
\end{proposition}

%By Corollary \ref{cor:dim End Fn} we know that the source and target of $\Theta_{(n)}$ have the same dimension, so it is enough to show that these elements are linearly independent.  
The proof of this will be based on a calculation of the action of these endomorphisms on the hyperbolic stalk $\Phi_n(\cF_n)$, where to simplify notation we put $\Phi_n = \Phi_{(n)}$.  This group has a concrete description as follows.  With the inclusions 
$\{n\cdot 0\} \stackrel{f}{\longrightarrow} S^n\Sigma \stackrel{g}{\longrightarrow} S^nX$ as in Section \ref{sec:hyperbolic restriction}, we have
\[\Phi_n(\cF_n) = H^0(f^*g^!\cF_n) \cong H^0(g^!\cF_n) \cong H^{2n}_Z(\Hilb^n(X);\kk),\]
the top cohomology of $\Hilb^n (X)$ with supports in the Lagrangian subvariety $Z = \rho^{-1}(S^n\Sigma)$.

  The irreducible components of $Z$ are 
$Z_\mu := \overline{\rho_n^{-1}(S^n_\mu\Sigma)}$ for integer partitions $\mu \in \Part{n}$.  
For each $\mu$ we can choose an open set $\cU_\mu$ of $\Hilb^n(X)$ which intersects $Z$ in a contractible subset of $Z_\mu$ contained in the smooth locus of $Z$.  Restriction to these open sets gives an isomorphism
\[H^{2n}_Z(\Hilb^n(X); \kk)\stackrel{\sim}{\longrightarrow} \!\!\bigoplus_{\mu\in\Part{n}} H^{2n}_{Z \cap\cU_\mu}(\cU_\mu; \kk) \,\cong\, \kk^{|\Part{n}|}. \]
Pulling back the standard basis by this isomorphism gives a basis $\{\alpha_\mu\}$ of $\Phi_n(\cF_n)$.

Similarly, we have
$\Phi_n(\cF_\lambda) \cong H^{2n}_{Z^\lambda}(\Hilb^\lambda(X))$, where
\[Z^\lambda = (\rho_\lambda)^{-1}(S^\lambda\Sigma) \cong \rho_{\lambda_1}^{-1}(S^{\lambda_1}\Sigma) \times \dots \times \rho_{\lambda_r}^{-1}(S^{\lambda_r}\Sigma).\]  The irreducible components of this Lagrangian are
\[Z_{\nu_1}\times \dots \times Z_{\nu_r}, \;\; (\nu_1,\dots,\nu_r) \in \Part{\lambda} := \Part{\lambda_1}\times\dots\times \Part{\lambda_r}.\]
Thus we have a basis 
\[\alpha_{\nu_1} \boxtimes \dots \boxtimes \alpha_{\nu_r}, \;\; (\nu_1,\dots,\nu_r) \in \Part{\lambda}\]
for $\Phi_n(\cF_\lambda)$.

\begin{proposition}\label{prop:action on basis} With respect to these bases, the map
	$\Phi_n({}_\lambda 1_{(n)}) \colon \Phi_n(\cF_n) \to \Phi_n(\cF_\lambda)$ induced by 
	${}_\lambda 1_{(n)}$  is given by 
	\[\alpha_\mu \mapsto \sum_{\nu_1|\nu_2|\dots|\nu_r = \mu} \alpha_{\nu_1} \boxtimes \dots \boxtimes \alpha_{\nu_r},\]
	where the sum is over tuples $(\nu_1,\dots,\nu_r) \in \Part{\lambda}$ whose concatenation is $\mu$.  In particular, $\alpha_\lambda$ is sent to $\alpha_{(\lambda_1)}\boxtimes \dots\boxtimes \alpha_{(\lambda_r)}$, and 
	$\alpha_\mu$ is sent to $0$ if $\mu \not\preceq \lambda$.
\end{proposition}

Before proving this proposition, let us explain how it implies that $\Theta_{(n)}$ is an isomorphism. 
First note that we have $\End(\cF_n) \cong \Phi_n(\cF_n)$, by Theorem \ref{thm-biadjoint}.  Thus the source and target of $\Theta_{(n)}$ have the same dimension, so it is enough to show that
the operators $\Phi_n({}_{(n)}1_\lambda \circ {}_\lambda 1_{(n)})$ are linearly independent.  
By the triangularity in the second statement of the proposition, it suffices to show that 
\[\Phi_n({}_{(n)}1_\lambda \cdot {}_\lambda 1_{(n)})(\alpha_\lambda) = \Phi_n({}_{(n)}1_\lambda)(\alpha_{(\lambda_1)}\boxtimes \dots \boxtimes \alpha_{(\lambda_r)})\] is nonzero, independent of the field $\kk$.

The sheaves $\cF_n$ and $\cF_\lambda$ are canonically isomorphic to their Verdier duals, and under duality ${}_\lambda 1_{(n)}$ is sent to ${}_{(n)}1_\lambda$.  Thus we can use Proposition \ref{prop:action on basis} to describe the action of $\Phi_n({}_{(n)}1_\lambda)$.  However, our basis $\{\alpha_\mu\}$ is not invariant under duality.  Instead we must use the dual basis of
\[\bar\Phi_n(\cF_n) = H^0(\bar f^!\bar g^*\cF_n )\cong H^{2n}_c(\bar Z;\kk) \cong H^{2n}(\bar Z, \bar Z \setminus (\bar Z\cap Z);\kk),\]
where $\bar\Sigma \subset X$ is the $y$-axis and $\bar Z = \rho_n^{-1}(S^n\bar\Sigma)$. The second isomorphism holds because $\bar Z\cap Z = \rho_n^{-1}(0)$ is the union of all $\C^*$-orbits in $\Hilb^n X$ whose closure is compact.  The irreducible components of $\bar Z$ are again indexed by $\Part{n}$, and 
restricting to these components determines the dual basis $\{\bar\alpha_\lambda\}$.   

We have $\D f^*g^! \D \simeq f^!g^*$, and the result of Proposition \ref{prop:action on basis} is the same using either curve $\Sigma$ or $\bar\Sigma$, so 
applying duality the Proposition gives
\[\bar\Phi_n({}_{(n)}1_\lambda)(\bar\alpha_{\mu_1} \boxtimes \dots \boxtimes \bar\alpha_{\mu_r}) = \bar\alpha_{\mu_1|\dots|\mu_r}.\]

To relate the two bases of the hyperbolic stalk, we use the main theorem of hyperbolic restriction \cite{Br03}, which gives gives an isomorphism of functors $\Phi_n = H^0f^*g^! \simeq H^0\bar f^!\bar g^* = \bar\Phi_n$.  

Applying this to $\cF_n$ gives the isomorphism
\[H^{2n}_Z(\Hilb^n(X); \kk)\stackrel{\sim}{\longrightarrow} H^{2n}(\bar Z, \bar Z \setminus (\bar Z\cap Z);\kk)\] which is just restriction from $\Hilb^n(X)$ to $\bar Z$.

\begin{lemma}\label{lem:dual basis}
	Under the isomorphism $\Phi_n(\cF_n) \cong \bar\Phi_n(\cF_n)$, we have
	\[\alpha_{(n)} = \bar\alpha_{(1^n)} + \mbox{other terms}.\]
\end{lemma}
\begin{proof}%[Proof of Lemma \ref{lem:dual basis}]
	The point $(x,y^n) \in \Hilb^n X$ has an affine neighborhood in which $Z_{(n)}$ and $\bar Z_{(1^n)}$ intersect transversely at a smooth point of each, and this neighborhood contains no other components of $Z$.  More precisely, we can take this neighborhood to be 
	\[ \bar\cU = \{ I \in \Hilb^n X \mid \C[y] \to \C[x,y]/I \textrm{~is~surjective}\},\]
	i.e.\ the set considered in Section \ref{sec:Hilbert-Chow and hyperbolic restriction} with the roles of $x$ and $y$ interchanged.  Then the inverse image of $Z_{(n)}$ under the isomorphism
	\[ \Pol_{n-1}(y) \times \Hilb^n(\bar\Sigma)  \to \bar{U}\]
	of Lemma \ref{lem-cU} is $\Pol_{n-1}(y) \times \{n\cdot 0\}$, and the inverse image of $Z_{(1^n)}$ is $\{0\}\times \Hilb^n(\bar\Sigma)$.	
\end{proof}

Putting this together we have 
\begin{align*}
\Phi_n({}_{(n)}1_\lambda \circ {}_\lambda 1_{(n)})(\alpha_\lambda) & = \Phi_n({}_{(n)}1_\lambda)(\alpha_{\lambda_1}\boxtimes \dots \boxtimes \alpha_{\lambda_r}) \\
& = \Phi_n({}_{(n)}1_\lambda)(\bar\alpha_{(1^{\lambda_1})}\boxtimes \dots \boxtimes \bar\alpha_{(1^{\lambda_r})} + \mbox{other terms}) \\
& = \bar\alpha_{(1^n)} + \mbox{other terms}.
\end{align*}
This proves the linear independence of ${}_{(n)}1_\lambda \circ {}_\lambda 1_{(n)}$, and thus the theorem.

\begin{proof}[Proof of Proposition \ref{prop:action on basis}]
Fix $\lambda\in \Part{[n]}$, and let $\sigma, \omega=\sigma\psi\colon S^\lambda X \to S^n X$ be corresponding symmetrization and factorization maps. 
On $D^b_\cS(S^nX, \kk)$ we have a natural transformation
\[1 \to \psi_*\psi^* \to \psi_*\]
which is the composition of the unit of the adjunction $(\psi^*,\psi_*)$ and the isomorphism of Lemma \ref{lem:psi isom id}.  Composing with $\sigma_*$ gives a natural transformation $\tau\colon \sigma_* \to \omega_*$.  

\begin{lemma} 
	The triangle
	\begin{equation}\label{eqn:triangle}
	\begin{gathered}
	\xymatrix{
	& \sigma_*\omega^* \ar[dr]^{\tau_{\omega^*}} \\
	1 \ar[ur]^{\eta}\ar[rr]^{\eta_\omega} & & \omega_*\omega^*}
	\end{gathered}
	\end{equation}
of functors $\Perv_\cS(S^nX,\kk) \to D^b(S^nX,\kk)$ commutes, where $\eta$ is the unit of the adjunction of Theorem \ref{thm-biadjoint} and $\eta_\omega$ is the unit of the standard adjunction $(\omega^*,\omega_*)$.
\end{lemma}
\begin{proof}
	Consider the following diagram of functors and natural transformations.
	\[\xymatrix@+1.5em{& & \psi^*\psi_*\psi^*\omega^* \ar[dr]^{\epsilon_\psi} & \\
	& \psi^*\omega^* \ar[ur]^{\eta_\psi} & \omega^*\sigma_*\psi_*\psi^*\omega^* \ar[u]^{\epsilon_\sigma}\ar[d]^{\text{Lemma} \;\ref{lem:psi isom id}}\ar[r]^{\epsilon_\omega} & \psi^*\omega \ar[d]^{\text{Lemma} \;\ref{lem:psi isom id}} \\
	\omega^* \ar[r]^{\omega^*\eta} & \omega^*\sigma_*\omega^* \ar[u]^{\epsilon_\sigma}\ar[ur]^{\eta_\psi}\ar[r]^{\omega^*\tau_{\omega^*}} & \omega^*\omega_*\omega^*\ar[r]^{\epsilon_\omega} & \omega^* &}\]
Here the natural transformation $\eta_\psi\colon 1 \to \psi_*\psi^*$  is the unit and $\epsilon_f\colon f^*f_* \to 1$ are the counits of the natural adjunctions, preceded and followed by appropriate functors.
The composition along the lower row is obtained from the transformation $\tau_{\omega^*}\circ\eta\colon 1 \to \omega_*\omega^*$ by adjunction, so the lemma will follow if we can show that it is the identity.  The diagram is commutative, so it is enough to show that the composition going up, over the top, and down is the identity.  But $\epsilon_\psi \eta_\psi$ is the identity, so the result follows from Corollary \ref{cor:counit formula}.	
%	It is enough to show that the composition
%	\[\omega^* \xrightarrow{\omega^*(\tau_{\omega^*} \circ \eta)} \omega^*\omega_*\omega^* \to \omega^*\] is the identity, where the second map is $\epsilon'\colon \omega^*\omega_* \to 1$ applied to $\omega^*$. The transformation $\eta$ is characterized by the fact that 
%	\[\omega^* \xrightarrow{\omega^*(\eta)} \omega^*\sigma_*\omega^* = \psi^*\sigma^*\sigma_*\omega^* \to \psi^*\omega^* \xrightarrow{\sim} \omega^*\]
%	is the identity.  Comparing these, the result follows from the fact that 
%	$\psi^*\to \psi^*\psi_*\psi^* \to \psi^*$ is the identity.
\end{proof}

We want to apply the triangle \eqref{eqn:triangle} to the Hilbert-Chow sheaf $\cF_n$ and take hyperbolic stalks of the result.  Note, however, that $\omega_*\omega^*\cF_n$ is not $\cS$-constructible, so we must be careful which form of hyperbolic restriction we use. We will apply the functor
$H^{0}_{S^n\Sigma}(-)$, which is isomorphic to $\Phi_n$ on $\cS$-constructible objects.
Note that unless the image of $\omega$ hits $S^n\Sigma$, the result of applying this functor to $\omega_*\omega^*\cF_n$ will be zero.  So we will choose our map $\psi$ so that $\psi^{-1}(S^\lambda\Sigma) = S^\lambda\Sigma$, and thus $\omega^{-1}(S^n\Sigma) = S^\lambda\Sigma$. 

We get the following commutative triangle:
	\begin{equation*}
\begin{gathered}
\xymatrix{
	& \llap{$\Phi_n(\cF_\lambda) =$} H^{2n}_{Z^\lambda}(\Hilb^\lambda X) \ar[dr]^{\Phi_n(\tau_{\omega^*})} \\
\Phi_n(\cF_n) = H^{2n}_{Z}(\Hilb^n X) \ar[ur]^<<<<<<<{\Phi_n(\eta)}\ar[rr] & & H^{2n}_{Z\cap\rho_n^{-1}(W)}(\rho_n^{-1}(W))}
\end{gathered}
\end{equation*}
Here $W = \omega(S^\lambda X)$, the horizontal map is restriction, and the map $\Phi(\tau_{\omega^*})$ is the isomorphism induced by the homeomorphism $h\colon  \Hilb^\lambda X \xrightarrow{\sim} \rho_n^{-1}(W)$, as in the proof of Lemma \ref{lem:omega* B}.  The proposition follows immediately from this, because starting with a neighborhood of a smooth point of the component $Z_{\mu_1} \times \dots\times Z_{\mu_r} \subset Z^\lambda$, taking its image under $h$ and mapping into $\Hilb^n X$ gives a neighborhood of a smooth point of $Z_{\mu_1|\dots|\mu_r} \subset Z$.
\end{proof}

%%%%%
\subsection{Finishing the proof} 
%%%%%

Finally, we can finish the proof of Theorem \ref{thm:main isom}.  By Proposition \ref{prop:dimensions agree}, the source and target of $\Theta_{(n)}$ have the same dimension.  So it is enough to show that the elements ${}_\lambda 1_\mu$ and $z_*$ generate the ring $\End(\oplus_{\lambda\in \Part{[n]}} \cF_n)$.

Looking more carefully at the proof of Proposition \ref{prop:dimensions agree}, we see that it in fact gives the following.
\begin{lemma}\label{lem:generating Hom(Fmu,Fla)}
	For any $\lambda, \mu\in \Part{[n]}$, the map
	\[\bigoplus_z \End(\cB_{\lambda\wedge z\mu}) \to \Hom(\cF_\mu, \cF_\lambda)\]
	given by
	\[ \phi\mapsto {}_\mu 1_{z^{-1}\lambda\wedge\mu}\cdot (z^{-1})_* \cdot (\sigma_{\lambda\wedge z\mu})_*(\phi)\cdot  {}_{\lambda\wedge z\mu}1_{\lambda}\]
	is an isomorphism of $\kk$-vector spaces.
\end{lemma}

On the other hand, for any partition $\nu \in \Part{[n]}$, we can apply 
Proposition \ref{prop:basis of End Fn} to each factor to get a basis for 
\[\End(\cB_\nu) \cong \End(\cF_{(\hat\nu_1)})\times \dots\times \End(\cF_{(\hat\nu_s)}).\]
Applying $(\sigma_\nu)_*$ gives the following.
\begin{lemma}
	A basis for $(\sigma_\nu)_*\End(\cB_\nu) \subset \End(\cF_\nu)$ is given by elements ${}_\nu 1_\eta \cdot {}_\eta 1_\nu$, where $\eta$ runs over a set of representatives for $\Part{\nu}/\fS_\nu$, i.e.\ one $\eta \le \nu$ for each choice of a tuple of integer partitions of $\hat\nu_1, \dots, \hat\nu_s$. 
\end{lemma}

In particular the elements $(\sigma_{\lambda\wedge z\mu})_*(\phi)$ appearing in Lemma \ref{lem:generating Hom(Fmu,Fla)} can be expressed in terms of our generators.  This implies that $\Theta$ is surjective, and thus completes the proof of Theorem \ref{thm:main isom}.

%%%%%
\section{The symmetrization and factorization functors} \label{sec:functors}
%%%%%
Fix $\mu \in \Part{[n]}$.  We have made heavy use of the biadjoint exact functors
\[\xymatrix@C+25pt{\Perv_\cS(S^n X,\kk)\ar@<.5ex>[r]^{(\omega_\mu)^*}
& \Perv_\cS(S^\mu X,\kk)\ar@<.5ex>[l]^{(\sigma_\mu)_*}}.\]
In this section we will describe these functors in terms of our algebraic description of these categories.

Our main theorem gives an equivalence of categories 
\[\Perv_\cS(S^nX,\kk) \cong (R_{[n]}\otimes \kk)\mmod\]
(we could also use the smaller Morita equivalent algebra $R_n$, but it will be simpler to work with all set partitions rather than a set of representatives).

\begin{theorem}
	There is an equivalance of categories
	\[\Perv_\cS(S^\mu X,\kk) \cong (R_{[\hat\mu_1]}\otimes \dots\otimes R_{[\hat\mu_r]} \otimes \kk)\mmod.\]
\end{theorem}
\begin{proof}
	By Proposition \ref{prop-projgen}, $\cF^\mu = \bigoplus_{\lambda\le \mu} \cF^\mu_\lambda$
	is a projective generator of $\Perv_\cS(S^\mu X,\kk)$.  A partition $\lambda\le \mu$ is given by set partitions $\lambda_i \in \Part{[\hat\mu_i]}$, $i=1,\dots,r$, and under the isomorphism $S^\mu X \cong S^{\hat\mu_1}X \times \dots \times S^{\hat\mu_r}X$, we have 
	\[\cF^\mu_\lambda \cong \cF_{\lambda_1} \boxtimes \dots \boxtimes \cF_{\lambda_r}.\]
	This gives 
	\[\cF^\mu \cong \cF^{(\hat{\mu}_1)} \boxtimes \dots \boxtimes \cF^{(\hat{\mu}_r)},\]
	and so we have
	\begin{align*}
		\Perv_\cS(S^\mu X,\kk) & \cong \End(\cF^\mu)\mmod \\
		 & \cong \End(\cF^{(\hat{\mu}_1)}) \otimes \dots \otimes \End(\cF^{(\hat{\mu}_1)})\mmod\\
		 & \cong  (R_{[\hat\mu_1]}\otimes \dots\otimes R_{[\hat\mu_r]} \otimes \kk)\mmod.\qedhere	\end{align*}
	 
\end{proof}

To describe the right side of this equivalence more precisely, we need a relative version of the $2$-category of Section \ref{sec:2-category}.
Define the
$2$-category $\cC_\mu$ to have objects all set partitions
$\lambda \le \mu$.  For objects $\lambda, \nu$, the $1$-morphisms are the category
\[\hom_{\cC_\mu}(\lambda,\nu) = \Sh([\fS_\nu\backslash \fS_\mu/\fS_\lambda]),\]
and composition is defined in the same way as before.  Then it is easy to see that the Grothendieck ring $R_\mu := K(\cC_\mu)$ is isomorphic to 
$R_{[\hat\mu_1]}\otimes\dots \otimes R_{[\hat\mu_r]}$.

We have a natural $2$-functor $\iota_\mu\colon \cC_\mu \to \cC_n$ which sends an object $\lambda \in \operatorname{Ob} \cC_\mu$ to the same partition in $\operatorname{Ob} \cC_n$, and on $1$-morphisms is given by the pushforward functors
\[\Sh([\fS_\nu\backslash \fS_\mu/\fS_\lambda])\to \Sh([\fS_\nu\backslash \fS_n/\fS_\lambda])\]
induced by the inclusion $\fS_\mu \to \fS_n$.  This induces a homomorphism of rings $R_\mu \to R_{[n]}$, which is easily seen to be an inclusion.  It identifies $R_\mu$ with the subalgebra of $R_{[n]}$ generated by 
\begin{itemize}
	\item idempotents $e_\lambda$ for $\lambda \le \mu$,
	\item ${}_\nu 1_\lambda$ for $\nu\le \lambda\le \mu$ and $\lambda \le \nu\le \mu$, and
	\item $z_*$ for $z \in \fS_\mu$.
\end{itemize} 
Note that this is properly contained in the algebra $eR_{[n]}e$ where $e = \sum_{\lambda \le \mu} e_\lambda$, because of the restriction that $z$ lie in $\fS_\mu$.  From now on we will use this inclusion to consider $R_\mu$ as a subalgebra of $R_{[n]}$.

\begin{theorem}\label{thm-interp}
	Under the identifications of $\Perv_\cS(S^nX,\kk)$ with $R_{[n]}\otimes \kk\mmod$ and 
	$\Perv_\cS(S^\mu X,\kk)$ with $R_\mu\otimes \kk\mmod$, the functor 
	$(\omega_\mu)^*$ is given by 
	\[M \mapsto eM = eR_{[n]} \otimes_{R_{[n]}} M,\] 
	where the $(eR_{[n]}e \otimes \kk)$-module $eM$ is restricted to a module for $R_{\mu}\otimes \kk$, and $eR_{[n]}$ is considered as an $R_\mu$-$R_{[n]}$ bimodule.
	
	The functor $(\sigma_\mu)_*$ is given by:
	\[M \mapsto \Hom_{R_\mu}(eR_{[n]}, M).\]
\end{theorem}

\begin{proof}
	The algebra $R_\mu$ is isomorphic to the endomorphism algebra of the projective generator $\oplus_{\lambda \le \mu} \cF^\mu_\lambda$ in $\Perv_{\cS}(S^\mu X,\kk)$.  Then for $\cG \in \Perv_\cS(S^nX,\kk)$, adjunction gives
	\[\Hom_{\Perv_\cS(S^\mu X,\kk)}(\oplus_{\lambda \le \mu} \cF^\mu_\lambda, (\omega_\mu)^*\cG) = \Hom_{\Perv_\cS(S^nX, \kk)}(\oplus_{\lambda\le \mu}(\sigma_\mu)_*\cF^\mu_\lambda, \cG).\]
	But $(\sigma_\mu)_*\cF^\mu_\lambda = (\sigma_\mu)_*(\sigma^\mu_\lambda)_*\cB_\lambda = \cF_\lambda$.  So our functor is given by passing from 
	$\Hom(\oplus_{\lambda\in \Part{[n]}} \cF_\lambda,\cG)$ to $\Hom(\oplus_{\lambda\le\mu} \cF_\lambda,\cG)$, which is just multiplication by the idempotent $e$, followed by pullback by the homomorphism
	\[(\sigma_\mu)_*\colon \End_{\Perv_\cS(S^\mu X,\kk)}(\oplus_{\lambda \le \mu} \cF^\mu_\lambda) \to \End_{\Perv_\cS(S^nX, \kk)}(\oplus_{\lambda \le \mu} \cF_\lambda) \cong eR_{[n]}e\otimes \kk.\]
	So we want to show that $(\sigma_\mu)_*$ induces the inclusion of $R_\mu\otimes \kk$ into $eR_{[n]}e \otimes \kk$ described above.  This is easy to check on generators.  The idempotent $e_\lambda$, $\lambda\le \mu$ is clearly sent to itself, since they are projections onto the summand $\cF^\mu_\lambda$ or $\cF_\lambda$, respectively.  
	
	For $\nu\le \lambda\le \mu$, the morphism ${}_\nu 1_{\lambda}\colon \cF^\mu_\lambda \to \cF^\mu_\nu$ is given by applying $(\sigma^\mu_\lambda)_*$ to the unit morphism
	$\cB_\lambda \to (\sigma^\lambda_\nu)_*(\omega^\lambda_\nu)^*\cB_\lambda \cong (\sigma^\lambda_\nu)_*\cB_\nu.$
	Then the isomorphism $(\sigma_\mu)_*(\sigma^\mu_\lambda)_* \cong (\sigma_\lambda)_*$ immediately shows that applying $(\sigma_\mu)_*$ to ${}_\nu 1_{\lambda}$ gives the morphism with the same name in   $\End_{\Perv_\cS(S^nX, \kk)}(\oplus_{\lambda \le \mu} \cF_\lambda)$.
	
	For $z\in \fS_\mu$, the morphism ${}_{z\lambda}(z_*)_{\lambda}\colon \cF^\mu_\lambda \to \cF^\mu_{z\lambda}$ is obtained by applying $(\sigma_\lambda^\mu)_*$ to the morphism $z_*\cB_\lambda \to \cB_{z\lambda}$.  Applying $(\sigma_\mu)_*$ and using $(\sigma_\mu)_*(\sigma^\mu_\lambda)_* \cong (\sigma_\lambda)_*$ immediately shows that it is sent to the morphism  $\cF_{\lambda}\to \cF_{z\lambda}$
	with the same name.  	 
	 
	 This proves the first part of the theorem.  The description of $(\sigma_\mu)_*$ then follows by the uniqueness of adjoint functors.
\end{proof}

\bibliography{./refs.bib}

\def\cprime{$'$} \newcommand{\arxiv}[1]{\href{http://arxiv.org/abs/#1}{\tt
  arXiv:\nolinkurl{#1}}}
\providecommand{\bysame}{\leavevmode\hbox to3em{\hrulefill}\thinspace}
\providecommand{\MR}{\relax\ifhmode\unskip\space\fi MR }
\providecommand{\MRhref}[2]{%
  \href{http://www.ams.org/mathscinet-getitem?mr=#1}{#2}
}
\providecommand{\href}[2]{#2}
\begin{thebibliography}{BLPW16}

\bibitem[Ach21]{AcharBook}
Pramod~N. Achar, \emph{Perverse sheaves and applications to representation
  theory}, Mathematical Surveys and Monographs, vol. 258, American Mathematical
  Society, Providence, RI, 2021.

\bibitem[AHR15]{AHR}
Pramod~N. Achar, Anthony Henderson, and Simon Riche, \emph{Geometric {S}atake,
  {S}pringer correspondence, and small representations {II}}, Represent. Theory
  \textbf{19} (2015), 94--166.

\bibitem[AM15]{AM}
Pramod~N. Achar and Carl Mautner, \emph{Sheaves on nilpotent cones, {F}ourier
  transform, and a geometric {R}ingel duality}, Mosc. Math. J. \textbf{15}
  (2015), no.~3, 407--423, 604. \MR{3427432}

\bibitem[BBD82]{BBD}
A.~A. Beilinson, J.~Bernstein, and P.~Deligne, \emph{Faisceaux pervers},
  Analysis and topology on singular spaces, {I} ({L}uminy, 1981), Ast\'erisque,
  vol. 100, Soc. Math. France, Paris, 1982, pp.~5--171.

\bibitem[BGS96]{BGS}
Alexander Beilinson, Victor Ginzburg, and Wolfgang Soergel, \emph{Koszul
  duality patterns in representation theory}, J. Amer. Math. Soc. \textbf{9}
  (1996), no.~2, 473--527.

\bibitem[BK04]{BezKal}
R.~V. Bezrukavnikov and D.~B. Kaledin, \emph{Mc{K}ay equivalence for symplectic
  resolutions of quotient singularities}, Tr. Mat. Inst. Steklova \textbf{246}
  (2004), no.~Algebr. Geom. Metody, Svyazi i Prilozh., 20--42.

\bibitem[BKR01]{BKR}
Tom Bridgeland, Alastair King, and Miles Reid, \emph{The {M}c{K}ay
  correspondence as an equivalence of derived categories}, J. Amer. Math. Soc.
  \textbf{14} (2001), no.~3, 535--554.

\bibitem[BLPW12]{BLPWtorico}
Tom Braden, Anthony Licata, Nicholas Proudfoot, and Ben Webster,
  \emph{Hypertoric category $\mathcal{O}$}, Adv. Math. \textbf{231} (2012),
  no.~3-4, 1487--1545.

\bibitem[BLPW16]{BLPWgco}
\bysame, \emph{Quantizations of conical symplectic resolutions $\text{II}$:
  category $\mathcal{O}$ and symplectic duality}, Ast\'erisque \textbf{384}
  (2016), 75--179.

\bibitem[BM81]{BoMcP81}
Walter Borho and Robert MacPherson, \emph{Repr\'{e}sentations des groupes de
  {W}eyl et homologie d'intersection pour les vari\'{e}t\'{e}s nilpotentes}, C.
  R. Acad. Sci. Paris S\'{e}r. I Math. \textbf{292} (1981), no.~15, 707--710.

\bibitem[BM83]{BoMcP}
\bysame, \emph{Partial resolutions of nilpotent varieties}, Analysis and
  topology on singular spaces, {II}, {III} ({L}uminy, 1981), Ast\'{e}risque,
  vol. 101, Soc. Math. France, Paris, 1983, pp.~23--74.

\bibitem[BM17]{BMMatroid}
Tom Braden and Carl Mautner, \emph{Matroidal {S}chur algebras}, J. Algebraic
  Combin. \textbf{46} (2017), no.~1, 51--75.

\bibitem[BM19]{BMHyperRingel}
\bysame, \emph{Ringel duality for perverse sheaves on hypertoric varieties},
  Adv. Math. \textbf{344} (2019), 35--98.

\bibitem[Bon06a]{Bon06}
C\'{e}dric Bonnaf\'{e}, \emph{A note on the {G}rothendieck ring of the
  symmetric group}, C. R. Math. Acad. Sci. Paris \textbf{342} (2006), no.~8,
  533--538.

\bibitem[Bon06b]{Bon08}
\bysame, \emph{On the character ring of a finite group}, Alg\`ebre et
  th\'{e}orie des nombres. {A}nn\'{e}es 2003--2006, Publ. Math. Univ.
  Franche-Comt\'{e} Besan\c{c}on Alg\`ebr. Theor. Nr., Lab. Math. Besan\c{c}on,
  Besan\c{c}on, 2006, pp.~5--23.

\bibitem[Bra03]{Br03}
Tom Braden, \emph{Hyperbolic localization of intersection cohomology},
  Transform. Groups \textbf{8} (2003), no.~3, 209--216.

\bibitem[Bri77]{Bri77}
Jo\"{e}l Brian\c{c}on, \emph{Description de
  {$\operatorname{Hilb}^{n}C\{x,y\}$}}, Invent. Math. \textbf{41} (1977),
  no.~1, 45--89.

\bibitem[CG97]{CG}
Neil Chriss and Victor Ginzburg, \emph{Representation theory and complex
  geometry}, Birkh\"auser Boston, Inc., Boston, MA, 1997.

\bibitem[EGNO15]{EGNO}
Pavel Etingof, Shlomo Gelaki, Dmitri Nikshych, and Victor Ostrik, \emph{Tensor
  categories}, Mathematical Surveys and Monographs, vol. 205, American
  Mathematical Society, Providence, RI, 2015.

\bibitem[Gro96]{Groj}
I.~Grojnowski, \emph{Instantons and affine algebras. {I}. {T}he {H}ilbert
  scheme and vertex operators}, Math. Res. Lett. \textbf{3} (1996), no.~2,
  275--291. \MR{1386846}

\bibitem[GS06]{GS}
I.~Gordon and J.~T. Stafford, \emph{Rational {C}herednik algebras and {H}ilbert
  schemes. {II}. {R}epresentations and sheaves}, Duke Math. J. \textbf{132}
  (2006), no.~1, 73--135.

\bibitem[Hai01]{Haim}
Mark Haiman, \emph{Hilbert schemes, polygraphs and the {M}acdonald positivity
  conjecture}, J. Amer. Math. Soc. \textbf{14} (2001), no.~4, 941--1006.

\bibitem[Iar77]{I77}
Anthony~A. Iarrobino, \emph{Punctual {H}ilbert schemes}, Mem. Amer. Math. Soc.
  \textbf{10} (1977), no.~188, viii+112.

\bibitem[Jut14]{Juteau}
D.~Juteau, \emph{Modular {S}pringer correspondence, decomposition matrices and
  basic sets}, arXiv:1410.1471, 2014.

\bibitem[Kal06]{Kal06}
D.~Kaledin, \emph{Symplectic singularities from the {P}oisson point of view},
  J. Reine Angew. Math. \textbf{600} (2006), 135--156.

\bibitem[Knu]{knudsen2018configuration}
Ben Knudsen, \emph{Configuration spaces in algebraic topology}, preprint,
  \url{https://arxiv.org/abs/1803.11165}.

\bibitem[Loj64]{Loj}
S.~Lojasiewicz, \emph{Triangulation of semi-analytic sets}, Ann. Scuola Norm.
  Sup. Pisa Cl. Sci. (3) \textbf{18} (1964), 449--474.

\bibitem[Mau14]{Mautner}
Carl Mautner, \emph{A geometric {S}chur functor}, Selecta Math. (N.S.)
  \textbf{20} (2014), no.~4, 961--977.

\bibitem[MV07]{MV}
I.~Mirkovi{\'c} and K.~Vilonen, \emph{Geometric {L}anglands duality and
  representations of algebraic groups over commutative rings}, Ann. of Math.
  (2) \textbf{166} (2007), no.~1, 95--143.

\bibitem[Nak97]{Nak97}
Hiraku Nakajima, \emph{Heisenberg algebra and {H}ilbert schemes of points on
  projective surfaces}, Ann. of Math. (2) \textbf{145} (1997), no.~2, 379--388.
  \MR{1441880}

\bibitem[Nak99]{NakBook}
\bysame, \emph{Lectures on {H}ilbert schemes of points on surfaces}, University
  Lecture Series, vol.~18, American Mathematical Society, Providence, RI, 1999.
  \MR{1711344}

\bibitem[Sag01]{Sag-book}
Bruce~E. Sagan, \emph{The symmetric group}, second ed., Graduate Texts in
  Mathematics, vol. 203, Springer-Verlag, New York, 2001, Representations,
  combinatorial algorithms, and symmetric functions. \MR{1824028}

\bibitem[Spr82]{Springer}
T.~A. Springer, \emph{Quelques applications de la cohomologie d'intersection},
  Bourbaki {S}eminar, {V}ol. 1981/1982, Ast\'{e}risque, vol.~92, Soc. Math.
  France, Paris, 1982, pp.~249--273.

\bibitem[Spr84]{SprContraction}
\bysame, \emph{A purity result for fixed point varieties in flag manifolds}, J.
  Fac. Sci. Univ. Tokyo Sect. IA Math. \textbf{31} (1984), no.~2, 271--282.

\end{thebibliography}
\bibliographystyle{amsalpha}

\end{document}